\documentclass[a4ppaper,14pt]{amsart}

\usepackage{enumerate}
\usepackage[latin1]{inputenc}
\usepackage{amsmath}
\usepackage{amsthm}
\usepackage{latexsym}
\usepackage{amssymb}
\usepackage{amsmath}
\usepackage{comment}
\usepackage[all]{xy}
\usepackage{calc}
\usepackage{multicol}
\usepackage{slashed}
\usepackage{bm}

\newtheorem{thm}{Theorem}[section]
\newtheorem{prop}{Proposition}[section]
\newtheorem{defi}{Definition}[section]
\newtheorem{lem}{Lemma}[section]
\newtheorem{cor}{Corollary}[section]

\newtheorem{rem}{Remark}[section]
\theoremstyle{notation}
\newtheorem*{notation}{Notation}

\newcommand{\R}{\mathbb{R}}

\numberwithin{equation}{section}
\newcommand{\N}{\mathbb{N}}

\newcommand{\eps}{\epsilon}

\newcommand{\wto}{\rightharpoonup}

\newcommand{\vertiii}[1]{{\left\vert\kern-0.25ex\left\vert\kern-0.25ex\left\vert #1
\right\vert\kern-0.25ex\right\vert\kern-0.25ex\right\vert}}

\usepackage[textwidth=138mm,
textheight=215mm,
left=30mm,
right=30mm,
top=25.4mm,
bottom=25.4mm,
headheight=2.17cm,
headsep=4mm,
footskip=12mm,
heightrounded,]{geometry}

\usepackage[colorlinks=true,urlcolor=blue,
citecolor=black,linkcolor=black,linktocpage,pdfpagelabels,
bookmarksnumbered,bookmarksopen]{hyperref}

\makeatletter
\newcommand{\leqnomode}{\tagsleft@true}
\newcommand{\reqnomode}{\tagsleft@false}
\makeatother

\begin{document}

\reqnomode

\title{Higher topological type semiclassical states for fractional nonlinear elliptic equations}

\author{Shaowei Chen \and Tianxiang Gou ${}^*$}

\address{Shaowei Chen
\newline \indent School of Mathematical Sciences, Huaqiao University,
\newline \indent Quanzhou 362021, People's Republic of China.}
\email{swchen6@163.com}

\address{Tianxiang Gou
\newline \indent School of Mathematics and Statistics, Xi'an Jiaotong University,
\newline \indent Xi'an, Shaanxi 710049, People's Republic of China.}
\email{tianxiang.gou@xjtu.edu.cn}
\thanks{ ${}^*$ Corresponding author. E-mail address: tianxiang.gou@xjtu.edu.cn.}

\begin{abstract} 
In this paper, we are concerned with semiclassical states to the following fractional nonlinear elliptic equation,
\begin{align*}
\eps^{2s}(-\Delta)^s u + V(x) u=\mathcal{N}(|u|)u \quad \mbox{in}
\,\,\, \R^N,
\end{align*}
where $0<s <1$, $\eps>0$ is a small parameter, $N>2s$, $V \in
C^1(\R^N, \R^+)$ and $\mathcal{N}\in C(\R, \R)$. The nonlinearity has Sobolev subcritical,
critical or supercritical growth. The fractional Laplacian
$(-\Delta)^s$ is characterized as $\mathcal{F}((-\Delta)^{s}u)(\xi)=|\xi|^{2s} \mathcal{F}(u)(\xi)$
for $\xi \in \R^N$, where $\mathcal{F}$ denotes the Fourier transform. We construct positive semiclassical states and an infinite sequence of sign-changing semiclassical states with higher energies clustering near the local minimum points of the potential $V$. The solutions are of higher topological type, which are obtained from a minimax characterization of higher dimensional symmetric linking structure via the symmetric mountain pass theorem. They correspond to critical points of the underlying energy functional at energy levels where compactness condition breaks down. The proofs are mainly based on penalization methods, s-harmonic extension theories and blow-up arguments along with local type Pohozaev identities.


\end{abstract}

\maketitle

\tableofcontents

\section{Introduction}\label{intro}

\reqnomode

In this paper, we are interested in semiclassical states to the following nonlinear Schr\"odinger
equation,
\begin{align}\label{fequ}
\eps^{2s}(-\Delta)^s u+V(x) u=\mathcal{N}(|u|)u \quad \mbox{in}\,\,
\mathbb{R}^N,
\end{align}
where $0<s <1$, $\eps>0$ is a small parameter, $N>2s$, $V \in C^1(\R^N, \R^+)$ and $\mathcal{N} \in C(\R, \R)$. The problem under consideration arises in the study of standing waves to the following time-dependent nonlinear Schr\"odinger equation,
\begin{align} \label{tequ}
\textnormal{i} \hbar \partial_t \psi=\hbar^{2s} (-\Delta)^s \psi + M(x) \psi-\mathcal{N}(|\psi|)\psi  \quad \mbox{in} \,\,\, \R \times \R^N,
\end{align}
where $\textnormal{i}$ is the imaginary unit, $\hbar$ is the Planck constant, $\psi$ represents wave function of particle states and $M(x)$ represents external field. Here a standing wave to \eqref{tequ} is a solution of the form
$$
\psi(x, t)=e^{\frac{\textnormal{i} \mu t}{\hbar}}u(x).
$$
It is clear that $\psi$ is a solution to \eqref{tequ} if and only if $u$ is a solution to \eqref{fequ} with $\eps=\hbar$ and $V(x)=M(x)+\mu$. The equation \eqref{tequ} was introduced by Laskin in \cite{La0} as an extension of the classical nonlinear Schr\"odinger equation for $s = 1$. It is of particular interest in the fractional quantum mechanics and used to describe many physical problems, such as phase transitions and conservation laws, see for example \cite{La1, La2} and references therein. When $\eps>0$ sufficiently small, solutions to \eqref{fequ} are often referred to as semiclassical states. Concentration phenomenon of semiclassical states as $\eps \to 0^+$ reflects the transition from quantum mechanics to classical mechanics and it gives rise to significant physical insights.

The aim of the present paper is to construct positive semiclassical states and an unbounded sequence of sign-changing semiclassical states to \eqref{fequ} concentrating around the trapping region of the potential $V$ in the Sobolev subcritical, critical and supercritical cases. The solutions we construct are of higher topological type and possess higher energies, which correspond to critical points of the underlying energy functional at energy levels where compactness condition breaks down. Indeed, energies of the solutions can tend to infinity. For the potential $V$, we formulate the following assumptions.
\begin{enumerate}
\item [($V_1$)] $V \in C^1(\R^N, \R^+)$ and there exist $a>0$ and $ b>0$
such that $a \leq  V(x) \leq b$
for any $x \in \R^N$.
\item [($V_2$)] There exists a bounded domain
$\Lambda \subset \R^N$ with smooth boundary $\partial \Lambda$
such that
$$
\nabla V(x) \cdot {\bf{n}}(x) >0 \quad \mbox{for any} \, \, x \in \partial \Lambda,
$$
where ${\bf{n}}(x)$ denotes the unit outward normal vector to $\partial \Lambda$ at $x$.
\end{enumerate}

\begin{rem}
Such assumptions on the potential $V$ seem initially addressed by the first author and Wang in \cite{ChWa} to consider semiclassical states to nonlinear Schr\"odinger equations, see also \cite{CLW}.
\end{rem}

\begin{rem}
Note that $(V_2)$ is fulfilled if $V$ has an isolated local minimum set, i.e. $V$ has a local trapping potential well. This shows that $(V_2)$ is more general than usual global or local conditions imposed on potentials to discuss semiclassical states, which actually causes the study of the concentration of semiclassical states to become difficult.
\end{rem}

Let us now fix some notations for further clarifications. Under the assumption $(V_2)$, we define the set of critical points of $V$ by
\begin{align} \label{defv}
\mathcal{V}:=\{x \in \R^N\ |\  \nabla V(x)=0\}.
\end{align}
Clearly, $\mathcal{V}$ is a nonempty compact subset of $\Lambda$. Without loss of generality, hereafter we shall assume that $0 \in \mathcal{V}$. 
For any $\Omega \subset \R^N$, $\delta>0$ and $\eps>0$, we define
$$
\Omega_{\eps}:=\left\{x \in \R^N\ |\  \eps x \in \Omega\right\}
$$
and
$$
\Omega^{\delta}:=\left\{x \in \R^N \ |\ \mbox{dist}(x, \Omega):=\inf_{y \in \Omega}|x-y| <\delta\right\}.
$$

Firstly, we consider semiclassical states to \eqref{fequ} with
the nonlinearity having Sobolev subcritical growth, i.e.
$\mathcal{N}(|u|)u=|u|^{p-2}u$ for $2<p<2^*_s:=\frac{2N}{N-2s}$ and $N>2s$.
In this situation, our main result reads as follows.

\begin{thm}\label{wejgh77rtff111}
Assume  $N \geq 2$,  $\mathcal{N}(|u|)u=|u|^{p-2}u$ with
$2<p<2^*_s,$ $(V_1)$ and $(V_2)$ hold. Then for any positive
integer $k$, there exists $\epsilon_k>0$ such that, for any
$0<\epsilon<\epsilon_k$, \eqref{fequ}
 has at least $k$ pairs of solutions $\pm
u_{j, \eps}$, $j=1,2\cdots, k,$ satisfying that, for any
$\delta>0$, there exists $C=C(\delta,k)>0$ such that
\begin{align*}
|u_{j,\epsilon}(x)|\leq \frac{C}{1+\left|\epsilon^{-1}\mbox{\textnormal{dist}}(x,\mathcal{V}^\delta)\right|^{N+2s}}, \quad x \in 
\R^N.
\end{align*}
Moreover, $u_{j,\epsilon}$ is positive solution for $j=1$ and $u_{j,\epsilon}$ is sign-changing solution for any $j \geq 2$.
\end{thm}

\begin{rem}
The assumption $N \geq 2$ is only applied to derive the concentration of semiclassical states to \eqref{fequ} with the help of a local type Pohozaev identity, see Lemma \ref{cnvbghyyt7888sdd}. It would be interesting to know whether Theorem \ref{wejgh77rtff111} still holds for the case $N=1$.
\end{rem}

Secondly, we study semiclassical states to \eqref{fequ} with
the nonlinearity having Sobolev critical growth, i.e.
$\mathcal{N}(|u|)u=\mu |u|^{p-2} u +|u|^{2^*_s-2} u$ for
$2<p<2^*_s$ and some $\mu>0$. In this situation, our main result
reads as follows.

\begin{thm}\label{wejgh77rtff11} 
Assume $N \geq 2$, $\mathcal{N}(|u|)u=\mu |u|^{p-2} u
+|u|^{2^*_s-2} u$ with $\mu>0$ and $\max\{2,
(N+2s)/(N-2s)\}<p<2^*_s$, $(V_1)$ and $(V_2)$ hold. Then for any
positive integer $k$, there exists $\epsilon_k>0$ such that, for any
$0<\epsilon<\epsilon_k$, \eqref{fequ} has at least $k$ pairs of
solutions $\pm u_{j, \eps}$, $j=1,2\cdots, k,$ satisfying that,
for any $\delta>0$, there exists $C=C(\delta,k)>0$ such
that
$$
|u_{j, \eps}(x)|\leq \frac{C}{1+\left|\epsilon^{-1}\mbox{\textnormal{dist}}(x,\mathcal{V}^\delta)\right|^{N+2s}}, \quad x \in \R^N.
$$
Moreover, $u_{j,\epsilon}$ is positive solution for $j=1$ and $u_{j,\epsilon}$ is sign-changing solution for any $j \geq 2$.
\end{thm}

\begin{rem}
The restriction on $p$ is technical, which comes from Lemma \ref{lames} employed to exclude blow-up of semiclassical states. It is unknown whether Theorem \ref{wejgh77rtff11} remains valid for the case $2<p \leq \max\{2,
(N+2s)/(N-2s)\}$.
\end{rem}

For the Sobolev critical case, it was found in \cite{HZ} that the Palais-Smale condition holds at the energy level $c<\frac s N {S(s, N)^{N/2s}}$, where $S(s, N)>0$ is the Sobolev embedding constant. Actually, one can further verify that $\frac s N {S(s, N)^{N/2s}}$ is the least energy level at which the compactness condition fails. 
Moreover, a global compactness result, see \cite[Theorem 3.1]{CF}, which describes the obstacles of the compactness for problems having Sobolev critical growth, reveals that it is unlikely to prove the Palais-Smale condition at energy level above certain energy level. This indicates that the study of the existence and concentration of multiple semiclassical states to \eqref{fequ} with higher energies is considerably complex in the Sobolev critical case. To overcome the difficulty of lack of compactness, we require to propose new ingredients in the framework of fractional Laplacian and carry out delicate and involved analysis to deduce sharp decay estimates of solutions in certain domains. Roughly speaking, in this case, by properly modifying the critical nonlinearity, the problem under consideration can be regarded as subcritical one. After that, by carefully establishing uniform $L^{\infty}$ estimate of semiclassical states, we can derive the desired conclusion.

Finally, we investigate semiclassical states to \eqref{fequ} with
the nonlinearity having Sobolev supercritical growth, i.e.
$\mathcal{N}(|u|)u= |u|^{p-2} u + \lambda |u|^{r-2}u$ for
$2<p<2^*_s<r$ and $\lambda>0$. In this situation, our main result
reads as follows.

\begin{thm}\label{wejgh77rtff1}
Assume $N \geq 2$,  $\mathcal{N}(|u|)u= |u|^{p-2} u +
\lambda |u|^{r-2}u$ with
 $2<p<2^*_s<r,$ $(V_1)$ and $(V_2)$ hold.
Then for any positive integer $k$, there exists $\lambda_k>0$ such
that, for any $0<\lambda<\lambda_k$, there exists $\eps_{k,
\lambda}>0$ such that, for any $0<\epsilon<\epsilon_{k, \lambda}$,
\eqref{fequ} has at least $k$ pairs of solutions $\pm u_{j, \eps,
\lambda}$, $j=1,2\cdots, k,$ satisfying that, for any $\delta>0$,
there exists $C=C(\delta,k, \lambda)>0$ such that
$$
|u_{j,\epsilon, \lambda}(x)|\leq \frac{C}{1+\left|\epsilon^{-1}\mbox{\textnormal{dist}}(x,\mathcal{V}^\delta)\right|^{N+2s}},\quad x\in \R^N.
$$
Moreover, $u_{j,\epsilon, \lambda}$ is positive solution for $j=1$ and $u_{j,\epsilon, \lambda}$ is sign-changing solution for any $j \geq 2$.
\end{thm}

Since the problem is treated in $H^s(\R^N)$, which is continuously embedded into $L^p(\R^N)$ for any $2 \leq p \leq 2^*_s$, then we first require to truncate the supercritical nonlinearity. By deducing uniform $L^{\infty}$ estimate of semiclassical states, we then obtain the conclusion of Theorem \ref{wejgh77rtff1}.

\begin{rem}
The strategies we present here to study semiclassical states to \eqref{fequ} are flexible, which can be adapted to consider semiclassical states to other types of fractional nonlinear Schr\"odinger equations.
\end{rem}

For the case $s=1$, the existence and concentration of semiclassical states to \eqref{fequ} have been extensively explored during recent decades and there already exists a great deal of literature, see for example \cite{BJ, BW1, BW2, CLW, ChWa, DF1, DF, FS, FW,Oh1, Oh2, Ra, Wang} and references therein. However, for the case $0<s<1$, as far as we know there exist relatively few papers treating the existence and concentration of semiclassical states to \eqref{fequ}. Let us now briefly recall some related results in this direction. For the Sobolev subcritical case, Figueiredo and Siciliano in \cite{FS} considered the existence and concentration of semiclassical states to \eqref{fequ} under the following global assumption on $V$,
\begin{align} \label{gv}
0<\inf_{x \in \R^N} V(x) < \liminf_{|x| \to \infty} V(x).
\end{align}
Later, the authors in \cite{OO, A1, CG} further investigated the existence and concentration of semiclassical states to \eqref{fequ} under the following local assumption on $V$,
\begin{align} \label{lv}
0<\inf_{x \in \Lambda} V(x) < \inf_{x \in \partial \Lambda} V(x),
\end{align}
where $\Lambda \subset \R^N$ is a bounded domain. Let us also refer the readers to \cite{Chen, CZ, DdW,MMV}, where the Lyapunov-Schmidt reduction method was adapted to discuss the existence and concentration of semiclassical states to \eqref{fequ}. For the Sobolev critical case, Shang et al. in \cite{SZY} obtained the existence and concentration of semiclassical states to \eqref{fequ} under the assumption \eqref{gv}. Subsequently, the authors in \cite{A2, AF, He, HZ} studied the existence and concentration of semiclassical states to \eqref{fequ} under the assumption \eqref{lv}.  Let us point out that these papers only concerned positive semiclassical states having energies below certain value where compactness condition holds true. In addition, the number of the solutions are finite. Inspired by the previous works, it would be interesting to question whether there exist (even infinitely many) sign-changing semiclassical states to \eqref{fequ} in the Sobolev subcritical, critical and supercritical cases.
This is the motivation of our present survey and Theorems \ref{wejgh77rtff111}-\ref{wejgh77rtff1} give an affirmative answer, which also complement and extend the results before.

Let us now highlight a few features of our study. Firstly, we deduce the existence and concentration of semiclassical states to \eqref{fequ} under the assumptions $(V_1)$ and $(V_2)$. In our situation, there do not exist limit problems to use, which causes the discussion of the concentration of semiclassical states more complex. 
This is in striking contrast with the preceding works under the assumptions \eqref{gv} and \eqref{lv}, where limit problems do exist and play an essential role in the study. Secondly, we derive the existence and concentration of positive semiclassical states and infinitely many sign-changing semiclassical states to \eqref{fequ} without applying any non-degeneracy condition. Indeed, as far as we know, non-degeneracy of sign-changing solutions to fractional nonlinear elliptic equations is unknown up to now. Lastly, the solutions we construct possess higher energies, which are of higher topological type.
They correspond to critical points of the underlying energy functionals at energy levels where compactness conditions fail. To our knowledge, such results are new and have not been established previously for fractional nonlinear elliptic equations. And our Theorem \ref{wejgh77rtff1} is new even for the case $s=1$. 

It is worth mentioning that our Theorems \ref{wejgh77rtff111} and \ref{wejgh77rtff11} can be regarded as counterparts of the results obtained in \cite{CLW,ChWa} for the case $s=1$. However, they are not straightforward and standard generalizations, the proofs of which are highly nontrivial. Due to the nonlocal nature of fractional Laplacian, it is quite hard to deduce the desired decay estimates of semiclassical states to \eqref{fequ} by directly using \eqref{fequ}. Therefore, in the present paper, we take advantage of the s-harmonic extension arguments from Caffarelli and Silvestre in \cite{CS} to transform our problem to a local one defined in the upper half space with nonlinear Neumann boundary condition. This transformation actually creates additional difficulties coming from the fact that the extended problem is not homogeneous and solutions of which only enjoy polynomial decay.
As a consequence, techniques employed to investigate semiclassical states to the local problems in \cite{CLW,ChWa} are no longer suitable to our problem. For this reason, we need to develop new elements in order to capture sharp interior and boundary estimates of solutions to the extended problem in weighted Sobolev spaces. 

Let us now outline the strategies to establish Theorems \ref{wejgh77rtff111}-\ref{wejgh77rtff1}. We first make use of the s-harmonic extension arguments originating from \cite{CS} to transform our problem to a local one. Notice that the solutions we construct admit higher energies, which may go beyond the least energy threshold for guaranteeing compactness condition. Thus we next need to introduce modified problems. In the Sobolev subcritical case, we adapt truncation technique from Wang and Zhang in \cite{WZ} to modify the nonlinearity. In the Sobolev critical case, our study becomes more complicated. We need to combine the approaches from Chen et al in \cite{CLW} and Zhao et al in \cite{ZLL} to modify the nonlinearity, because of the presence of the critical nonlinearity. When it comes to the Sobolev supercritical case, we need to combine the approach from Chabrowski and Yang in \cite{CY} to modify the nonlinearity, because the supercritical nonlinearity comes into being. By doing this, we are able to demonstrate that the underlying energy functionals satisfy the Palais-Smale condition, from which we immediately derive the existence of an unbounded sequence of semiclassical states to the modified problems. The last and crucial step is to demonstrate that the solutions to the modified problems are indeed ones to the original problems. In the Sobolev subcritical case, this is done by conducting a finer asymptotic analysis to the solutions and utilizing a local type Pohozaev identity. In the Sobolev critical case, things become tough and more delicate and involved analysis is required. Relying on a profile decomposition result, we first derive sharp decay estimates of the solutions in certain domains which do not contain any blow-up point, but are close to some blow-up points. We next apply a local type Pohozaev identity to rule out the possibility of blow-up of the solutions, which in turns leads to the uniform boundedness of the solutions. At this point, we can follow the way to deal with the Sobolev subcritical case to accomplish the proof. In the Sobolev supercritical case, the key argument is to deduce uniform boundedness of the solutions through Moser iterative technique, by which the problem under consideration is reduced to the Sobolev subcritical case and the proof is completed.

The remainder of the paper is laid out as follows. In Section \ref{pre}, we present some preliminary results used to porve our main results. In Section \ref{subcritical}, we treat the Sobolev subcritical case and show the proof of Theorem \ref{wejgh77rtff111}. In Section \ref{critical}, we handle the Sobolev critical case and give the proof of Theorem \ref{wejgh77rtff11}. Section \ref{supercritical} is devoted to the study of the Sobolev supercritical case and the proof of Theorem \ref{wejgh77rtff1}. In Appendix, we establish a local boundedness estimate of solutions to a type of degenerate nonlinear elliptic equations, which may be of independent interest.

\begin{notation}
Throughout the paper, for any $1 \leq q < \infty$, we denote by $L^{q}(\R^N)$ the Lebesgue space equipped with the norm
$$
\|u\|_{L^q(\R^N)}:=\left(\int_{\R^N} |u|^q \, dx \right)^{\frac 1q} \quad \mbox{for any} \,\, u \in L^q(\R^N).
$$
We use notations $\to$ and $\wto$ to denote strong convergence and weak convergence of sequences in associated spaces, respectively. We use letter $C$ to denote generic positive constants, whose values may change from line to line. In addition, we use $o_n(1)$ to denote quantities which tend to zero as $n \to \infty$.
\end{notation}

\section{Preliminaries} \label{pre}

In this section, we shall present some preliminary results used to establish our main results. First of all, by making a change of variable $x \to \eps x$, we may write \eqref{fequ} as
\begin{align} \label{equ1}
(-\Delta)^s u + V_{\eps}(x) u=\mathcal{N}(|u|) u\quad \mbox{in}
\,\, \R^N,
\end{align}
where $V_{\eps}(x):=V(\eps x)$ for $x \in \R^N$. From now on, we shall consider semiclassical states to \eqref{equ1}. To treat nonlocal problem \eqref{equ1}, we need to introduce the $s$-harmonic extension arguments developed in \cite{CS}, which allow us to transform nonlocal problems into degenerate elliptic problems in the upper half space $\R^{N+1}_+$ with nonlinear Neumann boundary condition. For $0<s<1$, the fractional Sobolev space $H^s(\R^N)$ is defined by
$$
H^s(\R^N):=\left\{u\in L^2(\R^N) \, \big | \, \int_{\R^N}\left(1+|\xi|^{2s}\right) |\mathcal{F}(u)|^2 \, d\xi<\infty\right\}
$$
equipped with the norm
$$
\|u\|^2_{H^s(\R^N)}=\int_{\R^N}\left(1+|\xi|^{2s}\right) |\mathcal{F}(u)|^2 \, d\xi,
$$
where $\mathcal{F}(u)$ stands for the Fourier transform of $u$.
For $u \in H^s(\R^N)$, we denote by $w \in X^s(\R^{N+1}_+)$ its $s$-harmonic extension to the upper half space $\R^{N+1}_+$, where $w$ is the solution to the equation
\begin{align} \label{extension}
\left\{
\begin{aligned}
-\mbox{div}(y^{1-2s} \nabla w)&=0 \quad \,\,\, \mbox{in} \,\, \R^{N+1}_+,\\
w&=u \quad  \,\,\mbox{on} \,\, \R^N \times \{0\}.
\end{aligned}
\right.
\end{align}
Here the space $X^s(\R^{N+1}_+)$ is defined by the completion of $C^{\infty}_0(\overline{\R^{N+1}_+})$ under the norm
$$
\|w\|_{s}^2:=k_s \int_{\R^{N+1}_+} y^{1-2s} |\nabla w|^2 \, dxdy
$$
and the constant $k_s$ is given by
$$
k_s=2^{1-2s} \frac{\Gamma(1-s)} {\Gamma(s)}.
$$
Indeed, the solution $w \in X^s(\R^{N+1}_+)$ to \eqref{extension} has an explicit expression given by
 \begin{align} \label{convolution}
w(x, y)=P^s_y  * u(x) =\int_{\R^N} P_y^s(x-\zeta, y) u(\zeta) \, d\zeta,
\end{align}
where
\begin{align} \label{defpys}
P_y^s(x):=c_{N, s} \frac{y^{2s}}{\left(|x|^2 + y^2\right)^{\frac{N+2s}{2}}},\quad \int_{\R^N} P^s_y(x) \, dx =1.
\end{align}

\begin{lem} \cite[Theorem 2.1]{BCdS}\label{embedding1}
There exists a constant $S=S(s, N)>0$ such that, for any $w \in X^s(\R^{N+1}_+)$, there holds that
$$
\left(\int_{\R^N}  |w(x, 0)|^{2^*_s}\, dx\right)^{\frac{1}{2^*_s}} \leq S \left( \int_{\R^{N+1}_+} y^{1-2s} |\nabla w|^2 \, dxdy\right)^{\frac 12}.
$$
\end{lem}

With the s-harmonic extension arguments, we are now able to transform \eqref{equ1} into the following equation,
\begin{align}\label{equ11}
\left\{
\begin{aligned}
-\mbox{div}(y^{1-2s} \nabla w)&=0 \hspace{4cm}\,\,\mbox{in} \,\, \R^{N+1}_+,\\
-k_s \frac{\partial w}{\partial {\nu}}&=-V_{\eps}(x) w +
\mathcal{N}(|w|) w \, \qquad  \mbox{on} \,\, \R^N \times \{0\},
\end{aligned}
\right.
\end{align}
where
$$
\frac{\partial w}{\partial {\nu}}:= \lim_{y \to 0^+} y^{1-2s}
\frac{\partial w}{\partial y}(x, y) =- \frac{1}{k_s} (- \Delta )^s
u(x).
$$
We shall investigate semiclassical states to \eqref{equ11} in $X^{1, s}(\R^{N+1}_+)$, where $X^{1, s}(\R^{N+1}_+)$ is defined by the completion of $C^{\infty}_0(\overline{\R^{N+1}_+})$ under the norm
$$
\|w\|^2_{1, s}:=k_s\int_{\R^{N+1}_+} y^{1-2s} |\nabla w|^2 \, dxdy + \int_{\R^N} |w(x, 0)|^2 \, dx.
$$

\begin{lem}  \label{embedding}
The embedding $X^{1, s}(\R^{N+1}_+) \hookrightarrow L^q(\R^N)$ is continuous for any $2 \leq q \leq 2^*_s$. Moreover, the embedding $X^{1, s}(\R^{N+1}_+) \hookrightarrow L^q(\R^N)$ is locally compact for any $2 \leq q <2^*_s$.
\end{lem}


To establish the existence of an infinite sequence of semiclassical states to \eqref{equ11}, we shall take advantage of \cite[Theorem 3.2]{ChWa}. Before stating this abstract theorem, let us introduce some notations and definitions. Let $X$ be a Banach space. For a set $P \subset X$, we define
$-P:=\{-u\ |\  u \in P\}$. We denote by $\gamma(A)$ the genus of a
closed symmetric subset $A \subset X$. For a functional $J \in
C^1(X, \R)$ and a constant $c \in \R$, we define
$$
J^c:=\{u \in X \ |\  J(u) \leq c\}, \quad K_c:=\{u \in X\ |\
J(u)=c, J'(u)=0\}.
$$

\begin{defi}
Let $J \in C^1(X, \R)$ be an even functional, $P \subset X$ be a nonempty open set and $W:=P \cup (-P)$. We say that $P$ is an admissible invariant set
with respect to $J$ at level $c \in \R$, if there is a constant $\tau_0>0$ and a symmetric open neighborhood
 $\mathcal{O}$ of $W$ with $\gamma (\overline{\mathcal{O}}) < \infty$ such that, for $0<\tau<\tau_0$, there exists $\eta \in C(X, X)$ satisfying
\begin{enumerate}
\item[$(\textnormal{i})$] $\eta(\partial P) \subset P, \,\,\,
\eta{(\partial (-P))} \subset -P, \,\,\,\eta{(P)} \subset P,
\,\,\,\eta{(-P)} \subset -P$; \item[$(\textnormal{ii})$]
$\eta(-u)=-\eta(u)$ for any $u \in X$; \item[$(\textnormal{iii})$]
$\eta_{\mid_{J^{c-2\tau}}} =id$; \item[$(\textnormal{iv})$]
$\eta(J^{c + 2 \tau} \backslash (\mathcal{O} \cup W)) \subset
J^{c- \tau}$.
\end{enumerate}
\end{defi}

\begin{thm} \cite[Theorem 3.2]{ChWa}\label{exist}
Let $J \in C^1(X, \R)$ be an even functional, $P \subset X$ be a
nonempty open set, $M:=P \cap (-P)$, $W:=P \cup (-P)$ and
$\Sigma:=\partial P \cap \partial (-P)$. Assume $P$ is an
admissible invariant set with respect to $J$ at level $c $ for $ c
\geq c^*:= \inf_{u \in \Sigma} J(u)$. Assume for any $n \in
\mathbb{N}$, there exists a continuous map $\varphi_n : B_n:=\{x
\in \R^n\ |\  |x| \leq 1\} \to X$ satisfying
\begin{enumerate}
\item[$(\textnormal{i})$] $\varphi_n(0) \in M$, $\varphi_n(-x)=-\varphi_n(x)$ for any $x \in B_n$;
\item[$(\textnormal{ii})$] $\varphi_n(\partial B_n) \cap M=\emptyset$;
\item[$(\textnormal{iii})$] $\max\{J(0), \sup_{u\in \varphi_n(\partial B_n)}J(u)\} < c^*$.
\end{enumerate}
For $j \in \mathbb{N}$, define
$$
c_j:=\left\{
\begin{array}
[c]{ll} \displaystyle\inf_{B\in \Lambda_j}\sup_{u\in B\setminus
W} J(u),& \mbox{if} \ j\geq 2,\\
\displaystyle\inf_{B\in \Lambda_1}\sup_{u\in
B}J(u), & \mbox{if} \ j= 1,
\end{array}
\right.
$$
where
\begin{align*} 
\begin{split}
&\Lambda_j:=\left\{B \subset X \ |\ B=\varphi (B_n \backslash Y) \,\, \mbox{for some} \,\, \varphi \in G_n, n \geq j \,\, \mbox{and} \right.\\
& \hspace{2.5cm} \left. \,\,\, Y \subset B_n  \,\, \mbox{is a symmetric open set with }\,\,  \gamma(\overline{Y}) \leq n-j \right \}
\end{split}
\end{align*}
and
$$
G_n:=\{\varphi \in C(B_n, X) \ |\ \varphi(-x)=-\varphi(x) \, \,
\mbox{for any} \,\, x \in B_n, \varphi_{\mid_{\partial
B_n}}={\varphi_n}_{|_{\partial B_n}}\}.
$$
Then for any $j \geq 2$,
$$
K_{c_j} \backslash W \neq \emptyset.
$$
Moreover, if $j \geq 2$ and $ c:=c_j= \cdots =c_{j+m} \geq c^* $, then
$$
\gamma(K_c\backslash W) \geq m+1.
$$
\end{thm}

\section{Sobolev subcritical growth} \label{subcritical}

In this section, we shall consider semiclassical states to
\eqref{equ11} in the Sobolev subcritical case, i.e.
$\mathcal{N}(u)u=|u|^{p-2} u$ for $2<p<2^*_s$. To begin with, let us introduce some notations. Let
$\mu:=2^{-1}(a/4)^{1/(p-2)}$, $$
\kappa<-\left(1+(p-2)\mu^{p-4}+(p-2)|p-3|\mu^{p-4} \right)$$ and
$\phi \in C(\R^+, \R)$ be such that
\begin{align*}
{\phi}(t) =\left\{
\begin{aligned}
&(p-2)t^{p-3}, \, \, \, \,\, \hspace{2.5cm} 0< t \leq \mu,\\
&\kappa(t-\mu)+(p-2)\mu^{p-3}, \,\,\,\,\,\, \qquad \mu< t
\leq\mu-\frac{p-2}{\kappa}\mu^{p-3},\\
&0,  \hspace{5cm} t\geq \mu-\frac{p-2}{\kappa}\mu^{p-3}.
\end{aligned}
\right.
\end{align*}
Define
\begin{align*}
 \quad g(t):=\int_{0}^t \phi(s) \, ds, \quad  G(t):=\int_{0}^t g(s)s \, ds \quad \mbox{for} \,\, t \geq 0.
\end{align*}
In view of the definition of $g$, we see that $g \in C^1(\R^+, \R)$ and
$$g(t)=\left\{
\begin{aligned}
&t^{p-2},  \,\quad t \leq (a/4)^{1/(p-2)},\\
&c_*, \qquad t > 2(a/4)^{1/(p-2)},
\end{aligned}
\right.$$ where
$$c_*:=\int^{\mu-\frac{p-2}{\kappa}\mu^{p-3}}_0\phi(t)\, dt\leq \int^{2\mu}_0\phi(t)\,dt\leq\int^{2\mu}_0(p-2)t^{p-3}\,dt= \frac a 4.$$
In addition, it holds that, if $-\kappa$ is large enough, then
\begin{align} \label{arg}
\frac 12 g(|t|)t^2 -G(|t|) \geq 0 \quad \mbox{for any}\,\, t \in \R.
\end{align}
Hereafter, we shall choose $\kappa<-\left(1+(p-2)\mu^{p-4}+(p-2)|p-3|\mu^{p-4} \right)$ such
that \eqref{arg} holds true. According to the assumption $(V_2)$, we know that there exists $\delta_0 >0$
such that, for any $y \in \Lambda^{\delta_0}$, if $B_{\delta_0}(y) \backslash \Lambda \neq \emptyset$, then
\begin{align} \label{vloc}
\inf_{x \in B_{\delta_0}(y) \backslash \Lambda} \nabla V(x)\cdot
\nabla \mbox{dist}(x, \Lambda)>0,
\end{align}
where $B_r(z):=\{x \in \R^N\ |\  |x-z|<r\}$ with $r>0$ and $z\in\R^N.$ Let $\eta\in C^{\infty}(\R^+, [0, 1])$ be
a cut-off function such that $\eta(t)=0$ for $t \leq 0$, $\eta(t)
>0$ for $t>0$, $\eta(t) =1$ for $t \geq \delta_0$ and $\eta'(t) \geq 0$
for $t \geq 0$, where the constant $\delta_0>0$ is determined by \eqref{vloc}. Define $\chi(x):=\eta(\mbox{dist}(x, \,
\Lambda))$ and
\begin{align*}
f_{\eps}(x, t):=\left(1- \chi(\eps x)\right) t^{p-2}+\chi(\eps x){g}(t),\quad F_{\eps}(x, t):=\int_{0}^t f_{\eps}(x, s)s \, ds \quad \mbox{for}\,\, x \in \R^N, t \in \R.
\end{align*}
Let us now introduce a modified equation as
\begin{align}\label{equ21}
\left\{
\begin{aligned}
-\mbox{div}(y^{1-2s} \nabla w)&=0 \hspace{4cm}\,\,\mbox{in} \,\, \R^{N+1}_+,\\
-k_s \frac{\partial w}{\partial {\nu}}&=-V_{\eps}(x) w + f_{\eps}(x, |w|) w \, \,\quad  \mbox{on} \,\, \R^N \times \{0\}.
\end{aligned}
\right.
\end{align}
The energy functional associated to \eqref{equ21} is defined by
$$
\Phi_{\eps}(w):=\frac {k_s}{2} \int_{\R^{N+1}_+} y^{1-2s} |\nabla w|^2 \, dxdy + \frac 12 \int_{\R^N} V_{\eps}(x)|w(x,0)|^2 \, dx - \int_{\R^N} F_{\eps}(x, |w(x, 0)|) \, dx.
$$
It is standard to check that $\Phi_{\eps} \in C^1(X^{1, s}(\R^{N+1}_+), \R)$ and
\begin{align*}
\Phi_{\eps}'(w) \psi&=k_s\int_{\R^{N+1}_+} y^{1-2s} \nabla w \cdot \nabla \psi \, dxdy + \int_{\R^N} V_{\eps}(x) w(x, 0) \psi(x, 0) \, dx \\
& \quad - \int_{\R^N} f_{\eps}(x, |w(x, 0)|) w(x, 0) \psi(x, 0) \, dx
\end{align*}
for any $\psi \in X^{1, s}(\R^{N+1}_+)$. Consequently, any critical point to $\Phi_{\eps}$ corresponds to a solution to \eqref{equ21}.

\subsection{Existence of semiclassical states}

In the following, we shall adapt Theorem \ref{exist} to discuss the existence of semiclassical states to \eqref{equ21}. To do this,  let us first demonstrate that the energy functional $\Phi_{\eps}$ satisfies the Palais-Smale sequence in $X^{1, s}(\R^{N+1}_+)$.

\begin{lem} \label{ps1}
For any $\eps>0$, the energy functional $\Phi_{\eps}$ satisfies the Palais-Smale condition in $X^{1, s}(\R^{N+1}_+)$.
\end{lem}
\begin{proof}
Suppose $\{w_n\} \subset X^{1, s}(\R^{N+1}_+)$ is a Palais-Smale sequence to $\Phi_{\eps}$, i.e.
$$
\left |\Phi_{\eps}(w_n) \right| \leq C, \quad \Phi_{\eps}'(w_n)=o_n(1).
$$
We shall deduce that $\{w_n\}$ has a convergent subsequence in $X^{1, s}(\R^{N+1}_+)$. To do this, we first show that $\{w_n\}$ is bounded in $X^{1, s}(\R^{N+1}_+)$. In light of \eqref{arg}, we get that
\begin{align} \label{bdd}
\begin{split}
C(1 +\|w\|_{1, s}) &\geq \Phi_{\eps} (w_n) -\frac 12 \Phi_{\eps}'(w_n)w_n \\
& =\int_{\R^N} \frac 12 f_{\eps}(x, |w_n(x, 0|) |w_n(x, 0)|^2 -F_{\eps}(x, |w_n(x, 0)|) \, dx \\
& \geq \frac{p-2}{2p} \int_{\R^N} \left(1- \chi(\eps x)\right) |w_n(x, 0)|^p \,dx.
\end{split}
\end{align}
Since $V(x) \geq a$ for any $x \in \R^N$ and $0 \leq g(|t|) \leq \frac a 4$ for any $t \in \R$, then
\begin{align*}
C \|w_n\|_{1, s} &\geq \Phi_{\eps}'(w_n) w_n \\
& = k_s\int_{\R^{N+1}_+} y^{1-2s} |\nabla w_n|^2 \, dxdy + \int_{\R^N} V_{\eps}(x)|w_n(x, 0)|^2 \, dx \\
& \quad -\int_{\R^N} f_{\eps}(x, |w_n(x, 0)|) |w_n(x, 0)|^2 \, dx \\
& \geq \min\left\{1, \frac{3a}{4}\right\} \|w_n\|_{1, s}^2 - \int_{\R^N}  \left(1- \chi(\eps x)\right) |w_n(x, 0)|^p \, dx.
\end{align*}
This together with \eqref{bdd} yields that $\{w_n\}$ is bounded in $X^{1, s}(\R^{N+1}_+)$. From Lemma \ref{embedding}, we then know that there is $w \in X^{1 ,s}(\R^{N+1}_+)$ such that $w_n \wto w$ in $X^{1, s}(\R^{N+1}_+)$ and $w_n(\cdot, 0) \to w(\cdot, 0)$ in $L_{\textnormal{loc}}^q(\R^N)$ for any $2 \leq q <2^*_s$ as $n \to \infty$. Furthermore, we have that $\Phi_{\eps}'(w)=0$.

We next prove that, up to a subsequence, $w_n \to w$ in $X^{1, s}(\R^{N+1}_+)$ as $n \to \infty$. Let us define $z_n:=w_n -w$, then $\Phi_{\eps}'(w_n) z_n=o_n(1)$ and $ \Phi_{\eps} '(w) z_n=0$. This indicates that
\begin{align} \label{conv1}
\begin{split}
o_n(1)&=\Phi_{\eps}'(w_n) z_n -\Phi_{\eps}'(w) z_n \\
&=k_s\int_{\R^{N+1}_+} y^{1-2s} |\nabla z_n|^2 \, dxdy + \int_{\R^N} V_{\eps}(x)|z_n(x, 0)|^2 \, dx \\
& \quad + \int_{\R^N}\left( f_{\eps}(x, |w(x, 0)|)w(x, 0) z_n(x, 0) -f_{\eps}(x, |w_n(x, 0)|)
w_n(x, 0) z_n(x, 0)\right) \, dx.
\end{split}
\end{align}
Since $z_n \wto 0$ in $X^{1, s}(\R^{N+1}_+)$ as $n\to \infty$, it then follows from Lemma \ref{embedding} that
$$
\int_{\R^N} f_{\eps}(x, |w(x, 0)|)w(x, 0) z_n(x, 0) \, dx=o_n(1)
$$
and
$$
\int_{\R^N} f_{\eps}(x, |w_n(x, 0)|)w(x, 0) z_n(x, 0) \,dx =o_n(1).
$$
Using again the assumption that $V(x) \geq a$ for any $x \in \R^N$ and the fact that $0 \leq g(|t|) \leq \frac a 4$ for any $t \geq 0$, from \eqref{conv1}, we then derive that
\begin{align} \label{conv11}
\min\left\{1, \frac{3a}{4}\right\} \|z_n\|_{1, s}^2 \leq
\int_{\R^N} \left(1- \chi(\eps x)\right)|w_n(x, 0)|^{p-2}|z_n(x,
0)|^2 \, dx + o_n(1).
\end{align}
Observe that $\mbox{supp}\, (1 -\chi(\eps x))$ is bounded in $\R^N$ for any $\eps>0$ and $z_n(\cdot, 0) \to 0$ in $L^q_{\textnormal{loc}}(\R^N)$ for any $2 \leq
q <2^*_s$ as $n \to \infty$. Applying \eqref{conv1}, we then obtain that $\|z_n\|_{1,s}=o_n(1)$. This shows that $\{w_n\}$ has a convergent subsequence in $X^{1, s}(\R^{N+1}_+)$, and the proof is completed.
\end{proof}

We now introduce some notations. Let
$$
P_+:=\{w \in X^{1, s}(\R^{N+1}_+)\ |\  w \geq 0 \}, \quad P_-:=\{w
\in X^{1, s}(\R^{N+1}_+)\ |\  w \leq 0 \}.
$$
For $\sigma>0$, we define
\begin{align} \label{defp1}
P_+^{\sigma}:=\{w \in X^{1, s}(\R^{N+1}_+)\ |\  \mbox{dist}_{X^{1,
s}}(w, P_+) : =\inf_{z \in P_+} \|w-z\|_{1, s}< \sigma\}
\end{align}
and
\begin{align} \label{defp2}
P_-^{\sigma}:=\{w \in X^{1, s}(\R^{N+1}_+)\ |\  \mbox{dist}_{X^{1,
s}}(w, P_-) : =\inf_{z \in P_-} \|w-z\|_{1, s}< \sigma\}.
\end{align}

\begin{lem} \label{infenergy1}
There exists a constant $\sigma_0>0$ such that, for any $0<\sigma<\sigma_0$, there holds that
\begin{align}\label{mpv1}
c^*:=\inf_{w \in \Sigma} \Phi_{\eps}(w)>0,
\end{align}
where $\Sigma:=\partial(P^{\sigma}_+) \cap \partial(P^{\sigma}_-)$.
\end{lem}
\begin{proof}
In view of the definition of $F_{\eps}$, we have that $|F_{\eps}(x, |t|)| \leq \frac 1 p |t|^{p}$ for any $ x \in \R^{N}$ and $t \in \R$.  Furthermore, according to $(V_1)$, we know that $V(x) \geq a$ for any $x \in \R^N$. Therefore, from Lemma \ref{embedding}, we obtain that
\begin{align*}
\Phi_{\eps}(w) &\geq \frac {k_s}{2} \int_{\R^{N+1}_+} y^{1-2s} |\nabla w|^2 \, dxdy + \frac {a}{2} \int_{\R^N} |w(x, 0)|^2 \, dx - \frac 1 p\int_{\R^N} |w(x, 0)|^p \, dx\\
& \geq \frac 12 \min \{1, a\}\|w\|_{1, s}^2 - C_{p, N} \|w\|_{1, s}^{p}.
\end{align*}
Due to $p>2$, then \eqref{mpv1} follows immediately, and the proof is completed.
\end{proof}

Since $0 \in \mathcal{V}$, then $B_{1}(0) \subset (\mathcal{V}^{\delta})_{\eps}$ for any $\delta>0$ and $\eps>0$ small enough.
We now introduce an auxiliary functional
\begin{align} \label{defphi0}
\hspace{-1cm}\Phi_0(w):=\frac {k_s}{2} \int_{\mathcal{C}_{B_1(0)}} y^{1-2s}
|\nabla w|^2 \, dxdy + \frac {b}{2} \int_{B_1(0)} |w(x, 0)|^2 \,
dx - \frac {1} {p} \int_{B_1(0)} |w(x, 0)|^p \, dx,
\end{align}
where $\mathcal{C}_{B_1{(0)}}:=B_1{(0) \times (0, \infty)} \subset \R^{N+1}_+$. Let $X^{1, s}(\mathcal{C}_{B_1(0)})$ be the completion of $C^{\infty}_{0}(\overline{\mathcal{C}_{B_1(0)}})$ under the norm
$$
\|w\|^2_{X^{1,
s}(\mathcal{C}_{B_1(0)})}:=k_s\int_{\mathcal{C}_{B_1(0)}} y^{1-2s}
|\nabla w|^2 \, dxdy + \int_{B_1(0)} |w(x, 0)|^2 \, dx.
$$
Let $\{e_n\} \subset X^{1, s}(\mathcal{C}_{B_1(0)})$ be an orthonormal basis to $X^{1, s}(\mathcal{C}_{B_1(0)})$ and $E_n:=\mbox{span}\{e_1, \dots, e_n\}$. Without restriction, we shall assume that $e_1>0$ in $\mathcal{C}_{B_1(0)}$. Notice that $p>2$, then it is easy to see that there exists a positive increasing sequence $\{R_n\}\subset \R^+$ such that
\begin{align}\label{cmvnbjjgug77tyruur1}
\Phi_0(w) <0 \quad \mbox{for any} \,\, w \in E_n, \|w\|_{1, s} \geq R_n.
\end{align}
For $n \in \mathbb{N}$, we are now able to define $\varphi_n \in C(B_n, X^{1, s}(\mathcal{C}_{B_1(0)}))$ by
\begin{align} \label{defvarphi}
\varphi_n(x):=R_n \sum_{i=1}^n x_i e_i.
\end{align}
Let $W:=P_+^{\sigma} \cup P_-^{\sigma}$ and $M:=P_+^{\sigma} \cap P_-^{\sigma}$. Applying Lemma \ref{infenergy1}, we then get that $\varphi_n$ satisfies the conditions $(\textnormal{i})$-$(\textnormal{iii})$ in Theorem \ref{exist}.

For $j \in \mathbb {N}$, we define
$$
c_{j, \eps}:=\left\{
\begin{array}
[c]{ll} \displaystyle\inf_{B\in \Lambda_j}\sup_{u\in B\setminus
W} \Phi_{\eps}(u),& \mbox{if} \ j\geq 2,\\
\displaystyle\inf_{B\in \Lambda_1}\sup_{u\in
B} \Phi_{\eps}(u), & \mbox{if} \ j= 1,
\end{array}
\right.
$$
and
$$
\widetilde{c}_j:=\left\{
\begin{array}
[c]{ll} \displaystyle\inf_{B \in \widetilde{\Lambda}_j} \sup_{u \in B \backslash W}
 \Phi_0(u),& \mbox{if} \ j\geq 2,\\
\displaystyle\inf_{B\in \widetilde{\Lambda}_1}\sup_{u\in
B} \Phi_0(u), & \mbox{if} \ j= 1,
\end{array}
\right.
$$
where
\begin{align} \label{defLa1}
\begin{split}
&\hspace{-1cm}\Lambda_j:=\left\{B \subset X^{1, s}(\R^{N+1}_+) \ |\ B=\varphi (B_n \backslash Y) \,\, \mbox{for some} \,\, \varphi \in G_n, n \geq j \,\, \mbox{and} \right.\\
& \hspace{-1cm}\hspace{2.5cm} \left. \quad \hspace{1.5cm} Y \subset B_n  \,\, \mbox{is a symmetric open set with }\,\,  \gamma(\overline{Y}) \leq n-j \right \}, \\
&\hspace{-1cm}G_n:=\{\varphi \in C(B_n, X^{1, s}(\R^{N+1}_+)) \ |\ \varphi(-x)=-\varphi(x) \, \,
\mbox{for any} \,\, x \in B_n, \varphi_{\mid_{\partial B_n}}={\varphi_n}_{|_{\partial B_n}}\}
\end{split}
\end{align}
and
\begin{align} \label{defLa2}
\begin{split}
&\hspace{-1cm}\widetilde{\Lambda}_j:=\left\{B \subset X^{1, s}(\mathcal{C}_{B_1(0)} ) \ |\ B=\varphi (B_n \backslash Y) \,\, \mbox{for some} \,\, \varphi \in \tilde{G}_n, n \geq j \,\, \mbox{and} \right.\\
&\hspace{-1cm} \hspace{2.5cm} \left. \quad \hspace{1.5cm}  Y \subset B_n  \,\, \mbox{is a symmetric open set with }\,\,  \gamma(\overline{Y}) \leq n-j \right \}, \\
&\hspace{-1cm}\widetilde{G}_n:=\{\varphi \in C(B_n, X^{1, s}(\mathcal{C}_{B_1(0)})) \ |\ \varphi(-x)=-\varphi(x) \, \,
\mbox{for any} \,\, x \in B_n, \varphi_{\mid_{\partial
B_n}}={\varphi_n}_{|_{\partial B_n}}\}.
\end{split}
\end{align}
According to the definitions of $c_{j,\eps}$ and $\tilde{c}_j$, it is not hard to deduce that
\begin{align} \label{cenergy11}
c_{1, \eps}>0, \quad 0< c_{2,\eps} \leq  \dots \leq c_{j, \eps} \leq \cdots
\end{align}
and
\begin{align} \label{cenergy12}
\widetilde{c}_1>0, \quad 0<\widetilde{c}_2\leq  \cdots \leq \widetilde{c}_j \leq \cdots.
\end{align}
Since $V(x) \leq b$ for any $x \in \R$, then $\Phi_{\eps}(w) \leq \Phi_{0}(w)$ for any $w \in X^{1, s}(\mathcal{C}_{B_1(0)})$. In addition, there holds that $\tilde{\Lambda}_j \subset \Lambda_j$ for any $j \in \mathbb{N}$. Hence, for any $\eps>0$ small enough, we obtain that
\begin{align} \label{cenergy21}
0<c_{j, \eps} \leq \widetilde{c}_j \quad \mbox{for any} \,\, j \geq 1.
\end{align}

Next we are going to prove that $P_+^{\sigma}$ is an admissible invariant set with respect to $\Phi_{\eps}$ at level $c$ for $c \geq c^*$, where $c^*>0$ is defined by Lemma \ref{infenergy1}. To this end, we need to introduce an operator $A_{\eps}$ on $X^{1, s}(\R^{N+1}_+)$. For $z \in X^{1, s}(\R^{N+1}_+)$, we define $w:=A_{\eps}(z) \in X^{1, s}(\R^{N+1}_+)$ by the unique solution to the equation
\begin{align*}
\left\{
\begin{aligned}
-\mbox{div}(y^{1-2s} \nabla w)&=0 \quad \hspace{3.5cm}\mbox{in} \,\, \R^{N+1}_+,\\
-k_s \frac{\partial w}{\partial {\nu}}&= -V_{\eps}(x) w+f_{\eps}(x, |z|) z \qquad \mbox{on} \,\, \R^N \times \{0\}.
\end{aligned}
\right.
\end{align*}

\begin{lem} \label{conti1}
The operator $A_{\eps}$ is well-defined and continuous on $X^{1,s}(\R^{N+1}_+)$.
\end{lem}
\begin{proof}
By the definition of $A_{\eps}$, it is not difficult to deduce that $A_{\eps}$ is well-defined on $X^{1,s}(\R^{N+1}_+)$. We now show that $A_{\eps}$ is continuous on $X^{1,s}(\R^{N+1}_+)$.  Supposing $z_n \to z$ in $X^{1, s}(\R^{N+1}_+)$ as $n \to \infty$, we shall prove that $A_{\eps}(z_n) \to A_{\eps}(z)$ in $X^{1, s}(\R^{N+1}_+)$ as $n \to \infty$. For simplicity, let us define $w_n:=A_{\eps}(z_n)$ and $w:=A_{\eps}(z)$. Observe that
$$
k_s\int_{\R^{N+1}_+} y^{1-2s} |\nabla w_n|^2 \, dxdy + \int_{\R^N} V_{\eps}(x)|w_n(x, 0)|^2 \, dx =\int_{\R^N} f_{\eps}(x, |z_n(x, 0)|) z_n(x, 0) w_n(x, 0) \, dx.
$$
In view of the definition of $f_{\eps}$, we find that $ |f_{\eps}(x, |t|)| \leq |t|^{p-2}$ for any $x \in \R^N$ and $t \in \R$. Using H\"older's inequality and Lemma \ref{embedding}, we then have that
\begin{align*}
\left|\int_{\R^N} f_{\eps}(x, |z_n(x, 0)|) z_n(x, 0) w_n(x, 0) \, dx \right| &\leq C \|z_n(x, 0)\|_{L^p(\R^N)}^{p-1}\|w_n(x, 0)\|_{L^p(\R^N)} \\
&\leq C \|z_n\|^{p-1}_{1, s}\|w_n\|_{1, s}.
\end{align*}
Therefore, $\{w_n\}$ is bounded in $X^{1, s}(\R^{N+1}_+)$. This suggests that there exists $w \in X^{1, s}(\R^{N+1}_+)$ such that $w_n \wto w$ in $X^{1, s}(\R^{N+1}_+)$ as $n \to \infty$.
Note that
\begin{align} \label{Aeps}
\begin{split}
&k_s\int_{\R^{N+1}_+} y^{1-2s}  |\nabla w_n-\nabla w|^2 \, dxdy + \int_{\R^N} V_{\eps}(x) |w_n(x, 0)-w(x, 0)|^2  \, dx\\
&= \int_{\R^N} \left(f_{\eps}(x, |z_n(x, 0)|) z_n(x, 0)-f(x, |z(x, 0)|) z(x, 0)\right) (w_n(x, 0)-w(x, 0))\, dx \\
&= \int_{\R^N} \left(1-\chi(\eps x) \right)\left(|z_n(x, 0)|^{p-2} z_n(x, 0)-|z(x, 0)|^{p-2} z(x, 0) \right) (w_n(x, 0)-w(x, 0)) \, dx \\
& \quad  + \int_{\R^N} \chi(\eps x) \left(g(|z_n(x, 0)|)z_n(x, 0)-g(|z(x, 0)|)z(x, 0)\right) (w_n(x, 0)-w(x, 0))\, dx.
\end{split}
\end{align}
Since $z_n \to z$ in $X^{1, s}(\R^{N+1}_+)$ as $n \to \infty$, by H\"older's inequality, the mean value theorem and Lemma \ref{embedding}, we obtain that
\begin{align*}
&\left| \int_{\R^N} \left(1-\chi(\eps x) \right)\left(|z_n(x, 0)|^{p-2} z_n(x, 0)-|z(x, 0)|^{p-2} z(x, 0) \right) (w_n(x, 0)-w(x, 0)) \, dx\right| =o_n(1)
\end{align*}
and
$$
\left|\int_{\R^N} \chi(\eps x)\left(g(|z_n(x, 0)|)z_n(x, 0)-g(|z(x, 0)|)z(x, 0)\right) (w_n(x, 0)-w(x, 0))\, dx \right| =o_n(1).
$$
It then follows from \eqref{Aeps} that $\|w_n-w\|_{1, s}=o_n(1)$. As a consequence, $A_{\eps}$ is continuous on $X^{1, s}(\R^{N+1}_+)$, and the proof is completed.
\end{proof}

\begin{lem}
For any $z \in X^{1, s}(\R^{N+1}_+)$, there holds that $\|\Phi_{\eps}'(z)\| \leq C\|z-A_{\eps}(z)\|_{1, s}$ for some $C>0$.
\end{lem}
\begin{proof}
For any $\psi \in X^{1, s}(\R^{N+1}_+)$, from the definition of $A_{\eps}$, we derive that
\begin{align*}
\Phi_{\eps}'(z) \psi &=k_s\int_{\R^{N+1}_+} y^{1-2s} \nabla z \cdot \nabla \psi \, dxdy + \int_{\R^N} V_{\eps}(x) z(x, 0) \psi(x, 0) \, dx \\
& \quad - \int_{\R^N} f_{\eps}(x, |z(x, 0)|) z(x, 0) \psi(x, 0) \, dx \\
&=k_s\int_{\R^{N+1}_+} y^{1-2s} \left(\nabla z-\nabla A_{\eps}(z) \right) \cdot \nabla \psi \, dxdy \\
& \quad + \int_{\R^N} V_{\eps}(x) \left(z(x, 0)- A_{\eps}(z)(x, 0)\right)\psi(x, 0) \, dx.
\end{align*}
By H\"older's inequality, the lemma then follows immediately, and the proof is completed.
\end{proof}

\begin{lem}
There exists a constant $\sigma_0>0$ such that, for any $0<\sigma < \sigma_0$,
$$
A_{\eps}(\partial P_-^{\sigma}) \subset P_-^{\sigma}, \quad
A_{\eps}(\partial P_+^{\sigma}) \subset P_+^{\sigma}.
$$
\end{lem}
\begin{proof}
For simplicity, we only prove that $A_{\eps}(\partial
P_-^{\sigma}) \subset P_-^{\sigma}$. By adapting a similar way,
one can easily show that $A_{\eps}(\partial P_+^{\sigma}) \subset
P_+^{\sigma}$. For any $z \in \partial P_-^{\sigma}$, we shall
prove that $w=A_{\eps}(z) \in P_-^{\sigma}$ for any $\sigma>0$
small enough. Note that $V(x) \geq a$ for any $x \in \R^N$ and $|f_{\eps}(x, |t|)| \leq |t|^{p-2}$ for any $x \in \R^N$ and $t \in \R$, then
\begin{align*}
\mbox{dist}_{X^{1, s}}(w, P^-) \|w^+\|_{1, s}  &\leq \|w^+\|_{1, s}^2 \\
& \leq C \left(k_s\int_{\R^{N+1}} y^{1-2s} \nabla w \cdot \nabla w^+ \, dxdy + \int_{\R^N} V_{\eps}(x) w(x, 0) w^+(x, 0) \, dx \right) \\
& = C\int_{\R^N} f_{\eps}(x, |z(x, 0)|) z(x, 0) w^+(x, 0) \, dx \\
& \leq C\|z^+\|_{L^p(\R^N)}^{p-1} \|w^+\|_{L^p(\R^N)}
=C\left(\mbox{dist}_{L^p}(z, P_-)\right)^{p-1}
\|w^+\|_{L^p(\R^N)},
\end{align*}
where $w^+:=\max\{w, 0\}$, $C:=(\min\{1,a\})^{-1}$ and we used H\"older's inequality and the fact that
$$
\mbox{dist}_{L^q}(z, P_-):=\inf_{v \in
P_-}\|z-v\|_{L^p(\R^N)}=\|z^+\|_{L^q(\R^N)} \quad \mbox{for any} \,\, 2 \leq q \leq 2^*_s.
$$
As a result of Lemma \ref{embedding}, we then have that $\mbox{dist}_{X^{1,
s}}(w, P_-) \leq C \sigma^{p-1}$, from which the lemma follows, and the proof is completed.
\end{proof}

Note that we can only prove that $A_{\eps}$ is continuous on $X^{1, s}(\R^{N+1}_+)$, see Lemma \ref{conti1}. We now introduce a locally Lipschitz perturbation of $A_{\eps}$.

\begin{lem}
Let $K$ be the set of fixed points of $A_{\eps}$, then there exists a locally Lipschitz continuous operator $B_{\eps}: X^{1, s}(\R^{N+1}_+) \backslash K  \to X^{1, s}(\R^{N+1}_+)$ such that
\begin{enumerate}
\item[$(\textnormal{i})$]  $B_{\eps}(\partial P_-^{\sigma})
\subset P_-^{\sigma}$, $B_{\eps}(\partial P_+^{\sigma}) \subset
P_+^{\sigma}$ for any $0<\sigma<\sigma_0$;
\item[$(\textnormal{ii})$] $\frac 12 \|z-B_{\eps}(z)\|_{1,s} \leq
\|z-A_{\eps}(z)\|_{1,s} \leq 2  \|z-B_{\eps}(z)\|_{1,s} $ for any
$z \in X^{1, s}(\R^{N+1}_+) \backslash K$;
\item[$(\textnormal{iii})$] $\langle \Phi_{\eps}(z), z-B_{\eps}(z)
\rangle \geq \frac 12 \|z-A_{\eps}(z)\|^2_{1, s} $ for any $z \in
X^{1, s}(\R^{N+1}_+) \backslash K$; \item[$(\textnormal{iv})$]
$B_{\eps}$ is odd.
\end{enumerate}
\end{lem}
\begin{proof}
The proof of this lemma is similar to the ones of \cite[Lemma 4.1]{BL} and \cite[Lemma 2.1]{BLW}, then we omit its proof.
\end{proof}

At this point, using the same arguments as the proof of \cite[Lemma 3.6]{LWZ}, we are able to derive the following result.

\begin{lem} \label{deformation}
Let $\mathcal{N}$ be a symmetric closed neighborhood of $K_c$, then there exists a constant $\tau_0>0$ such that, for $0<\tau<\bar{\tau} <\tau_0$, there is a continuous map $\zeta: [0, 1] \times X^{1, s}(\R^{N+1}_+) \to X^{1, s}(\R^{N+1}_+)$ satisfying
\begin{enumerate}
\item[$(\textnormal{i})$] $\zeta(0, w)=w$ for any $w \in X^{1.
s}(\R^{N+1}_+)$; \item[$(\textnormal{ii})$] $\zeta(t, w)=w$ for
any $t \in [0, 1]$ and $\Phi_{\eps}(w) \not\in [c-\bar{\tau}, c+
\bar{\tau}]$; \item[$(\textnormal{iii})$] $\zeta(t, -w)=-\zeta(t,
w)$ for any $t \in [0, 1]$  and $u \in X^{1, s}(\R^{N+1}_+)$;
\item[$(\textnormal{iv})$] $\zeta (1, \Phi_{\eps}^{c+
\tau}\setminus \mathcal{N}) \subset \Phi_{\eps}^{c-\tau}$;
\item[$(\textnormal{v})$] $\zeta(t,
\partial P_+^{\sigma}) \subset P_+^{\sigma}$,  $\zeta(t,
\partial P_-^{\sigma}) \subset P_-^{\sigma}$,  $\zeta(t,
P_+^{\sigma}) \subset P_+^{\sigma}$ and $\zeta(t, P_-^{\sigma})
\subset P_-^{\sigma}$ for any $t\in[0,1]$.
\end{enumerate}
\end{lem}

Let $D$ be a closed symmetric neighborhood of $K_c \backslash W$,
then $\mathcal{N}:=D \cup \overline{P_+^{\sigma}} \cup
\overline{P_-^{\sigma}}$ is a closed symmetric neighborhood of
$K_c$. Define $\eta:=\zeta(1, \cdot)$, then Lemma
\ref{deformation} leads to the following desired result.

\begin{lem} \label{adm}
There exists a constant $\sigma_0>0$ such that, for any $0< \sigma< \sigma_0$,
$P_+^{\sigma}$ is an admissible invariant set with respect to $\Phi_{\eps}$ at level $c$ for $c \geq c^*$.
\end{lem}

\begin{thm}\label{cvnbmiif9f8ufjhfy1}
For any $k \in \mathbb{N}$, there exists a constant $\eps_k>0$
such that, for any $0<\eps<\eps_k$, \eqref{equ21} admits at least
$k$ pairs of solutions $\pm w_{j, \eps} \in X^{1, s}(\R^{N+1}_+)$ satisfying $\Phi_{\eps}(w_{j, \eps})=c_{j,\eps} \leq \widetilde{c}_{k}$ for any $1 \leq j \leq k$. Moreover, $w_{j,\eps}$ is sign-changing solution to \eqref{equ21} for any $2 \leq j \leq k$.
\end{thm}
\begin{proof}
From the proof of \cite[Theorem 9.12]{Rabinowitz} and \eqref{cenergy21}, we then obtain the result of this theorem for $j=1.$ From Theorem \ref{exist} and \eqref{cenergy11}-\eqref{cenergy21}, we then obtain the result of this theorem for any $j \geq 2$.
\end{proof}

\subsection{Decay estimates of semiclassical states}

In what follows, our goal is to deduce decay estimates of the solutions obtained in Theorem \ref{cvnbmiif9f8ufjhfy1}.

\begin{lem} \label{bddsol}
Let $w_{j,\eps} \in X^{1, s}(\R^{N+1}_+)$ be the solution to \eqref{equ21} obtained in Theorem \ref{cvnbmiif9f8ufjhfy1}. Then, for any $k \in \mathbb{N}$ and $0<\eps<\eps_k$, there exist a constant $\rho>0$ depending only on $p, N$ and a constant $\eta_k>0$ independent of $\eps$ such that $\rho \leq \|w_{j, \eps}\|_{1, s} \leq \eta_k$ for any $1 \leq j \leq k$, where the constant $\eps_k$ is determined in Theorem \ref{cvnbmiif9f8ufjhfy1}.
\end{lem}
\begin{proof}
From the definition of $f_{\eps}$, we know that $f_{\eps}(x,|t|) \leq |t|^{p-2}$ for any $x \in \R^N$ and $t \in \R$. Since $\Phi_{\eps}'(w_{j, \eps}) w_{j, \eps}=0$, then
\begin{align} \label{ide} \nonumber
k_s\int_{\R^{N+1}_+} y^{1-2s} |\nabla w_{\epsilon,j}|^2 \, dxdy +
\int_{\R^N} V_{\eps}(x) |w_{\epsilon,j}(x, 0)|^2 \, dx &=\int_{\R^N}
f_{\eps}(x, |w_{j, \eps}(x, 0)|) |w_{j, \eps}(x, 0)|^2 \,dx\\
& \leq \int_{\R^N} |w_{j, \eps}(x, 0)|^p \, dx.
\end{align}
Due to $p>2$, it then follows from Lemma \ref{embedding} that there exists a constant
$\rho>0$ depending only on $p$ and $N$ such that $\|w_{j, \eps}\|_{1, s} \geq \rho$. Since
$\Phi_{\eps}'(w_{j, \eps})=0$ and $\Phi_{\eps}(w_{j, \eps})  \leq
\widetilde{c}_{k}$ for any $1 \leq j \leq k$, by using the same
way as the proof of the boundedness of the Palais-Smale sequence in
Lemma \ref{ps1}, we then obtain that there exists a constant $\eta_k>0$ independent of $\eps$ such
that $\|w_{j, \eps}\|_{1, s}\leq \eta_k$. Thus we have completed the proof.
\end{proof}


\begin{lem} \label{lions} \cite[Lemma 3.3]{HZ}
If $\{w_n\} $ is bounded in $X^{1, s}(\R^{N+1}_+)$ and
$$
\lim_{n \to \infty} \sup_{z \in \R^N} \int_{B_R{(z)}} |w_n(x,0)|^2
\, dx =0
$$
 for some $R>0,$ then $w_n(\cdot,0) \to 0$ in $L^p(\R^N)$ for any $2 <q< 2^*_s$.
\end{lem}

\begin{lem}\label{cnvbhhfyf6yr66}
Let $w_{j,\eps} \in X^{1, s}(\R^{N+1}_+)$ be the solution to \eqref{equ21} obtained in Theorem \ref{cvnbmiif9f8ufjhfy1}. Then, for any $1 \leq j \leq k$, there exist a constant $m_j \in \mathbb{N}$, $m_j$ nontrivial functions $w_{j, l} \in X^{1, s}(\R^{N+1}_+)$ and $m_j$ sequences $\{z_{j, \eps}^l\} \subset \R^N$ for $1 \leq l \leq m_j$ such that, up to subsequences if necessary,
\begin{enumerate}
\item[$(\textnormal{i})$] $\eps z_{j, \eps}^l \to z_j^l \in
\Lambda^{\delta_0}$ and $|z_{j, \eps}^l-z_{j, \eps}^{l'}| \to
+\infty$ as $\eps \to 0^+$ for any $1 \leq l \neq l' \leq m_j$.
\item[$(\textnormal{ii})$] There holds that
\begin{align}\label{wnj}
w_{j, \eps} -\sum_{l=1}^{m_j} w_{j,l}(\cdot-z_{j, \eps}^l, \cdot) =o_{\eps}(1) \quad \mbox{in} \,\, X^{1, s}(\R^{N+1}_+)
\end{align}
and $w_{j,l}$ is a nontrivial solution to the equation
\begin{align}\label{cdhfdhfgyyfyf66ftf1}
\left\{
\begin{aligned}
-\mbox{div} \, (y^{1-2s} \nabla w_{j,l})&=0 \hspace{4cm}\qquad \quad \mbox{in} \,\, \R^{N+1}_+,\\
-k_s \frac{\partial w_{j,l}}{\partial {\nu}}&=-V(z_j^l) w_{j,l} + f(z_j^l,
|w_{j,l}|) w_{j,l} \ \qquad \mbox{on}\,\, \R^N \times \{0\},
\end{aligned}
\right.
\end{align}
where $f(x, t):=\left(1- \chi(x)\right) t^{p-2}+\chi(x){g}(t)$ for $x \in \R^N$ and $t \in \R$.
\end{enumerate}
\end{lem}
\begin{proof}
We first claim that there exists a constant $\beta_0>0$ such that
\begin{align} \label{con1}
\lim_{\eps \to 0^+} \sup_{z \in \R^N} \int_{B_1(z)} |w_{j,\eps}(x, 0)|^2 \, dx \geq \beta_0.
\end{align}
Otherwise, from Lemma \ref{lions}, we can obtain that $w_{j, \eps}(\cdot, 0) \to 0$ in $L^{q}(\R^N)$ as $\eps \to 0^+$ for any $2<q<2^*_s$. By the definition of $f_{\eps}$, we know that $|f_{\eps}(x, |t|)|
\leq t^{p-2}$ for any $x \in \R^N$ and $t \in \R$. Therefore, it holds that
\begin{align} \label{1zero}
\int_{\R^N} f_{\eps}(x, |w_{j, \eps}(x, 0)|) |w_{j, \eps}(x, 0)|^2 \, dx =o_{\eps}(1).
\end{align}
Since $\Phi_{j, \eps}'(w_{j, \eps})w_{j, \eps}=0$, it then follows
from \eqref{ide} and \eqref{1zero} that $\|w_{j, \eps}\|_{1,
s}=o_{\eps}(1)$, which contradicts Lemma \ref{bddsol}. This infers
that \eqref{con1} holds true. Hence we get that there exists a sequence $\{z_{j, \eps}^1\} \subset \R^N$
such that
\begin{align} \label{com1}
\int_{B_1(z_{j, \eps}^1)} |w_{j, \eps}(x, 0)|^2 \,dx \geq \frac{\beta_0}
{2}.
\end{align}
Setting $\widetilde{w}_{j, \eps}:=w_{j, \eps}(\cdot+z_{j, \eps}^1, \cdot)$ and applying \eqref{com1} and Lemma \ref{embedding}, we then obtain that there exists $w_{j,1} \in X^{1, s}(\R^{N+1}_+)$ such that $\widetilde{w}_{j, \eps} \wto w_{j,1}$ in $X^{1, s}(\R^{N+1}_+)$ as $\eps \to 0^+$ and $w_{j,1} \neq 0$. In addition, we have that
\begin{align*}
\left\{
\begin{aligned}
-\mbox{div}(y^{1-2s} \nabla \widetilde{w}_{j, \eps})&=0 \hspace{7cm}\,\,\mbox{in} \,\, \R^{N+1}_+,\\
-k_s \frac{\partial \widetilde{w}_{j, \eps}}{\partial {\nu}}&=-V_{\eps}(x+z_{j, \eps}^1) \widetilde{w}_{j, \eps} + f_{\eps}(x+z_{j, \eps}^1, |\widetilde{w}_{j, \eps}|) \widetilde{w}_{j, \eps} \quad \quad \ \, \, \,\mbox{on} \,\, \R^N \times \{0\}.
\end{aligned}
\right.
\end{align*}
We now deduce that $\eps z_{j, \eps}^1 \to z_j^1 \in \Lambda^{\delta_0}$ in $\R^N$ as $\eps \to 0^+$. To do this, we first prove that $\{\eps z_{j, \eps}^1\} \subset \R^N$ is bounded in $\R^N$. We argue by contradiction that $\{\eps z_{j, \eps}^1\}$ is unbounded in $\R^N$. We may assume that $\eps |z_{j, \eps}| \to + \infty$ as $\eps \to 0^+$. By the definition of $f_{\eps}$, then $w_{j,1}$ satisfies the equation
\begin{align} \label{11equw1}
\left\{
\begin{aligned}
-\mbox{div}(y^{1-2s} \nabla w_{j,1})&=0 \hspace{3cm} \qquad \, \ \ \ \mbox{in} \,\, \R^{N+1}_+,\\
-k_s \frac{\partial w_{j,1}}{\partial {\nu}}&=-V_1 w_{j,1}+ g(|w_{j,1}|)w_{j,1} \quad \quad \mbox{on} \,\, \R^N \times \{0\},
\end{aligned}
\right.
\end{align}
where $V_1:=\liminf_{\epsilon\rightarrow 0^+}V(\epsilon
z^1_{\epsilon,j})$. Multiplying \eqref{11equw1} by $w_{j,1}$ and
integrating on $\R^{N+1}_+$, we derive that
\begin{align*}
k_s\int_{\R^{N+1}_+} y^{1-2s} |\nabla w_{j,1}|^2 \, dxdy + \int_{\R^N} V_1 |w_{j,1}(x, 0)|^2 \, dx &=\int_{\R^N} g (|w_{j,1}(x, 0)|) |w_{j,1}(x, 0)|^2 \,dx \\
& \leq \frac{a}{4} \int_{\R^2} |w_{j,1}(x, 0)|^2 \, dx.
\end{align*}
where we used the fact that $g(|t|) \leq a/4$ for any $t \in \R$. This suggests that $w_{j,1}=0$ and we then reach a contradiction. Thus we have that $\{\eps z_{j, \eps}^1\}$ is bounded in $\R^N$. Consequently, we know that there exists $z_j^1 \in \R^N$ such that $\eps z_{j, \eps}^1 \to z_{j}^1$ in $\R^N$ as $\eps \to 0^+$. It is not hard to see that $z_j^1 \in \Lambda^{\delta_0}$. Contrarily, if there holds that $z_j^1 \notin \Lambda^{\delta_0}$, then we can similarly get that $w_{j,1}=0$ and this is impossible. As a result, we conclude that $w_{j,1}$ satisfies the equation
\begin{align} \label{equw1}
\left\{
\begin{aligned}
-\mbox{div}(y^{1-2s} \nabla w_{j,1})&=0 \hspace{4.5cm} \quad \ \ \ \mbox{in} \,\, \, \R^{N+1}_+,\\
-k_s \frac{\partial w_{j,1}}{\partial {\nu}}&=-V(z_j^1) w_{j,1}+ f(z_j^1, |w_{j,1}|)w_{j,1} \, \quad \quad \mbox{on} \,\, \R^N \times \{0\},
\end{aligned}
\right.
\end{align}
Moreover, by using the same way as the proof of Lemma \ref{bddsol}, we are able to show that there is a constant $\tau>0$ depending only on $p$ and $N$ such that $\|w_{j,1}\| \geq \tau$.

Let $w_{j, \eps}^1:=w_{j, \eps}-w_{j,1}(\cdot-z_{j, \eps}^1, \cdot)$. If $w_{j, \eps}^1 \to 0$ in $X^{1, s}(\R^{N+1}_+)$ as $\eps \to 0^+$, then the proof is completed. Otherwise, we may assume that $ \lim_{\eps \to 0^+}\|w_{j, \eps}^1\|_{1, s} >0$. Setting
$$
v_{j, \eps}^1:=w_{j, \eps}^1(\cdot +z_{j, \eps}^1, \cdot)=w_{j, \eps}(\cdot + z_{j, \eps}^1, \cdot) -w_{j,1},
$$
we then have that $v_{j, \eps}^1 \wto 0$ in $X^{1, s}(\R^{N+1}_+)$ as $\eps \to 0^+$ and
\begin{align} \label{lieb1}
\|w_{j, \eps}^1\|_{1, s}^2=\|v_{j, \eps}^1\|_{1, s}^2=\|w_{j, \eps}\|_{1, s}^2- \|w_{j,1}\|_{1, s}^2 +o_{\eps}(1).
\end{align}
Furthermore, it is easy to deduce that
\begin{align*}
\left\{
\begin{aligned}
-\mbox{div}(y^{1-2s} \nabla {w}_{j, \eps}^1)&=0 \hspace{6cm} \quad \mbox{in} \,\, \R^{N+1}_+,\\
-k_s \frac{\partial {w}_{j, \eps}^1}{\partial {\nu}}&=-V_{\eps}(x) {w}_{j, \eps}^1 + f_{\eps}(x, |{w}_{j, \eps}^1|) {w}_{j, \eps}^1 +o_{\eps}(1) \quad \quad \, \, \,\mbox{on} \,\, \R^N \times \{0\}.
\end{aligned}
\right.
\end{align*}
Consequently, we find that there is a constant $\beta_1>0$ such that
\begin{align} \label{con2}
\lim_{\eps \to 0^+} \sup_{z \in \R^N} \int_{B_1(z)} |w_{\eps,
j}^1(x, 0)|^2 \, dx \geq \beta_1.
\end{align}
Otherwise, with the help of Lemma \ref{lions}, we can derive that $\|w_{j, \eps}^1\|_{1,s}=o_{\eps}(1)$. This is impossible, because we
assumed that $ \lim_{\eps \to 0^+}\|w_{j, \eps}^1\|_{1, s} >0$.
Hence, making use of \eqref{con2}, we obtain that there exists a sequence
$\{z_{j, \eps}^2\} \subset \R^N$ such that
\begin{align} \label{com2}
\int_{B_1(z_{j, \eps}^2)} |w_{j, \eps}^1(x, 0)|^2 \,dx
=\int_{B_1(z_{j, \eps}^2)} |v_{j, \eps}^1(x-z_{j, \eps}^1, 0)|^2
\,dx \geq \frac{\beta_1} {2}.
\end{align}
Since $v_{j, \eps}^1 \wto 0$ in $X^{1, s}(\R^{N+1}_+)$ as $\eps
\to 0^+$, it then yields from \eqref{com2} and Lemma
\ref{embedding} that $|z_{j, \eps}^1-z_{j, \eps}^2| \to + \infty$
as $ \eps \to 0^+$. Applying \eqref{com2}, we also have that there
exists $w_{j,2} \in X^{1, s}(\R^{N+1}_+)$ such that $w_{\eps,
j}^1(\cdot+z_{j, \eps}^2, \cdot) \wto w_{j,2}$ in $X^{1, s}(\R^{N+1}_+)$
as $\eps \to 0^+$ and $w_{j,2} \neq 0$. In addition, we can
obtain that $\eps z_{j, \eps}^2 \to z_{j}^2 \in
\Lambda^{\delta_0}$ in $\R^N$ as $\eps \to 0^+$, $w_{j,2}$ satisfies the equation
\begin{align*}
\left\{
\begin{aligned}
-\mbox{div}(y^{1-2s} \nabla w_{j,2})&=0 \hspace{4.25cm} \qquad\,\,\, \mbox{in} \,\, \, \R^{N+1}_+,\\
-k_s \frac{\partial w_{j,2}}{\partial {\nu}}&=-V(z_j^2) w_{j,2}+ f(z_j^2, |w_{j,2}|)w_{j,2} \quad \quad  \mbox{on} \,\, \R^N \times \{0\},
\end{aligned}
\right.
\end{align*}
and $\|w_{j,2}\|_{1, s} \geq \tau$.

Let $w_{j, \eps}^2:=w_{j, \eps}^1-w_{j,2}(\cdot-z_{j, \eps}^2, \cdot)$.
If $w_{j, \eps}^2\to 0$ in $X^{1, s}(\R^{N+1}_+)$ as $\eps \to
0^+$, then the proof is completed. Otherwise, we may suppose that $
\lim_{\eps \to 0^+}\|w_{j, \eps}^2\|_{1, s} >0$. Setting
$$
v_{j, \eps}^2:=w_{j, \eps}^2(\cdot +z_{j, \eps}^2, \cdot)=w_{\eps,
j}^1(\cdot + z_{j, \eps}^2, \cdot) -w_{j,2},
$$
we then see that $v_{j, \eps}^2 \wto 0$ in $X^{1, s}(\R^{N+1}_+)$ as $\eps \to 0^+$ and
\begin{align} \label{lieb2}
\begin{split}
\|w_{j, \eps}^2\|_{1, s}^2=\|v_{j, \eps}^2\|_{1, s}^2&=\|w_{j, \eps}^1\|_{1, s}^2- \|w_{j,2}\|_{1, s}^2 +o_{\eps}(1) \\
&=\|w_{j, \eps}\|^2_{1, s}-\|w_{j,1}\|_{1, s}^2-\|w_{j,3}\|_{1, s}^2 +o_{\eps}(1),
\end{split}
\end{align}
where we used \eqref{lieb1}. Repeating the procedure above, we
know that there is a sequence $\{z_{j, \eps}^3\} \subset \R^N$
such that $z_{j, \eps}^3 \to z_{j}^3 \in \Lambda^{\delta_0}$ and
$|z_{j, \eps}^{l}-z_{j, \eps}^{l'}| \to + \infty$ as $\eps \to 0^+$ for any $1 \leq l \neq l' \leq 3$. In addition, $w_{\eps,
j}^2(\cdot+z_{j, \eps}^2, \cdot) \wto w_{j,3}$ in $X^{1, s}(\R^{N+1}_+)$
as $\eps \to 0^+$, $w_{j,3}$ satisfies the equation
\begin{align*}
\left\{
\begin{aligned}
-\mbox{div}(y^{1-2s} \nabla w_{j,3})&=0 \hspace{4.5cm} \ \quad \,\, \, \mbox{in} \, \, \R^{N+1}_+,\\
-k_s \frac{\partial w_{j,3}}{\partial {\nu}}&=-V(z_j^3) w_{j,3}+ f(z_j^3, |w_{j,3}|)w_{j,3} \quad \quad \mbox{on} \,\, \R^N \times \{0\}.
\end{aligned}
\right.
\end{align*}
and $\|w_{j,m}\|_{1, s} \geq \tau$.

By iterating $m$ times, we have that there is a sequence $\{z_{j, \eps}^m\} \subset \R^N$ such that $z_{j, \eps}^m \to z_{j}^m \in \Lambda^{\delta_0}$ and $|z_{j, \eps}^{l}-z_{j, \eps}^{l'}| \to + \infty$ as $\eps \to 0^+$ for any $1 \leq l \neq l' \leq m$. In addition, $w_{j, \eps}^{m-1}(\cdot+z_{j, \eps}^m, \cdot) \wto w_{j,m}$ in $X^{1, s}(\R^{N+1}_+)$ as $\eps \to 0^+$, $w_{j,m}$ satisfies the equation
\begin{align*}
\left\{
\begin{aligned}
-\mbox{div}(y^{1-2s} \nabla w_{j,m})&=0 \hspace{4.5cm} \qquad \quad \, \, \mbox{in} \,\, \R^{N+1}_+,\\
-k_s \frac{\partial w_{j,m}}{\partial {\nu}}&=-V(z_j^m) w_{j,m}+ f(z_j^m, |w_{j,m}|)w_{j,m} \quad  \quad \mbox{on} \,\, \R^N \times \{0\}.
\end{aligned}
\right.
\end{align*}
and $\|w_{j,m}\|_{1, s} \geq \tau$ and. It also holds that
\begin{align} \label{liebm}
\|w_{j, \eps}^m\|_{1, s}^2=\|w_{j, \eps}\|^2_{1, s}-\sum_{l=1}^m\|w_{j,l}\|_{1, s}^2+o_{\eps}(1).
\end{align}
Since $\|w_{j, \eps}\|_{1, s} \leq \eta_k$ for any $1 \leq j \leq k$, see Lemma \ref{bddsol},
and $\|w_{j,l}\|_{1, s} \geq \tau$ for any $1 \leq l \leq m$,
it then follows from  \eqref{liebm} that the iteration has to terminate at some finite index $m$. Thus we have completed the proof.
\end{proof}

Let $\{\epsilon_n\} \subset \R^+$ be such that $\eps_n=o_n(1)$ and $w_{j, \epsilon_n} \in X^{1, s}(\R^{N+1}_+)$ be the solution to \eqref{equ21} with $\eps=\eps_n$ obtained in Theorem \ref{cvnbmiif9f8ufjhfy1}. For simplicity, we shall denote $w_{j, \eps_n}$ by $w_{j, n}$ in the following.

\begin{lem} \label{jdhfggf7fydttdyy}
There exists a constant $C>0$ independent
of $n$ such that $\|u_{j,n}\|_{L^\infty(\R^N)} \leq C$, where
$u_{j,n}:=w_{j,n}(\cdot, 0)$. Moreover, for any $n \in \N$,
$u_{j, n}(x)\rightarrow 0$ as $|x|\rightarrow\infty.$
\end{lem}
\begin{proof}
In light of \cite[Proposition 3.1.1]{DMV}, we know that there exists a constant $C>0$ such that
\begin{align} \label{xmxnbdvf77dyetgdgdd}
\left(\int_{\R^{N+1}_+}y^{1-2s}|w_{j ,n}|^{2\gamma}\, dxdy\right)^{\frac{1}{\gamma}}\leq
C\|w_{j ,n}\|^2_{s},
\end{align}
where $\gamma=1+2/(N-2s).$ Utilizing H\"older's inequality, we get that, for any $z\in\R^N,$
\begin{align} \label{xncbvhhfy66ryee}
\int^1_0\int_{B_1(z)}y^{1-2s}|w_{j ,n}|^2\,dxdy
&=\int^1_0\int_{B_1(z)}y^{\frac{1-2s}{\gamma}}|w_{j ,n}|^2y^{\frac{(1-2s)(\gamma-1)}{\gamma}}\,dxdy\nonumber\\
&\leq\left(\int^1_0\int_{B_1(z)}y^{1-2s}|w_{j ,n}|^{2\gamma}dxdy\right)^{\frac{1}{\gamma}}
\left(\int^1_0\int_{B_1(z)}y^{1-2s} \,dxdy\right)^{\frac{\gamma-1}{\gamma}}\nonumber\\
&=C\left(\int^1_0\int_{B_1(z)}y^{1-2s}|w_{j ,n}|^{2\gamma}dxdy\right)^{\frac{1}{\gamma}}.
\end{align}
It then yields from \eqref{xmxnbdvf77dyetgdgdd} and Lemma \ref{bddsol} that
\begin{align} \label{bddwn}
\int^1_0\int_{B_1(z)}y^{1-2s}|w_{j ,n}|^2\,dxdy  \leq C.
\end{align}
Let $a_{j ,n}(x):=-V_{\eps_n}(x)+f_{\eps_n}(x, |u_{j ,n}(x)|)$ for $x \in \R^N$. By the definition of $f_{\eps_n}$, Lemmas \ref{embedding} and \ref{bddsol}, we see that
\begin{align}\label{cxncbvydhdte9912}
\begin{split}
\|a_{j ,n}\|_{L^{\frac{p}{p-2}}(B_{1}(z))}&\leq
C+C\left(\int_{B_{1}(z)}|f_{\eps_n}(x, |u_{j ,n}|)|^{\frac{p}{p-2}}\,dx\right)^{\frac{p-2}{p}}\\ 
&\leq C+C\left(\int_{B_{1}(z)}|u_{j ,n}|^p\,dx\right)^{\frac{p-2}{p}} \\ 
&\leq C',
\end{split}
\end{align}
where $C'>0$ depends only on $N,$ $p$ and $k.$
Notice that $w_{j ,n}$ satisfies the equation
\begin{align*} 
\left\{
\begin{aligned}
-\mbox{div}(y^{1-2s} \nabla w_{j ,n})&=0  \quad \, \hspace{2cm} \mbox{in} \,\, \R^{N+1}_+,\\
-k_s \frac{\partial w_{j ,n}}{\partial {\nu}}&=a_{j,n}(x)u_{j ,n} \ \ \qquad
\mbox{on} \,\, \R^N \times \{0\}.
\end{aligned}
\right.
\end{align*}
Applying \eqref{bddwn}, \eqref{cxncbvydhdte9912} and \cite[Proposition 2.6]{JLX}, we then derive that there exists a constant $C''>0$ independent of $n$ such that
\begin{align}\label{cxbcvfggftr666ee}
\|w_{j ,n}\|_{L^\infty(Q_{1/2}(z))}\leq
C''\left(\int^1_0\int_{B_1(z)}y^{1-2s}|w_{j ,n}|^2dxdy\right)^{1/2} \leq C,
\end{align}
where $Q_{1/2}(z):=B_{1/2}(z)\times(0,1/2)$. Since $z$ is an arbitrary point in $\R^N$,
then there exists a constant $C>0$ independent of $n$ such that
$\|u_{j ,n}\|_{L^\infty(\R^N)} \leq C$.
In view of Lemma \ref{bddsol}, we have that
\begin{align}\label{cxncbvhhfgyr6fyyf}
\int^1_0\int_{B_1(z)}y^{1-2s}|w_{j ,n}|^2\, dxdy\rightarrow 0 \,\, \mbox{as}\
|z|\rightarrow\infty.
\end{align}
Combining \eqref{cxbcvfggftr666ee} and \eqref{cxncbvhhfgyr6fyyf}
gives that, for any $n \in \N$, $u_{j ,n}(x)\rightarrow 0$ as $|x|\rightarrow\infty.$ Thus we have completed the proof.
\end{proof}

\begin{lem}\label{cmvnvhguufu7ryfggf}
Let $\lambda\geq0$ and $u \in H^s(\R^N)$ be a continuous functions.
Assume $u$ satisfies the inequality
\begin{align*}
(-\Delta )^su+\lambda  u\geq 0 \quad \mbox{in}\ \R^N
\end{align*}
in the sense that, for any $\varphi\in H^s(\R^N)$ with $\varphi\geq 0
$ in $\R^N,$
\begin{align}\label{cnmvn8f8fhyrhff}
\int_{\R^N}\varphi(-\Delta )^su \,dx+\lambda\int_{\R^N}u\varphi\,
dx\geq 0,
\end{align}
then $u\geq 0$ in $\R^N$.
\end{lem}
\begin{proof}
From \cite[Lemma 3.1]{FLS}, we know that
\begin{align*}
\int_{\R^N}u(-\Delta)^su\,dx=\int_{\R^N}|\xi|^{2s} |\mathcal{F} u|^2 \, dx=C_{N, s}\int_{\R^N}\int_{\R^N}\frac{\left(u(x)-u(y)\right)^2}{|x-y|^{N+2s}}\,dxdy,
\end{align*}
where $C_{N,s}>0$ is some constant depending only on $N$ and $s$.
Then we get that
\begin{align}
\int_{\R^N}\varphi(-\Delta)^su\,dx=C_{N,s}\int_{\R^N}\int_{\R^N}\frac{(u(x)-u(y))(\varphi(x)-\varphi(y))}{|x-y|^{N+2s}}\,dxdy.\nonumber
\end{align}
Thus \eqref{cnmvn8f8fhyrhff} can be rewrite as
\begin{align}
C_{N,s}\int_{\R^N}\int_{\R^N}\frac{(u(x)-u(y))(\varphi(x)-\varphi(y))}{|x-y|^{N+2s}}\,dxdy+\lambda\int_{\R^N}u\varphi \,dx\geq 0\nonumber
\end{align}
for any non-negative function $\varphi\in H^s(\R^N)$. Since $u\in
H^s(\R^N)$, by \cite[Theorem 3, Section 5.5.2]{RS}, we have that $u^-:=\max\{-u,0\}\in H^s(\R^N).$
Choosing $\varphi=u^-$ and using the method in the proof of
\cite[Lemma 9]{LL}, we can obtain that $u^-=0,$ This completes the
proof.
\end{proof}

\begin{lem}\label{gjhuuhujguf99difujg}
Let $h \in C(\R^N, \R)$ satisfy that there exists a constant $C>0$ such that
\begin{align} \label{cnvbhyyf000d9dttgd}
 0\leq h(x)\leq \frac{C}{1+|x|^{N+2s}}, \ x\in\R^N.
\end{align}
Let $c\in C(\R^N,\R)$ satisfy that there exists a constant $\lambda>0$ such that
$c(x)\geq \lambda$ for any $x\in\R^N$. Assume $v\in H^s(\R^N)\cap L^\infty(\R^N)$ is nonnegative and
 satisfies the inequality
\begin{align} \label{cnvb88f7fyyyff}
(-\Delta)^s v+c(x) v\leq h.
\end{align}
Then there exists a constant $C'>0$ such that
$$
0\leq v(x)\leq \frac{C'}{1+|x|^{N+2s}}, \ x\in\R^N.
$$
\end{lem}
\begin{proof}
Let $\psi \in C(\R^N, \R)$ be the solution to the equation
\begin{align} \label{equpsi}
(-\Delta)^s\psi+\psi=\phi, 
\end{align}
where $\phi: \R \to \R^N$ is a nonnegative continuous function with $\mbox{supp}\, \phi\subset \R^N \backslash B_2(0)$. From \cite[Lemmas 4.2-4.3]{FQT}, we know that
there are $c_1>0$ and $c_2>0$ such that
\begin{align}\label{cnvbvgft66ftfff}
\frac{c_1}{|x|^{N+2s}}\leq \psi(x)\leq\frac{c_2}{|x|^{N+2s}}\quad
\mbox{for} \ |x|>2.
\end{align}
For $\mu>0$ and $\nu >0$, we define $\psi_{\mu,\nu}(x):=\mu\psi(\nu x)$
for $x \in \R^N$. It follows from \eqref{equpsi} that
$\psi_{\mu,\nu}$ satisfies the equation
\begin{align}\label{cmvnbhyyf7fyfff6}
(-\Delta)^s\psi_{\mu,\nu}+\lambda\psi_{\mu,\nu}=(\lambda-\nu^{2s})\psi_{\mu,\nu}+\mu\nu^{2s}\phi(\nu
x). 
\end{align}
  Taking into account \eqref{cnvbhyyf000d9dttgd} and \eqref{cnvbvgft66ftfff}, we can deduce
that there exist $\mu_*>0$ and $0<\nu_*^{2s}<\lambda$ such that
\begin{align}\label{qcnvb988dudyhhdggfd}
(\lambda-\nu^{2s}_*)\psi_{\mu_*,\nu_*}(x)\geq h(x),
\ x\in\R^N.
\end{align}
From \eqref{cnvb88f7fyyyff}, \eqref{cmvnbhyyf7fyfff6} and
\eqref{qcnvb988dudyhhdggfd} and Lemma \ref{cmvnvhguufu7ryfggf}, we
now get that
$$
v(x)\leq \psi_{\mu_*,\nu_*}(x), \ x\in\R^N.
$$ This along with \eqref{cnvbvgft66ftfff} and the fact
that $v\in L^\infty(\R^N)$ implies the result of this lemma, and
the proof is completed.
\end{proof}


\begin{lem}\label{cnvnbghfyf7yfyftttd}\cite[Theorem 2.3]{A}
Let $u\in L^2(\R^N)$ be such that $(-\Delta)^su\in
L^1_{loc}(\R^N)$. Then it holds that
$$
(-\Delta)^s|u|\leq \mbox{\textnormal{sign}}(u)(-\Delta)^su \quad
\mbox{in}\ (C^\infty_0(\R^N))',
$$
that is
$$
\int_{\R^N}|u|(-\Delta)^s\varphi\,dx\leq\int_{\R^N}\left(\mbox{\textnormal{sign}}(u)(-\Delta)^s u \right) \varphi\,dx
$$
for any $\varphi\in C^\infty_0(\R^N)$ such that $\varphi\geq 0.$
\end{lem}

\begin{lem}\label{mcnvbhhfyfufuu}
Let $w_{j,l} \in X^{1, s}(\R^{N+1}_+)$ satisfy
\eqref{cdhfdhfgyyfyf66ftf1} and $u_{j,l}(x)=w_{j,l}(x, 0)$ for $x \in
\R^N$. Then there exists a constant $C>0$ such that
$$
|u_{j,l}(x)|+|\nabla u_{j,l}(x)|\leq \frac{C}{1+|x|^{N+2s}},\ x\in\R^N.
$$
\end{lem}
\begin{proof}
Since $w_{j,l} \in X^{1, s}(\R^{N+1}_+)$ satisfies \eqref{cdhfdhfgyyfyf66ftf1}, from the s-harmonic extension arguments, then $u_{j,l} \in H^s(\R^N)$ satisfies the equation
\begin{align}\label{cbvbhfyf7yf77fyyy}
(-\Delta)^s u_{j,l}+V(z^l_j) u_{j,l}= f(z^l_j, |u_{j,l}|) u_{j,l}. 
\end{align}
By using a similar way as the proof of Lemma
\ref{jdhfggf7fydttdyy}, it is not hard to deduce that $|u_{j,l}(x)|
\to 0$ as $|x| \to \infty$. As a consequence, we know that there exists $R_1>0$ large enough such that $
f(z^l_j, |u_{j,l}|)\leq a/2$ for $|x|\geq R_1.$ Let $\varrho\in
C^{\infty}(\R^+, [0, 1])$ be a cut-off function such that
$\varrho(t)=0$ for $t \leq R_1$,  $\varrho(t) =1$ for $t \geq
R_1+1$ and $\varrho'(t) \geq 0$ for $t \geq 0$. We now rewrite \eqref{cbvbhfyf7yf77fyyy} as
\begin{align*}
(-\Delta)^s u_{j,l}+(V(z^l_j)-\varrho(|x|)f(z^l_j, |u_{j,l}|)) u_{j,l}=
(1-\varrho(|x|))f(z^l_j, |u_{j,l}|) u_{j,l}.
\end{align*}
In virtue of Lemma \ref{cnvnbghfyf7yfyftttd}, we have that
$$
(-\Delta)^s |u_{j,l}|+(V(z^l_j)-\varrho(|x|)f(z^l_j, |u_{j,l}|)) |u_{j,l}| \leq
(1-\varrho(|x|))f(z^l_j, |u_{j,l}|) |u_{j,l}|.
$$
Making use of Lemma \ref{gjhuuhujguf99difujg}, we then infer that
\begin{align} \label{cnmvn99fd8duyhydgft}
|u_{j,l}(x)|\leq \frac{C}{1+|x|^{N+2s}},\ x\in\R^N.
\end{align}
Since $u_{j,l}$ satisfies \eqref{cbvbhfyf7yf77fyyy}, from \cite[Lemma 4.4]{Cabre}, we derive that $u_{j,l}\in
C^{2,\beta}(\R^N)$ for some $0<\beta<1.$ Let $\varphi_{j,l}(t):=f(z^l_j,|t|)t$ for $t \in \R$, then $\varphi_{j,l} \in C^1(\R, \R)$ and $\varphi'_{j,l}(0)=0.$ For $1 \leq i\leq N$, we define $u_{j, l, i}:=\partial_{x_i} u_{j,l}$. Differentiating both sides of \eqref{cbvbhfyf7yf77fyyy} with respect to $x_i$, we then obtain that
\begin{align*} 
(-\Delta)^su_{j, l, i}+V(z^l_j)u_{j, l, i}=\varphi'_{j,l}(u_{j,l})
u_{j, l, i}.
\end{align*}
Due to $|u_{j,l}(x)| \to 0$ as $|x| \to \infty$, we then get that
$|\varphi'_{j,l}(u_{j,l})(x)| \to 0$ as $|x| \to \infty$. At this point, using the
same arguments as the proof of \eqref{cnmvn99fd8duyhydgft},
we are able to derive that
\begin{align*} 
\left|u_{l, i}(x)\right|\leq \frac{C}{1+|x|^{N+2s}},\
x\in\R^N.
\end{align*}
Thus we have completed the proof.
\end{proof}

\begin{lem}\label{cnvbhgy77f7f6dd}
For $p\geq 2$, there exists $C_p>0$ such that, for any
$a_1, a_2\in\R,$
\begin{align*} 
\Big||a_1+a_2|^{p-2}-|a_1|^{p-2}\Big|\leq
C_p(\tau_p|a_1|^{p-3}|a_2|+|a_2|^{p-2}),
\end{align*}
where
$$
\tau_p=\left\{
\begin{array}
[c]{ll}
0,& \mbox{if}\,\, p-2\leq 1,\\
1, & \mbox{if}\,\, p-2>1.
\end{array}
\right.
$$
\end{lem}
\begin{proof}
If $p-2\leq 1$, then, for any $a_1,a_2\in\R$,
\begin{align*} 
\Big||a_1+a_2|^{p-2}-|a_1|^{p-2}\Big|\leq|a_2|^{p-2}.
\end{align*}
If $p-2>1 $, it then follows from the mean value theorem that, for any $a_1, a_2\in\R$,
\begin{align*}
\Big||a_1+a_2|^{p-2}-|a_1|^{p-2}\Big|&=\Big|(p-2)|a_1+\theta
a_2|^{p-4}\left(a_1+\theta a_2\right)a_2\Big| \\ 
&\leq C_p(|a_1|^{p-3}|a_2|+|a_2|^{p-2}),
\end{align*}
where $0<\theta<1.$ Thus we have completed the proof.
\end{proof}

\begin{lem}\label{2qcnvbhgy77f7f6dd}
There exists $C>0$ such that, for any $a_1,a_2\in\R,$
\begin{align*}
\left|g(|a_1+a_2|)-g(|a_1|) \right|\leq C|a_2|^\sigma,
\end{align*}
where $\sigma=\min\{p-2,1\}.$
\end{lem}
\begin{proof}
If $p-2\geq 1$, then $g$ is a Lipschitz function. It follows that
there exists $C>0$ such that, for any $a_1,a_2\in\R$,
\begin{align}
|g(|a_1+a_2|)-g(|a_1|)|\leq C|a_2|.\nonumber
\end{align}
If $p-2<1,$ by the definition of $g$, then $g'$ is
non-increasing on $(0,+\infty)$. Therefore, for any $a_1>0$ and $a_2>0$,
\begin{align*} 
g(a_1+a_2)&=\int^{a_1+a_2}_0g'(t)\,dt\nonumber\\
&=\int^{a_1}_0g'(t)\, dt +\int^{a_1+a_2}_{a_1}g'(t)\,dt\nonumber\\
&=g(a_1)+\int^{a_2}_{0}g'(t+a_1) \, dt\nonumber\\
&\leq g(a_1)+\int^{a_2}_{0}g'(t) \, dt\nonumber\\
&=g(a_1)+g(a_2).
\end{align*}
It then follows that there exists $C>0$ such that, for any
$a_1,a_2\in\R$,
\begin{align*}
\left |g(|a_1+a_2|)-g(|a_1|)\right|\leq g(|a_2|)\leq |a_2|^{p-2},
\end{align*}
where we used the fact that $g(t) \leq t^{p-2}$ for any $t \geq 0$. Thus the proof is completed.
\end{proof}

Let $\{\epsilon_n\} \subset \R^+$ be such that $\eps_n=o_n(1)$. Let $w_{j, \eps_n} \in X^{1, s}(\R^{N+1}_+)$ be the solution to \eqref{equ21} with $\eps=\eps_n$ obtained in Theorem \ref{cvnbmiif9f8ufjhfy1}. Let $w_{j,l}  \in X^{1, s}(\R^{N+1}_+)$ be the solution to \eqref{cdhfdhfgyyfyf66ftf1} obtained in Lemma \ref{cnvbhhfyf6yr66}. Define
$$
R_{j, n}:=w_{j, \eps_n} -\sum_{l=1}^{m_j} w_{j,l}(\cdot-z_{j, \eps_n}^l, \cdot),
$$
and
$$
r_{j, n}:=R_{j, n}(\cdot,0), \quad u_{j, \eps_n}:=w_{j,\eps_n}(\cdot,0), \quad u_{j,l}:=w_{j,l}(\cdot, 0),
$$
where $m_j \in \N$ is given in Lemma \ref{cnvbhhfyf6yr66}. For simplicity, we shall denote $R_{j, n}$, $r_{j, n}$, $z^l_{j, \eps_n}$, $w_{j, \eps_n}$, $u_{j, \eps_n}$ and $u_{j,l}$ by $R_n$, $r_n$, $z^l_n$, $w_n$, $u_n$ and $u_l$, respectively.

\begin{lem}\label{cxnvbvhhfytfg66}
There holds that $ \lim_{n\rightarrow\infty}\|r_n\|_{L^\infty(\R^N)}=0.$
\end{lem}
\begin{proof}
By using the s-harmonic extension arguments, it is straightforward to see that $u_n \in H^s(\R^N)$ satisfies the equation
$$
(-\Delta)^s u_n + V_{\eps_n}(x) u_n=f_{\eps_n}(x, |u_n|) u_n.
$$
Furthermore, $u_{l} \in H^s(\R^N)$ satisfies the equation
\begin{align*} 
(-\Delta)^s u_{l}+V(z_j^l)u_{l}=f(z_j^l ,|u_{l}|)u_{l}.
\end{align*}
In view of the definition of $r_n$, we then get that $r_n$ satisfies the equation
\begin{align}\label{cnvb88f876dydd}
(-\Delta)^sr_n+(V_{\eps_n}(x)-f_{\eps_n}(x,|u_n|))r_n=h_n,
\end{align}
where
\begin{align*}
h_n(x)&=f_{\eps_n}(x,|u_n|)u_n-f_{\eps_n}(x,|u_n|)r_n-\sum^{m_j}_{l=1}f_{\epsilon_n}(x,
|u_{l}(x-z_{n}^l)|)u_{l}(x-z_{n}^l)\nonumber\\
&\quad+\sum^{m_j}_{l=1}(f_{\epsilon_n}(x,
|u_{l}(x-z_{n}^l)|)-f(z^l_j,
|u_{l}(x-z_{n}^l)|))u_{l}(x-z_{n}^l)\nonumber\\
&\quad-\sum^{m_j}_{l=1}(V_{\eps_n}(x)-V(z^l_j))u_{l}(x-z_{n}^l)\nonumber\\
&:=I^1_n(x)+I^2_n(x)+I^3_n(x).
\end{align*}
Let us now estimate the term $I^1_n(x)$ for $x \in \R^N$. From the definitions of $r_n$ and $f_{\eps_n}$, Lemmas
\ref{cnvbhgy77f7f6dd} and \ref{2qcnvbhgy77f7f6dd}, we deduce that
there exists $C>0$ such that
\begin{align}
|I^1_n(x)|&=\left|f_{\eps_n}(x,|u_n|)u_n-f_{\eps_n}(x,|u_n|)r_n-\sum^{m_j}_{l=1}f_{\epsilon_n}(x,
|u_{l}(x-z_{n}^l)|)u_{l}(x-z_{n}^l)\right|\nonumber\\
&=\left|\sum^{m_j}_{l=1}f_{\epsilon_n}(x,|u_n|)u_{l}(x-z_{n}^l)-\sum^{m_j}_{l=1}f_{\epsilon_n}(x,
|u_{l}(x-z_{n}^l)|)u_{l}(x-z_{n}^l)\right|\nonumber\\
&\leq\sum^{m_j}_{l=1}\left|f_{\epsilon_n}(x,|u_n|)-f_{\epsilon_n}(x,
|u_{l}(x-z_{n}^l)|)\right| \left|
u_{l}(x-z_{n}^l)\right|\nonumber\\
&\leq \sum^{m_j}_{l=1}\left||u_n|^{p-2}-
|u_{l}(x-z_{n}^l)|^{p-2}\right|  \left|
u_{l}(x-z_{n}^l)\right|\nonumber\\
&\quad+ \sum^{m_j}_{l=1}\left|g(|u_n|)-
g(|u_{l}(x-z_{n}^l)|)\right| \left|
u_{l}(x-z_{n}^l)\right|\nonumber\\
&\leq\sum^{m_j}_{l=1}\left(\tau_p \left|
u_{l}(x-z_{n}^l) \right|^{p-3}\left| r_n +\sum_{l'\neq l}^{m_j} u_{l'}(x-z_{n}^{l'})\right|+
\left|r_n+\sum_{l'\neq l}^{m_j}u_{l'}(x-z_{n}^{l'})\right|^{p-2}\right)\left|
u_{l}(x-z_{n}^l)\right|\nonumber\\
&\quad+ \sum^{m_j}_{l=1}\left|r_n+\sum_{l'\neq
l}^{m_j}u_{l'}(x-z_{n}^{l'})\right|^\sigma \left|
u_{l}(x-z_{n}^l)\right|.\nonumber
\end{align}
It then follows that
\begin{align*}
|I^1_n(x)|&\leq
C\tau_p\left(\sum^{m_j}_{l=1}\left|u_{l}(x-z_{n}^l)\right|^{p-2}\right)|r_n|
+C\left(\sum^{m_j}_{l=1}\left|u_{l}(x-z_{n}^l)\right|\right)|r_n|^{p-2}
\nonumber\\
&\quad +C\sum_{l=1}^{m_j} \sum_{l'\neq l}^{m_j}\left( \tau_p\left|u_{l}(x-z_{n}^{l})\right|^{p-2}\left|u_{l'}(x-z_{n}^{l'})\right|+\left|u_{l'}(x-z_{n}^{l'})\right|^{p-2} \left|u_{l}(x-z_{n}^{l})\right| \right)\nonumber\\
 &\quad+C \left(\sum^{m_j}_{l=1}\left|u_{l}(x-z_{n}^l)\right|\right)|r_n|^{\sigma}
 +C \sum_{l=1}^{m_j}  \sum_{l'\neq l}^{m_j}\left|u_{l'}(x-z_{n}^{l'})\right|^{\sigma} \left|u_{l}(x-z_{n}^{l})\right|.
\end{align*}
By the definition of $r_n$, \eqref{wnj} and Lemma \ref{embedding}, we have that
\begin{align*} 
\lim_{n\rightarrow\infty}\|r_n\|_{L^q(\R^N)}=0 \quad \mbox{for}\,\, 2\leq q\leq 2^*_s.
\end{align*}
From Lemma \ref{mcnvbhhfyfufuu} and the fact that $|z_n^l-z_n^{l'}| \to \infty$
as $n \to \infty$ for $l \neq l'$, see Lemma \ref{cnvbhhfyf6yr66},
we get that, for any $x_0\in\R^N$,
\begin{align}\label{cnvbfgtdgdttd}
\lim_{n\rightarrow\infty}\int_{B_1(x_0)}|I_n^1|^{\frac{p}{p-2}}\,dx=0.
\end{align}
Applying Lemma \ref{cnvbhhfyf6yr66}, we can also conclude that
\begin{align}\label{12cnvbfgtdgdttd}
\lim_{n\rightarrow\infty}\int_{B_1(x_0)}|I^2_n|^{\frac{p}{p-2}}\,dx=0,\quad
\lim_{n\rightarrow\infty}\int_{B_1(x_0)}|I^3_n|^{\frac{p}{p-2}}\,dx=0.
\end{align}
By using a similar way as the proof of \eqref{xncbvhhfy66ryee}, it is not hard to show that
$$
\int^{1}_0\int_{B_1(x_0)}y^{1-2s}|R_n|^2\,dxdy \leq C_N\left(\int^1_0\int_{B_1(z)}y^{1-2s}|R_n|^{2\gamma}dxdy\right)^{\frac{1}{\gamma}}.
$$
Taking into account \cite[Proposition 3.1.1]{DMV} and \eqref{wnj}, we then obtain that
\begin{align}\label{xbcvfggftr66ett}
\lim_{n \to \infty}\int^{1}_0\int_{B_1(x_0)}y^{1-2s}|R_n|^2\,dxdy=0.
\end{align}
Note that $2<p<2^*_s$, then $p/(p-2)>N/2s$. Since $r_n$ satisfies \eqref{cnvb88f876dydd}, by using \cite[Proposition
2.6]{JLX},  $\eqref{cnvbfgtdgdttd}$, $\eqref{12cnvbfgtdgdttd}$ and $\eqref{xbcvfggftr66ett}$, we are
now able to deduce that
\begin{align*} 
\lim_{n\rightarrow\infty}\|r_n\|_{L^\infty(B_{1/2}(x_0))}=0.
\end{align*}
This readily yields the result of this lemma, because $x_0$ is an arbitrary point in $\R^N$. Thus the proof is completed.
\end{proof}

\begin{lem}\label{cnvbhhfyf7yfyddsd}
There exists $C>0$ such that, for any $n \in \N$,
\begin{align*} 
|r_n(x)|\leq\sum^{m_j}_{l=1}\frac{C}{1+|x-z^l_{n}|^{N+2s}},\
x\in\R^N.
\end{align*}
\end{lem}
\begin{proof}
In view of the definition of $r_n$, we know that $r_n$ satisfies the equation
\begin{align}\label{qawcnvb88f876dydd}
(-\Delta)^sr_n+V_{\eps_n}(x)r_n=\psi_n,
\end{align}
where
\begin{align*}
\psi_n(x):&=f_{\eps_n}(x,|u_n|)u_n -\sum^{m_j}_{l=1}f(z^l_j,
|u_{l}(x-z_{n}^l)|)u_{l}(x-z_{n}^l)\nonumber\\
&\quad-\sum^{m_j}_{l=1}(V_{\eps_n}(x)-V(z^l_j))u_{l}(x-z_{n}^l).
\end{align*}
By the definitions of $f_{\eps_n}$ and $r_n$, we see that there exists $C>0$ such that
\begin{align*}
|\psi_n(x)|\leq
C|r_n|^{p-1}+C \sum^{m_j}_{l=1}\left(|u_{l}(x-z_{n}^l)+|u_{l}(x-z_{n}^l)|^{p-1}\right).
\end{align*}
Hence Lemma \ref{mcnvbhhfyfufuu} gives that there exists $C>0$ such that
\begin{align}\label{2cnmvcnghhgyfhdydd}
|\psi_n(x)|\leq
C|r_n|^{p-1}+\sum^{m_j}_{l=1}\frac{C}{1+|x-z^l_{n}|^{N+2s}}.
\end{align}
Taking advantage of the assumption that $V(x) \geq a$ for any $x \in \R^N$, \eqref{qawcnvb88f876dydd}, \eqref{2cnmvcnghhgyfhdydd} and Lemma \ref{cnvnbghfyf7yfyftttd}, we then have that
\begin{align*} 
(-\Delta)^s|r_n|+(a-C|r_n|^{p-2})|r_n|\leq
\sum^{m_j}_{l=1}\frac{C}{1+|x-z^l_{n}|^{N+2s}}.
\end{align*}
By Lemma \ref{cxnvbvhhfytfg66}, we now get that, for any $n \in \N$ large enough,
\begin{align}\label{23xbcbfggft6fyfyyy}
(-\Delta)^s|r_n|+\frac{a}{2}|r_n|\leq
\sum^{m_j}_{l=1}\frac{C}{1+|x-z^l_{n}|^{N+2s}}. 
\end{align}
Let $\omega$ be the unique solution to the equation
\begin{align}\label{23hh2ft6fyfyyy}
(-\Delta)^s\omega+\frac{a}{2}\omega=\frac{C}{1+|x|^{N+2s}}. 
\end{align}
Using \eqref{23xbcbfggft6fyfyyy},  \eqref{23hh2ft6fyfyyy} and
Lemma \ref{cmvnvhguufu7ryfggf}, we have that
\begin{align} \label{xnvcbfggfyf7r6ttet} 
|r_n(x)|\leq \sum^{m_j}_{l=1}|\omega(x-z_{n}^l)|,\ x\in\R^N.
\end{align}
On the other hand, in light of Lemma \ref{gjhuuhujguf99difujg}, we obtain that there exists $C>0$ such that
\begin{align*}
|\omega(x)|\leq \frac{C}{1+|x|^{N+2s}},\ x\in\R^N.
\end{align*}
This together with \eqref{xnvcbfggfyf7r6ttet} indicates the result of this lemma, and the proof is completed.
\end{proof}

As a consequence of the definition of $r_n$, Lemmas \ref{mcnvbhhfyfufuu} and \ref{cnvbhhfyf7yfyddsd},
we then have the following lemma.

\begin{lem}\label{11123cnvbhhfyf7yfyddsd}
There exists $C>0$ such that, for any $n \in \N$,
\begin{align*} 
|u_n(x)|\leq\sum^{m_j}_{l=1}\frac{C}{1+|x-z^l_{n}|^{N+2s}},\
x\in\R^N.
\end{align*}
\end{lem}

\begin{lem}\label{irffhgyytyyenn}
There exists $C>0$ such that, for any $n \in \N$,
\begin{align*}
|\nabla r_n(x)|\leq\sum^{m_j}_{l=1}\frac{C}{1+|x-z^l_{n}|^{N+2s}},\
x\in\R^N.
\end{align*}
\end{lem}
\begin{proof}
According to Lemma \ref{jdhfggf7fydttdyy}, we know that there exists $C>0$
independent of $n$ such that $\|u_n\|_{ L^\infty(\R^N)}\leq C$. In
virtue of \cite[Proposition 4.5]{Radulescu}, we then derive that
there exist $\alpha'\in (0,1)$ and $C'>0$ independent of $n$ such
that $u_n\in C^{\alpha'}(\R^N)$ and $\|u_n\|_{C^{\alpha'}{(\R^N)}}
\leq C'$. Since $f$ and $V$ are $C^1$ functions, by \cite[Propositions 2.8-2.9]{S},
we then obtain that there exists $\alpha''\in (0,1)$ and  $C''>0$ independent of $n$ such that
$u_n\in C^{1,\alpha''}(\R^N)$ and
$\|u_n\|_{C^{1,\alpha''}(\R^N)}\leq C''$.
For simplicity, we define
\begin{align*}
u_{n,i}:=\partial_{x_i} u_n,\quad  r_{n,i}:=\partial_{x_i} r_n,
\quad u_{l,i}:=\partial_{x_i} u_{j,l},
\end{align*}
and
\begin{align*}
V_i:=\partial_{x_i} V, \quad \chi_i=\partial_{x_i} \chi \quad \mbox{for}\,\, 1 \leq i\leq N.
\end{align*}
By using the definition of $r_n$, we then get that
\begin{align} \label{cmvnvbfhfyhfhf}
u_{n,i}=r_{n,i}+\sum^{m_j}_{i=1}u_{l,i}(\cdot-z^l_n).
\end{align}
As a consequence of Lemma \ref{mcnvbhhfyfufuu}, we then have
that there exists $C>0$ independent of $n$ such that
$\|r_{n,i}\|_{C^{1,\alpha''}(\R^N)}\leq C$.
Since $u_n$ satisfies the equation
\begin{align*} 
(-\Delta)^su_n+V_{\eps_n}( x)u_n=f_{\epsilon_n}(x,|u_n|)u_n 
\end{align*}
and $u_{l}$ satisfies the equation
\begin{align*}
(-\Delta)^s u_{l}+V(z_j^l)u_{l}=f(z_j^l ,|u_{l}|)u_{l},
\end{align*}
then
\begin{align*}
\begin{split}
(-\Delta)^su_{n,i}+V_{\eps_n}(x)u_{n,i}= & -\epsilon_nV_i(\epsilon_n
x)u_n+(p-1)(1-\chi(\epsilon_n x))|u_n|^{p-2}u_{n,i} \\
&-\eps_n\chi_i(\eps_n x) |u_n|^{p-2} u_n+\chi(\epsilon_n x)H'(|u_n|)u_{n,i} \\
&+\eps_n\chi_i(\eps_n x) g(|u_n|) u_n
\end{split}
\end{align*}
and
\begin{align*}
\begin{split}
(-\Delta)^s u_{l,i}(x-z^l_n)+V(z_{j}^l)u_{l,i}(x-z^l_n) &=  (p-1)\sum^{m_j}_{l=1}(1-\chi(z^l_j))|u_{l}(\cdot-z^l_n)|^{p-2}u_{l,i}(x-z^l_n) \\
& \quad +\sum^{m_j}_{l=1}\chi(z^l_j)H'(|u_{l}(x-z^l_n)|)u_{l,i}(x-z^l_n),
\end{split}
\end{align*}
where $H(t)=g(t)t$ for $t \in \R$.
In view of \eqref{cmvnvbfhfyhfhf}, we then deduce that
\begin{align} \label{cnvbuuf7fyfyyyfdd}
(-\Delta)^sr_{n,i}+V_{\eps_n}(x)r_{n,i} &=-\epsilon_nV_i(\epsilon_n x)u_n+(p-1)(1-\chi(\epsilon_n x))|u_n|^{p-2}u_{n,i}
+\chi(\epsilon_n x)H'(|u_n|)u_{n,i}\nonumber\\
&\quad-(p-1)\sum^{m_j}_{l=1}(1-\chi(z^l_j))|u_{l}(\cdot-z^l_n)|^{p-2}u_{l,i}(x-z^l_n) \nonumber \\
&\quad -\sum^{m_j}_{l=1}\chi(z^l_j)H'(|u_{l}(x-z^l_n)|)u_{l,i}(x-z^l_n)\\
&\quad-\sum^{m_j}_{l=1}(V_{\eps_n}(x)-V(z^l_j))u_{l,i}(x-z_{n}^l)-\eps_n\chi_i(\eps_n x) \left( |u_n|^{p-2} u_n -g(|u_n|) u_n \right) \nonumber\\
&=-\epsilon_nV_i(\epsilon_n x)u_n+T_n(x)r_{n,i}+\mathcal{L}_n(x),  \nonumber
\end{align}
where
$$
T_n(x):=(p-1)(1-\chi(\epsilon_n
x))|u_n|^{p-2}+\chi(\epsilon_n x)H'(|u_n|)
$$
and
\begin{align*}
\mathcal{L}_n(x)&:=(p-1)(1-\chi(\epsilon_n
x))\bigg(\sum^{m_j}_{l=1}|u_n|^{p-2}u_{l,i}(x-z^l_n)-\sum^{m_j}_{l=1}|u_{l}(x-z^l_n)|^{p-2}u_{l,i}(x-z^l_n)\bigg)
\nonumber\\
&\quad+\chi(\epsilon_n
x)\Big(\sum^{m_j}_{l=1}H'(|u_n|)u_{l,i}(x-z^l_n)-\sum^{m_j}_{l=1}H'(|u_{l}(x-z^l_n)|)u_{l,i}(x-z^l_n)\Big)
\nonumber\\ &\quad-(p-1)\sum^{m_j}_{l=1}(\chi(\epsilon_n
x)-\chi(z^l_j))|u_{l}(x-z^l_n)|^{p-2}u_{l,i}(x-z^l_n)\nonumber\\
&\quad+\sum^{m_j}_{l=1}(\chi(\epsilon_n
x)-\chi(z^l_j))H'(|u_{l}(x-z^l_n)|)u_{l,i}(x-z^l_n)\nonumber\\
&\quad-\sum^{m_j}_{l=1}(V(\epsilon_n x)-V(z^l_j))u_{l,i}(x-z_{n}^l)-\eps_n\chi_i(\eps_n x) \left( |u_n|^{p-2} u_n -g(|u_n|) u_n \right).
\end{align*}
Using Young's inequality and the fact that $\|r_{n,i}\|_{C^{1,\alpha''}(\R^N)}\leq C$, we see that, for any $\eta>0$,
\begin{align*}
|T_n(x)r_{n,i}|&\leq \frac{p-2}{p-1}\eta^{-\frac{p-1}{p-2}}\left|(p-1)(1-\chi(\epsilon_n
x))|u_n|^{p-2}+\chi(\epsilon_n
x)H'(|u_n|)\right|^{\frac{p-1}{p-2}}+\frac{1}{p-1}\eta^{p-1}|r_{n,i}|^{p-1} \\ 
&\leq C\eta^{-\frac{p-1}{p-2}}|u_n|^{p-1}+\eta^{p-1}C_1^{p-2}|r_{n,i}|.
\end{align*}
Choosing $\eta>0$ such that $\eta^{p-1}C_1^{p-2}=a/2$ and using
Lemma \ref{11123cnvbhhfyf7yfyddsd},
we then get that there exists $C>0$ such that
\begin{align}\label{2xcnvcbvhfhgftr66ft}
|T_n(x)r_{n,i}|\leq
\sum^{m_j}_{l=1}\frac{C}{1+|x-z^l_n|^{N+2s}}+\frac{a}{2}|r_{n,i}|
\end{align}
On the other hand, from Lemmas  \ref{mcnvbhhfyfufuu} and \ref{11123cnvbhhfyf7yfyddsd}, we derive that there exists $C>0$ such that
\begin{align}\label{4xcnvcbvhfhgftr66ft}
\left|-\epsilon_nV_i(\epsilon_n x)u_n(x) \right|\leq
\sum^{m_j}_{l=1}\frac{C}{1+|x-z^l_n|^{N+2s}}
\end{align}
and
\begin{align}\label{6xcnvcbvhfhgftr66ft}
|\mathcal{L}_n(x)|\leq
\sum^{m_j}_{l=1}\frac{C}{1+|x-z^l_n|^{N+2s}}
\end{align}
Making use of \eqref{cnvbuuf7fyfyyyfdd}-\eqref{6xcnvcbvhfhgftr66ft} and Lemma \ref{cnvnbghfyf7yfyftttd}, we now get that
\begin{align*}
(-\Delta)^s|r_{n,i}|+\frac{a}{2}|r_{n,i}|\leq
\sum^{m_j}_{l=1}\frac{C}{1+|x-z^l_{n}|^{N+2s}}. 
\end{align*}
At this point, reasoning as the proof of Lemma \ref{cnvbhhfyf7yfyddsd}, we are able to get the desired result, and the proof is completed.
\end{proof}

Taking advantage of \eqref{cmvnvbfhfyhfhf}, Lemma \ref{mcnvbhhfyfufuu} and Lemma \ref{irffhgyytyyenn}, we
then obtain the following lemma.

\begin{lem}\label{immm9ibhhfyf7yfyddsd}
There exists $C>0$ such that, for any $n \in \N$,
\begin{align}\label{idaaaaz6f8gutyygh}
|\nabla u_n(x)|\leq\sum^{m_j}_{l=1}\frac{C}{1+|x-z^l_{n}|^{N+2s}},\
x\in\R^N.
\end{align}
\end{lem}

\begin{lem}\label{oooldhfyyyf6trrrr}
There exists $C>0$ such that, for any $n \in \N$,
\begin{align} \label{idaaaaziiii87utyygh}
|\nabla_x
w_n(\xi)|\leq\sum^{m_j}_{l=1}\frac{C}{1+|\xi-\xi^l_n|^{N}},\
\xi\in\overline{\R^{N+1}_+},
\end{align}
where $\xi=(x,y)$ and $\xi^l_n=(z^l_{n},0).$
\end{lem}
\begin{proof}
For $1\leq i\leq N$, we define $w_{n,i}:=\partial_{x_i}w_n$. Then we see that
\begin{align}\label{qqawequ3}
\left\{
\begin{aligned}
-\mbox{div}(y^{1-2s} \nabla w_{n,i})&=0\quad\quad  \,\mbox{in} \,\,\R^{N+1}_+,\\
w_{n,i}(x, 0)&=u_{n,i} \quad \mbox{on} \,\, \R^N \times \{0\}.
\end{aligned}
\right.
\end{align}
Let $v$ be the solution to the equation
\begin{align}\label{1q55qraddqqawequ31}
\left\{
\begin{aligned}
-\mbox{div}(y^{1-2s} \nabla v)&=0 \hspace{2cm} \quad  \mbox{in} \,\, \R^{N+1}_+,\\
v(x,0)&=\frac{1}{1+|x|^{N+2s}} \quad \ \,\,\, \mbox{on} \,\,
\R^N \times \{0\}.
\end{aligned}
\right.
\end{align}
From \eqref{qqawequ3}, \eqref{1q55qraddqqawequ31},  Lemma
\ref{immm9ibhhfyf7yfyddsd} and the maximum principle, we then get that
\begin{align}\label{cnvbfggfy88w4se1}
|w_{n,i}(\xi)|\leq C\sum^{m_j}_{l=1}v(\xi-\xi^l_n),\
\xi\in\overline{\R^{N+1}_+}.
\end{align}
By \eqref{1q55qraddqqawequ31} and \cite[Remark 3.8]{Cabre}, we
have that
\begin{align}\label{bxvcttfgd5tgsg0001aaz1}
v(\xi)=\int_{\R^N}  \frac{1}{1+|\zeta|^{N+2s}}P_y^s (x-\zeta) \,
d\zeta,
\end{align}
where $P_y^s(x)$ is defined by \eqref{defpys} for $x \in \R^N$. If
$|x| \leq y$, then $|\xi| \leq \sqrt{2} y $.  It follows that
\begin{align*}
\left|v(\xi)\right| & = c_{n, s} \int_{\R^N} \frac{y^{2s}}{\left (|x-\zeta|^2 + y^{2}\right)^{\frac{N+2s}{2}}} \frac{1}{1+|\zeta|^{N+2s}} \, d\zeta \\
& \leq \frac{c_{n, s}}{|y|^{N}} \int_{\R^N} \frac{1}{1+|\zeta|^{N+2s}} \, d\zeta  \\
&\leq  \frac{C}{|\xi|^{N}}.
\end{align*}
This along with the fact that $v\in L^\infty(\mathbb{R}^N)$ (see
\cite[Corollary 3.5]{Cabre}) implies that
$$
|v(\xi)| \leq\frac{C}{1+|\xi|^{N}}\quad \mbox{if}\ |x| \leq y.
$$
If  $|x| \geq y$, then $|\xi| \leq \sqrt{2} |x| $. By using \eqref{bxvcttfgd5tgsg0001aaz1}, we know that
\begin{align}\label{xncbvufd6ryfgdvxccc1}
\begin{split}
|v(\xi)| & \leq\left|\int_{|\zeta-x| >
|x|/2} \frac{1}{1+|\zeta|^{N+2s}}P_y^s(x-\zeta)  \,
d\zeta\right| \\
& \quad +\left|\int_{|\zeta-x| \leq |x|/2}
\frac{1}{1+|\zeta|^{N+2s}}P_y^s(x-\zeta)  \, d\zeta\right|.
\end{split}
\end{align}
The first term in the right hand side of \eqref{xncbvufd6ryfgdvxccc1} can be estimated as
\begin{align}
\begin{split}
&\left|\int_{|\zeta-x| > |x|/2}
\frac{1}{1+|\zeta|^{N+2s}}P_y^s(x-\zeta)  \,
d\zeta\right|\\ 
&= c_{N,s}\int_{|\zeta-x| > |x|/2} \frac{y^{2s}}{\left
(|x-\zeta|^2 + y^{2}\right)^{\frac{N+2s}{2}}}
\frac{1}{1+|\zeta|^{N+2s}} \,
d\zeta\\ 
&\leq\frac{ c_{N,s}}{|x/2|^{N}}\int_{|\zeta-x| > |x|/2}
\frac{1}{1 +|\zeta|^{N+2s}}\,d\zeta\\ 
&\leq\frac{C}{|x|^{N}}  \leq \frac{C}{|\xi|^{N}}.
\end{split}
\end{align}
Since $|\zeta-x| \leq |x|/2$ implies that $|\zeta| \geq |x|/2$,
then the second term in the right hand side of
\eqref{xncbvufd6ryfgdvxccc1} can be estimated as
\begin{align}\label{cmvbbghyt777rfedrf1}
\begin{split}
&\left|\int_{|\zeta-x| \leq |x|/2}
\frac{1}{1+|\zeta|^{N+2s}}P_y^s(x-\zeta)  \,
d\zeta\right|\\ 
&= c_{N,s}\int_{|\zeta-x| \leq |x|/2} \frac{y^{2s}}{\left
(|x-\zeta|^2 + y^{2}\right)^{\frac{N+2s}{2}}}
\frac{1}{1+|\zeta|^{N+2s}} \,
d\zeta \\
&\leq  c_{N,s}\int_{|\zeta|\geq |x|/2} \frac{y^{2s}}{\left
(|x-\zeta|^2 + y^{2}\right)^{\frac{N+2s}{2}}}
\frac{1}{1+|\zeta|^{N+2s}} \,
d\zeta \\ 
&\leq\frac{ c_{N,s}}{1+|x/2|^{N+2s}}\int_{|\zeta|\geq |x|/2}
\frac{y^{2s}}{\left (|x-\zeta|^2 +
y^{2}\right)^{\frac{N+2s}{2}}}\,d\zeta\\ 
&\leq\frac{ c_{N,s}}{1+|x/2|^{N+2s}}\int_{\R^N}\frac{y^{2s}}{\left
(|x-\zeta|^2 + y^{2}\right)^{\frac{N+2s}{2}}}\,d\zeta \\
&=\frac{ 1}{1+|x/2|^{N+2s}}  \leq \frac{C}{1+|\xi|^{N+2s}}.
\end{split}
\end{align}
Combining \eqref{xncbvufd6ryfgdvxccc1}-\eqref{cmvbbghyt777rfedrf1}
yields that$$ |v(\xi)| \leq\frac{C}{|\xi|^{N}}\quad \mbox{if}\ |x|
\geq y.
$$
This along with the fact that $v\in L^\infty(\mathbb{R}^N)$
implies that
$$
|v(\xi)| \leq\frac{C}{1+|\xi|^{N}}\quad \mbox{if}\ |x| \geq y.
$$
Hence
\begin{align}\label{mmmmz9999agtrrr4}
|v(\xi)| \leq\frac{C}{1+|\xi|^{N}},\quad
\xi\in\overline{\R^{N+1}_+}. \end{align}
This along with
\eqref{cnvbfggfy88w4se1} implies \eqref{idaaaaziiii87utyygh}. Thus
the proof is completed.
\end{proof}

\begin{lem}\label{ooo9dudyttdtrrrr}
There exists $C>0$ such that, for any $n \in \N$,
\begin{align}\label{2idaaaaziiii87utyygh}
\left|y^{1-2s}\partial_y
w_n(\xi)\right|\leq\sum^{m_j}_{l=1}\frac{C}{1+|\xi-\xi^l_n|^{N}},\
\xi\in\overline{\R^{N+1}_+}.
\end{align}
\end{lem}
\begin{proof}
Define $v_n(\xi):=y^{1-2s}\partial_y w_n(\xi).$  By the dual principle
(see \cite[Section 2.3]{CS} or \cite[Section 3.3]{Cabre}), we know
that $v_n$ satisfies the equation
\begin{align}\label{a55ras23qqawequ3} \left\{
\begin{aligned}
-\mbox{div}(y^{1-2s} \nabla v_n)&=0 \hspace{4cm} \qquad \mbox{in} \,\, \R^{N+1}_+,\\
v_n(x,0)&=-V_{\eps_n}(x)u_n+f_{\epsilon_n}(x, |u_n|)u_n \quad \quad
\mbox{on} \,\, \R^N \times \{0\}.
\end{aligned}
\right.
\end{align}
It follows from Lemma \ref{11123cnvbhhfyf7yfyddsd} and the definition of $f_{\eps_n}$ that there exists $C>0$
independent of $n$ such that
\begin{align}\label{xncbvggftr6ftyyyd}
\left|-V(\epsilon_n x)u_n+f_{\epsilon_n}(x,|u_n|)u_n \right|\leq\sum^{m_j}_{l=1}\frac{C}{1+|x-z^l_n|^{N+2s}},\ x\in
\R^N.
\end{align}
At this point, arguing as the proof of Lemma \ref{oooldhfyyyf6trrrr}, we are able to obtain the result of this lemma. Thus the proof is completed.
\end{proof}

\subsection{Proof of Theorem \ref{wejgh77rtff111}} Relying on the arguments above, we are now ready to prove Theorem \ref{wejgh77rtff111}. To this end, let us first introduce some notations.
For $r>0$, $x=(x_1,\cdots,x_N)\in\R^N$ and
$\xi=(x,y)=(x_1,\cdots,x_N,y)\in\R^{N+1}$, we define
\begin{align*} 
{Q}_r(x):=[x_1-r,x_1+r]\times\cdots\times[x_N-r,x_N+r],
\end{align*}
\begin{align*} 
\mathcal{Q}_r(\xi):=[x_1-r,x_1+r]\times\cdots\times[x_N-r,x_N+r]\times[y-r,y+r]
\end{align*}
and
\begin{align*} 
\mathcal{Q}^+_r(\xi):=\mathcal{Q}_r(\xi)\cap\overline{\R^{N+1}_+}.
\end{align*}
We denote by $\partial Q_r(x)$ the boundary of $Q_r(x)$ in $\R^N$ and $\partial\mathcal{Q}^+_r(\xi)$ the boundary of $\mathcal{Q}^+_r(\xi)$ in $\R^{N+1}$. Let
$$
\partial^+\mathcal{Q}^+_r(\xi):=\partial\mathcal{Q}^+_r(\xi)\cap\R^{N+1}_+, \quad \partial^*\mathcal{Q}^+_r(\xi):=\{(x,0) \ | \ (x,0)\in \partial\mathcal{Q}^+_r(\xi)\},
$$
then
\begin{align*} 
\partial\mathcal{Q}^+_r(\xi)=\partial^+\mathcal{Q}^+_r(\xi)\cup\partial^*\mathcal{Q}^+_r(\xi).
\end{align*}
For $\xi\in\R^{N+1}$, we define
$$
\varrho(\xi):=\inf\left\{r>0 \Big|\ \frac{\xi}{r}\in
\mathcal{Q}_1(0)\right\}.
$$
It is easy to check that $\varrho$ is a norm in $\R^{N+1}$ and there exists $C>0$ such that
\begin{align}\label{cnvbvhfyf6ryfyyyf}
C^{-1}|\xi|\leq\varrho(\xi)\leq C|\xi|,\ \xi\in\R^{N+1}.
\end{align}
Furthermore, there holds that
\begin{align}\label{cnvbghhgy99r899r9t}
\left\{\zeta\in\R^{N+1}\ |\ \varrho(\zeta-\xi)\leq
r\right\}=\mathcal{Q}_r(\xi).
\end{align}

Let $\{\eps_n\} \subset \R^+$ be such that $\eps_n=o_n(1)$.
Up to a subsequence, we assume that, for any $1 \leq l \leq m_j$, $\lim_{n\rightarrow
\infty}\epsilon_nz^l_{j,\epsilon_n}$ exists. As a result of Lemma \ref{cnvbhhfyf6yr66}, we have that
\begin{align}\label{cvnhg88guuyytgd}
\left\{z^*_1,\cdots,z^*_{l_j}\right\}:=\left\{\lim_{n\rightarrow
\infty}\epsilon_nz^l_{j,\epsilon_n}\ |\ 0\leq l\leq m_j\right\}\subset\Lambda^{\delta_0}
\end{align}
for some $1\leq l_j\leq m_j$. We now define
\begin{align}\label{fh77gyutu776409}
\vartheta_*:=\left\{
\begin{array}
[c]{ll} \frac{1}{100}\min\{\varrho(z^*_l-z^*_{l'}) \ |\ 1\leq
l<l'\leq
l_j\},& \mbox{if} \; l_j\geq 2,\\
+\infty, & \mbox{if} \; l_j=1.
\end{array}
\right.
\end{align}
For simplicity, we shall denote $w_{j, \epsilon_n}$,  $u_{j,\eps_n}$ and $z^l_{j, \eps_n}$ by $w_n,$ $u_n$ and
$z^l_n$, respectively.

\begin{lem}\label{ncvbghfy7ryrtgftt}
 If
\begin{align}\label{vnjh99jmmmxc4dre}
0<\delta< \vartheta_*,
\end{align}
then there exists $C>0$
independent of $n$ such that, for any $0\leq l\leq m_j$ and $n \in \N$ large enough,
\begin{align}\label{23mb00hldfgfyyr7}
|u_{n}(x)|+|\nabla u_n(x)|\leq C\epsilon^{N+2s}_n,\ x\in
\partial {Q}_{\delta\epsilon^{-1}_n}(z^l_n)
\end{align}
and
\begin{align}\label{mb00hldfgfyyr7}
|\nabla_x w_n(\xi)|+|y^{1-2s}\partial_y w_n(\xi)|\leq
C\epsilon^{N}_n,\ \xi\in
\partial^+\mathcal{Q}^+_{\delta\epsilon^{-1}_n}(\xi^l_n),
\end{align}
where $\xi^l_n:=(z^l_n, 0)$.
\end{lem}
\begin{proof}
Combining \eqref{cnvbvhfyf6ryfyyyf}, \eqref{vnjh99jmmmxc4dre}, Lemmas \ref{11123cnvbhhfyf7yfyddsd}, \ref{immm9ibhhfyf7yfyddsd},
\ref{oooldhfyyyf6trrrr} and \ref{ooo9dudyttdtrrrr} yields the results of the lemma, and the proof is completed.
\end{proof}

\begin{lem} \label{phossjdh7777213}
Let ${\bf{n}}=(n_1,\cdots, n_{N+1})$ and
${\bm{\nu}}=(\nu_1,\cdots,\nu_N)$ be the unit outward normal vectors to
$\partial^+ \mathcal{Q}^+_{\delta\epsilon^{-1}_n}(\xi^l_n)$ and
$\partial Q_{\delta\epsilon^{-1}_n}(z^l_n)$, respectively. Then
\begin{align}\label{cmbnug8g7ufyfyyfyf}
\begin{split}
&-\int_{\partial^+ \mathcal{Q}^+_{\delta\epsilon^{-1}_n}(\xi^l_n)}
y^{1-2s}\left( \nabla w_n\cdot\bf {n}\right) \partial_{x_i}
w_n \, dS+\frac{1}{2}\int_{\partial^+
\mathcal{Q}^+_{\delta\epsilon^{-1}_n}(\xi^l_n)}y^{1-2s} |\nabla
w_n|^2 n_i  \, dS  \\ &+\frac{1}{2k_s}\int_{\partial
Q_{\delta\epsilon^{-1}_n}(z^l_n)}V(\eps_n x) |u_n|^2  \nu_i\,
dS-\frac{1}{k_s}\int_{\partial
Q_{\delta\epsilon^{-1}_n}(z^l_n)}F_{\eps_n}(x, |u_n|))\nu_i \, dS
\\
&=\frac{\epsilon_n}{2k_s}\int_{Q_{\delta\epsilon^{-1}_n}(z^l_n)} \partial_{x_i}
V(\eps_n x) |u_n|^2\, dx
+\frac{\eps_n}{k_s}\int_{Q_{\delta\epsilon^{-1}_n}(z^l_n)}\partial_{x_i}\chi(\epsilon_n x)\left(\frac{1}{p}|u_n|^p-G(|u_n|)\right)\, dx,
\end{split}
\end{align}
where $\delta>0$ is a constant.
\end{lem}
\begin{proof}
Since $w_n$ satisfies \eqref{equ21} with $\eps=\eps_n$, by multiplying \eqref{equ21} by $ \partial_{x_i}
w_n$ and integrating on $\mathcal{Q}^+_{\delta\epsilon^{-1}_n}(\xi^l_n)$, then
\begin{align} \label{cnbvbhg7tyfhfhfff}
\int_{\mathcal{Q}^+_{\delta\epsilon^{-1}_n}(\xi^l_n)}
\textnormal{div} \left(y^{1-2s} \nabla w_n\right)  \partial_{x_i}
w_n  \, dx dy=0.
\end{align}
Applying the divergence theorem, we know that
\begin{align*} 
\begin{split}
\int_{\mathcal{Q}^+_{\delta\epsilon^{-1}_n}(\xi^l_n)}
\textnormal{div} \left(y^{1-2s} \nabla w_n\right)  \partial_{x_i}
w_n  \, dx dy
&=\int_{\partial\mathcal{Q}^+_{\delta\epsilon^{-1}_n}(\xi^l_n)}
y^{1-2s}\left( \nabla w_n\cdot\bf {n}\right)  \partial_{x_i}
w_n  \, dS \\ 
&\quad- \int_{\mathcal{Q}^+_{\delta\epsilon^{-1}_n}(\xi^l_n)}
y^{1-2s} \nabla w_n\cdot\nabla\left( \partial_{x_i} w_n\right) \, dx dy.
\end{split}
\end{align*}
It then follows from \eqref{cnbvbhg7tyfhfhfff} that
\begin{align}  \label{mcnvjug78fuuffff}
\hspace{-1cm}\int_{\partial\mathcal{Q}^+_{\delta\epsilon^{-1}_n}(\xi^l_n)}
y^{1-2s}\left( \nabla w_n\cdot\bf {n}\right)  \partial_{x_i}
w_n  \, dS=\int_{\mathcal{Q}^+_{\delta\epsilon^{-1}_n}(\xi^l_n)}
y^{1-2s} \nabla w_n\cdot\nabla\left( \partial_{x_i} w_n\right) \, dx dy.
\end{align}
Let us first calculate the term in the left hand side of \eqref{mcnvjug78fuuffff}.
By the definition of $\partial \mathcal{Q}^+_{\delta\epsilon^{-1}_n}(\xi^l_n)$ and \eqref{equ21}, we find that
\begin{align*} 
&\int_{\partial\mathcal{Q}^+_{\delta\epsilon^{-1}_n}(\xi^l_n)}
y^{1-2s}\left( \nabla w_n\cdot\bf {n}\right)  \partial_{x_i}
w_n \, dS\nonumber\\
&=\int_{\partial^+\mathcal{Q}^+_{\delta\epsilon^{-1}_n}(\xi^l_n)}
y^{1-2s}\left( \nabla w_n\cdot\bf {n}\right)  \partial_{x_i}
w_n  \, dS
-\int_{\partial^*\mathcal{Q}^+_{\delta\epsilon^{-1}_n}(\xi^l_n)}\left(\lim_{y\rightarrow0^+}y^{1-2s}\partial_y w_n\right) \partial_{x_i}
w_n  \,dx\nonumber\\
&=\int_{\partial^+\mathcal{Q}^+_{\delta\epsilon^{-1}_n}(\xi^l_n)}
y^{1-2s}\left( \nabla w_n\cdot\bf {n}\right)  \partial_{x_i}
w_n  \,dS +\frac{1}{k_s}\int_{Q_{\delta\epsilon^{-1}_n}(z^l_n)}\left(-V_{\eps_n}(x) u_n + f_{\eps_n}(x, |u_n|) u_n\right) \partial_{x_i} u_n \, dx.
\end{align*}
Using the divergence theorem and the definition of $F_{\eps_n}$, we see that
\begin{align*} 
&\int_{Q_{\delta\epsilon^{-1}_n}(z^l_n)}\left(-V_{\eps_n}(x) u_n +
f_{\eps_n}(x, |u_n|)u_n\right)  \partial_{x_i}
u_n \, dx\nonumber\\
&=\int_{Q_{\delta\epsilon^{-1}_n}(z^l_n)}\left(-\frac{1}{2}V_{\eps_n}(x)
 \partial_{x_i}\left(|u_n|^2\right) + {\partial_{x_i}}(F_{\eps_n}(x, |u_n|))\right)\, dx-\eps_n\int_{Q_{\delta\epsilon^{-1}_n}(z^l_n)}\partial_{x_i} F(\eps_nx, |u_n|)\, dx\nonumber\\
&=-\frac{1}{2}\int_{\partial
Q_{\delta\epsilon^{-1}_n}(z^l_n)}V_{\eps_n}(x) |u_n|^2 \nu_i\,
dS+\frac{\epsilon_n}{2}\int_{
Q_{\delta\epsilon^{-1}_n}(z^l_n)}\partial_{x_i}
V(\eps_n x) |u_n|^2\, dx\nonumber\\
&\quad+\int_{\partial Q_{\delta\epsilon^{-1}_n}(z^l_n)}F_{\eps_n}(x,
|u_n|))\nu_i \,
dS-\eps_n\int_{Q_{\delta\epsilon^{-1}_n}(z^l_n)}\partial_{x_i}F(\eps_nx,|u_n|)\, dx\nonumber\\
&=-\frac{1}{2}\int_{\partial
Q_{\delta\epsilon^{-1}_n}(z^l_n)}V(\eps_n x) |u_n|^2 \nu_i\,
dS+\frac{\epsilon_n}{2}\int_{
Q_{\delta\epsilon^{-1}_n}(z^l_n)}\partial_{x_i}
V(\eps_n x) |u_n|^2\, dx\nonumber\\
&\quad+\int_{\partial
Q_{\delta\epsilon^{-1}_n}(z^l_n)}F_{\eps_n}(x, |u_n|))\nu_i \,
dS+\eps_n\int_{Q_{\delta\epsilon^{-1}_n}(z^l_n)}\partial_{x_i}\chi(\epsilon_n x)\left(\frac{1}{p}|u_n|^p-G(|u_n|)\right)\, dx.
\end{align*}
We next compute the term in the right hand side of \eqref{mcnvjug78fuuffff}. Using again the divergence theorem, we derive that
\begin{align*}
&\int_{\mathcal{Q}^+_{\delta\epsilon^{-1}_n}(\xi^l_n)}  y^{1-2s}
\nabla w_n\cdot\nabla\left(\partial_{x_i}w_n \right) \, dx dy\nonumber\\
&=\sum^N_{j=1}\int_{\mathcal{Q}^+_{\delta\epsilon^{-1}_n}(\xi^l_n)}y^{1-2s}
\partial_{x_j} w_n\partial_{x_j}\left(\partial_{x_i} w_n \right) \, dx
dy+\int_{\mathcal{Q}^+_{\delta\epsilon^{-1}_n}(\xi^l_n)}y^{1-2s}
\partial_y w_n\partial_y\partial_{x_i}w_n  \, dx dy\nonumber\\
&=\frac{1}{2}\sum^N_{j=1}\int_{\mathcal{Q}^+_{\delta\epsilon^{-1}_n}(\xi^l_n)}y^{1-2s}
\partial_{x_i}\left(\partial_{x_j}
w_n\right)^2\, dxdy+\frac{1}{2}\int_{\mathcal{Q}^+_{\delta\epsilon^{-1}_n}(\xi^l_n)}y^{1-2s}
\partial_{x_i}\left(\partial_y
w_n\right)^2 \, dx dy\nonumber\\
&=\frac{1}{2}\int_{\mathcal{Q}^+_{\delta\epsilon^{-1}_n}(\xi^l_n)}y^{1-2s}
\partial_{x_i}\left(|\nabla w_n|^2\right) \, dx dy\nonumber\\
&=\frac{1}{2}\int_{\partial^+\mathcal{Q}^+_{\delta\epsilon^{-1}_n}(\xi^l_n)}y^{1-2s}
|\nabla w_n|^2 n_i  \, dS.
\end{align*}
Coming back to \eqref{mcnvjug78fuuffff}, we then derive \eqref{cmbnug8g7ufyfyyfyf}. This completes the proof.
 \end{proof}

\begin{cor} \label{2cmbnug8g7ufyfyyfyf1}
For any ${\bf{t}}=(t_1,\cdots,t_N)\in\R^N$, there holds that
\begin{align} \label{2cmbnug8g7ufyfyyfyf}
\begin{split}
&-\int_{\partial^+ \mathcal{Q}^+_{\delta\epsilon^{-1}_n}(\xi^l_n)}
y^{1-2s}\left( \nabla w_n\cdot\bf {n}\right)\left(\nabla_x w_n  \cdot {\bf{t}}\right) \, dS +\frac{1}{2} \sum^N_{i=1}\int_{\partial^+\mathcal{Q}^+_{\delta\epsilon^{-1}_n}(\xi^l_n)}y^{1-2s} |\nabla w_n|^2 t_i {n}_i   \, dS\\
&+\frac{1}{2k_s}\int_{\partial Q_{\delta\epsilon^{-1}_n}(z^l_n)}V(\eps_n x) |u_n|^2
\left({\bf{t}}\cdot {\bm{\nu}}\right)\, dS-\frac{1}{k_s}\int_{\partial Q_{\delta\epsilon^{-1}_n}(z^l_n)}
F_{\eps_n}(x, |u_n|))\left({\bf{t}}\cdot {\bm{\nu}}\right) \, dS\\
&=\frac{\epsilon_n}{2k_s}\int_{Q_{\delta\epsilon^{-1}_n}(z^l_n)}
\left((\nabla V)(\eps_n x) \cdot {\bf{t}} \right) |u_n|^2\,dx\\
& \quad +\frac{\eps_n}{k_s}\int_{Q_{\delta\epsilon^{-1}_n}(z^l_n)}((\nabla\chi)(\epsilon_n
x) \cdot{\bf{t}})\left(\frac{1}{p}|u_n|^p-G(|u_n|)\right)\, dx,
\end{split}
\end{align}
where ${\bf{n}}=(n_1,\cdots, n_{N+1})$ is the unit outward normal vector to $\partial^+ \mathcal{Q}^+_{\delta\epsilon^{-1}_n}(\xi^l_n)$.
\end{cor}
\begin{proof}
Multiplying \eqref{equ21} with $\eps=\eps_n$ by $\nabla_x w_n \cdot {\bf{t}}$,
integrating on $\mathcal{Q}^+_{\delta\epsilon^{-1}_n}(\xi^l_n)$ and using the divergence theorem, we can obtain that
$$
\hspace{-1cm}\int_{\partial\mathcal{Q}^+_{\delta\epsilon^{-1}_n}(\xi^l_n)}
y^{1-2s}\left( \nabla w_n\cdot\bf {n}\right) \left(\nabla_x w_n \cdot {\bf{t}} \right)\, dS=\int_{\mathcal{Q}^+_{\delta\epsilon^{-1}_n}(\xi^l_n)}
y^{1-2s} \nabla w_n\cdot\nabla\left( \nabla_x w_n \cdot {\bf{t}} \right) \, dx dy.
$$
At this point, reasoning as the proof of Lemma \ref{phossjdh7777213}, it is not difficult to deduce the result of this corollary. Thus the proof is completed.
\end{proof}

\begin{lem}\label{cnvbghyyt7888sdd}
If $N \geq 2$, then, for any $1\leq l\leq m_j$,
\begin{align} \label{con}
\lim_{\epsilon\rightarrow 0^+}\mbox{\textnormal{dist}}(\epsilon
z^l_{\epsilon,j},\mathcal{V})=0,
\end{align}
where $\mathcal{V}$ is defined by \eqref{defv}.
\end{lem}
\begin{proof}
Arguing indirectly, we assume that there exist $1\leq l_0\leq m_j$
and $\{\eps_n\} \subset \R^+$ with $\eps_n=o_n(1)$ such that
\begin{align} \label{assume}
\lim_{n\rightarrow\infty}\mbox{dist}(\epsilon_n z^{l_0}_{j,\epsilon_n},\mathcal{V})>0.
\end{align}
Without loss of generality, we assume that $\lim_{n\rightarrow
\infty}\epsilon_nz^{l_0}_{j,\epsilon_n}$ exists. For simplicity, we shall denote $w_{j,\epsilon_n}$,  $u_{j,\eps_n}$, $z^{l_0}_{j,\eps_n}$ and $(z^{l_0}_{j,\eps_n},0)$ by $w_n,$
$u_n$, $z^{l_0}_n$ and $\xi^{l_0}_n$, respectively.

According to Lemma \ref{cnvbhhfyf6yr66} and \eqref{assume}, we have that $z^{l_0}_j:=\lim_{n\rightarrow \infty}\epsilon_n z^{l_0}_{n}\in\Lambda^{\delta_0}\setminus\mathcal{V}$. Thus there exists $\delta'>0$ independent of $n$ such that, for any $0<\delta<\delta',$
\begin{align}\label{bcvyfhiifjy}
\inf_{x\in B_{\delta\epsilon^{-1}_n}(z^{l_0}_{n})}\nabla
V(\epsilon_n x)\cdot\nabla V(\epsilon_n
z^{l_0}_{n})\geq\frac{1}{2}\big|\nabla V(z^{l_0}_j)\big|^2>0.
\end{align}
Using \eqref{vloc} and the fact that $z^{l_0}_j\in\Lambda^{\delta_0}\setminus\mathcal{V},$ we see that
\begin{align}\label{cnvbhfyyfyf55re}
\inf_{x\in B_{\delta_0\epsilon^{-1}_n}(z^{l_0}_{n})}\nabla
V(\epsilon_n z^{l_0}_{n})\cdot(\nabla\chi)(\epsilon_n x) \geq 0.
\end{align}
In light of the definition of
$\partial^+\mathcal{Q}^+_{\delta\epsilon^{-1}_n}(\xi_n^{l_0})$,
there holds that
\begin{align}\label{cnvbuudf8d77e7eds}
\int_{\partial^+ \mathcal{Q}^+_{\delta\epsilon^{-1}_n}(\xi_n^{l_0})} \,
=\int_{\Gamma_1}+\cdots+\int_{\Gamma_N}+\int_{\Gamma_{N+1}},
\end{align}
where $z^{l_0}_n=(z^{l_0}_{n,1},\cdots, z^{l_0}_{n,N}) \in \R^N$,
\begin{align*}
\Gamma_i&=\{\xi=(x,y)\in \R^{N+1}\  |\
x_i=z^{l_0}_{n,i}-\delta\epsilon^{-1}_n\ \mbox{or}\
x_i=z^{l_0}_{n,i}+\delta\epsilon^{-1}_n,\nonumber\\
&\quad\quad\ |x_k-z^{l_0}_{n,k}|\leq \delta\epsilon^{-1}_n \
\mbox{for}\ k\neq i \ \mbox{and}\ 0\leq y\leq
\delta\epsilon^{-1}_n\} \ \mbox{for}\ 1\leq i\leq N
\end{align*}
and
$$
\Gamma_{N+1}=\{\xi=(x,y)\in \R^{N+1}\  |\ |x_k-z^{l_0}_{n,k}|\leq
\delta\epsilon^{-1}_n \ \mbox{for}\ 1\leq k\leq N \ \mbox{and}\
y=\delta\epsilon^{-1}_n\}.
$$
Let ${\bf{n}}=(n_1,\cdots, n_{N+1})$ be the unit outward normal vector to
$\partial^+ \mathcal{Q}^+_{\delta\epsilon^{-1}_n}(\xi_n^{l_0})$ and define
$$
{\bf{t}}_n:={\nabla V(\epsilon_n z^{l_0}_{n})}/{|\nabla V(\epsilon_n z^{l_0}_{n})|}.
$$
In the following, we shall take advantage of Corollary
\ref{2cmbnug8g7ufyfyyfyf1} to reach a contradiction. We now
estimate every term in \eqref{2cmbnug8g7ufyfyyfyf} with
${\bf{t}}={\bf{t}}_n$, $l=l_0$ and
$\delta:=\frac{1}{2}\min\{\delta_0, \delta', \vartheta_*\}$, where
$\vartheta_*>0$ is defined by (\ref{fh77gyutu776409}).
In view of \eqref{mb00hldfgfyyr7}, we get that, for any
$1\leq i\leq N,$
\begin{align} \label{cnvbghgyt77tyrddfc}
\begin{split}
&\left|\int_{\Gamma_i} y^{1-2s}\left( \nabla w_n\cdot \bf
{n}\right)({\nabla_x w_n \cdot \bf{t}}_n) \, dS\right| \\
&\leq\int_{\Gamma_i} y^{1-2s}|\nabla_x w_n|^2 \, dS \\
&\leq C\epsilon_n^{2N}\int_{\Gamma_i} y^{1-2s} \, dS \\
&=C\epsilon_n^{2N}
(2\delta\epsilon^{-1}_n)^{N-1}\int^{\delta\epsilon^{-1}_n}_0 y^{1-2s}\, dy \\ 
&\leq C \epsilon^{N+2s-1}_{n}.
\end{split}
\end{align}
Moreover, we derive that
\begin{align}\label{nncnvbghgyt77tyrddfc}
\begin{split}
&\left|\int_{\Gamma_{N+1}} y^{1-2s}\left( \nabla w_n\cdot \bf
{n}\right)(\nabla_x w_n \cdot {\bf{t}_n}) \, dS\right|\\
&\leq \int_{\Gamma_{N+1}} | y^{1-2s}\partial_y w_n| |\nabla_x w_n| \, dS \\ 
&\leq C\epsilon_n^{2N}\int_{\Gamma_{N+1}} \, dS\\
&\leq C\epsilon^{N}_{n}.
\end{split}
\end{align}
From \eqref{cnvbuudf8d77e7eds}, \eqref{cnvbghgyt77tyrddfc}
and \eqref{nncnvbghgyt77tyrddfc}, we then have that
\begin{align}\label{vcmbnjghguy88y8ee}
\left|\int_{\partial^+
\mathcal{Q}^+_{\delta\epsilon^{-1}_n}(\xi^{l_0}_n)} y^{1-2s}\left(
\nabla w_n\cdot\bf {n}\right)(\nabla_x w_n \cdot {\bf{t}}_n) \,
dS\right|\leq C\epsilon^{\min\{N+2s-1,N\}}_{n}.
\end{align}
We write ${\bf{t}}_n=(t_{n,1},\cdots,t_{n,N})$. It then follows from \eqref{mb00hldfgfyyr7} that, for any $1\leq i\leq N,$
\begin{align}\label{cnvbghhfyftyy77dy}
\begin{split}
&\left |\sum^N_{j=1}\int_{\Gamma_i}y^{1-2s} |\nabla w_n|^2 t_{n,j}
n_j  \, dS\right| =\left|\int_{\Gamma_i}y^{1-2s} |\nabla
w_n|^2t_{n,i}n_i \,
dS\right|\\
&\leq \int_{\Gamma_i}y^{1-2s} |\nabla_x w_n|^2 \,
dS+\int_{\Gamma_i}y^{2s-1} |y^{1-2s}\partial_y w_n|^2 \,
dS\\
&\leq C\epsilon^{2N}_n\int_{\Gamma_i}y^{1-2s}\, dS+
C\epsilon^{2N}_n\int_{\Gamma_i}y^{2s-1}\, dS\\
&=C\epsilon_n^{2N}
(2\delta\epsilon^{-1}_n)^{N-1}\int^{\delta\epsilon^{-1}_n}_0y^{1-2s}dy+C\epsilon_n^{2N}
(2\delta\epsilon^{-1}_n)^{N-1}\int^{\delta\epsilon^{-1}_n}_0y^{2s-1}dy
\\
&=C\epsilon^{N+2s-1}_{n}+C\epsilon^{N-2s+1}_{n}\\
&\leq C\epsilon^{\min\{N+2s-1, N-2s+1\}}_{n}.
\end{split}
\end{align}
Since $n_j=0$ on $\Gamma_{N+1}$ for $1\leq j\leq N$, then
\begin{align}\label{cnvbght7gyiiidh}
\sum^N_{j=1}\int_{\Gamma_{N+1}}y^{1-2s} |\nabla
w_n|^2 t_{n,j} n_j \, dS=0.
\end{align}
Combining \eqref{cnvbuudf8d77e7eds}, \eqref{cnvbghhfyftyy77dy} and
\eqref{cnvbght7gyiiidh} leads to
\begin{align}\label{cnvb99ff7tyhfhffff}
\left|\sum^N_{j=1} \int_{\partial^+
\mathcal{Q}^+_{\delta\epsilon^{-1}_n}(\xi^{l_0}_n)}y^{1-2s}
|\nabla w_n|^2 t_{n,j} n_j \, dS\right|\leq
C\epsilon^{\min\{N+2s-1, N-2s+1\}}_{n}.
\end{align}
By the definition of $F_{\eps_n}$, we obtain that $|F_{\eps_n}(x, |t|)| \leq
C(|t|^2+|t|^{2^*_s})$ for any $x \in \R^N $ and $t \in \R$. Taking into account \eqref{23mb00hldfgfyyr7}, we then conclude that
\begin{align}\label{mmcnvcbvyyf7fydd}
&\left|\frac{1}{2}\int_{\partial
Q_{\delta\epsilon^{-1}_n}(z^{l_0}_n)}V_{\eps_n}(x) |u_n|^2
\left({\bf{t}}_n\cdot {\bm{\nu}}\right)\, dS-\int_{\partial
Q_{\delta\epsilon^{-1}_n}(z^{l_0}_n)}F_{\eps_n}(x, |u_n|)\left({\bf{t}}_n\cdot {\bm{\nu}}\right)\, dS\right|\nonumber\\
&\leq C\epsilon^{2(N+2s)}_n\int_{\partial
Q_{\delta\epsilon^{-1}_n}(z^{l_0}_n)}\, dS=C\epsilon^{N+4s+1}_n.
\end{align}
Denoting the left hand side of \eqref{2cmbnug8g7ufyfyyfyf} by $\mathcal{L}_n$ and using \eqref{vcmbnjghguy88y8ee}, \eqref{cnvb99ff7tyhfhffff} and \eqref{mmcnvcbvyyf7fydd}, we then get that
\begin{align}\label{ttttqassd}
\mathcal{L}_n\leq C\epsilon^{\min\{N+2s-1, N-2s+1\}}_n.
\end{align}
On the hand other, from Lemma \ref{cnvbhhfyf6yr66}, we have that
$w_n(\cdot+\xi^{l_0}_{n})\rightharpoonup w_{l_0}\neq 0$ in
$X^{1,s}(\mathbb{R}^{N+1}_+)$ as $n \to \infty$. This indicates that $u_n(\cdot+z^{l_0}_{n})\rightharpoonup u_{l_0}\neq 0$ in $H^{s}(\mathbb{R}^{N})$ as $n \to \infty$, where $u_{l_0}=w_{l_0}(\cdot,0)$. Therefore, by
(\ref{bcvyfhiifjy}), we get that, for any $n \in \N$ large enough,
\begin{align*} 
&\epsilon_n\int_{Q_{\delta\epsilon^{-1}_n}(z^{l_0}_n)}\left((\nabla
V)(\epsilon_n x) \cdot {\bf{t}}_n\right)|u_n|^2 \,
dx\geq\frac{\epsilon_n}{2}\big|\nabla
V(z^{l_0}_j)\big|\int_{Q_{\delta\epsilon^{-1}_n}(0)}|u_n(\cdot+z^{l_0}_{n})|^2
\,dx\geq C_0\epsilon_n,
\end{align*}
where $$ C_0:=\frac{1}{4}\big|\nabla
V(z^{l_0}_j)\big|\int_{\mathbb{R}^N}|u_{l_0}|^2\,dx>0.
$$
In addition, by \eqref{cnvbhfyyfyf55re} and the fact that $ \frac{|t|^p}{p}-G(|t|)\geq 0$ for any $t \in \R$,
we obtain that
\begin{align*} 
\int_{Q_{\delta\epsilon^{-1}_n}(z^{l_0}_n)}\left({\bf{t}}_n\cdot(\nabla\chi)(\epsilon_n
x)\right)\left(\frac{1}{p}|u_n|^p-G(|u_n|)\right)\, dx\geq 0.
\end{align*}
Denoting the right hand side of \eqref{2cmbnug8g7ufyfyyfyf} by $\mathcal{R}_n$, we then have that
\begin{align}\label{12xcttttqassd}
\mathcal{R}_n\geq \frac{C_0}{2k_s}\epsilon_n.
\end{align}
By applying Corollary \ref{2cmbnug8g7ufyfyyfyf1}, \eqref{ttttqassd} and \eqref{12xcttttqassd}, we then deduce that
\begin{align} \label{contr}
C\epsilon^{\min\{N+2s-1,
N-2s+1,\}}_n\geq\mathcal{L}_n=\mathcal{R}_n\geq
\frac{C_0}{2k_s}\epsilon_n.
\end{align}
Due to $N \geq 2$, we find that $N>\max\{2s,2-2s\}$. Then we see that
$\min\{N+2s-1, N-2s+1\}>1$. As a result, \eqref{contr} induces a
contradiction. This in turn suggests that \eqref{con} holds true,
and the proof is completed.
\end{proof}

We are now in a position to prove Theorem \ref{wejgh77rtff111}.

\begin{proof}[\bf Proof of Theorem \ref{wejgh77rtff111}.] Since
$\mathcal{V}$ is a compact subset of $\Lambda$, then
$\mbox{dist}(\mathcal{V}, \partial \Lambda)>0.$ Choosing $\delta\in
(0, \mbox{dist}(\mathcal{V}, \Lambda))$ and applying Lemmas
\ref{11123cnvbhhfyf7yfyddsd} and \ref{cnvbghyyt7888sdd}, we then deduce
that
\begin{align}\label{cmvnvjfu99sscv}|u_{j, \epsilon}(x)|\leq\frac{C}{1+|\mbox{dist}(x,(\mathcal{V}^\delta)_\epsilon)|^{N+2s}},\
1\leq j\leq k,
\end{align}
where $u_{j,\epsilon}=w_{j,\epsilon}(\cdot,0)$. For any $x \in
\R^N \setminus \Lambda_\epsilon$, we know that
$$
\mbox{dist}(x, (\mathcal{V}^\delta)_\epsilon)\rightarrow\infty \,\, \mbox{as} \,\, \epsilon\rightarrow 0^+.
$$
Taking $\epsilon_k>0$ smaller if necessary and using \eqref{cmvnvjfu99sscv}, we then obtain that, if
$0<\epsilon<\epsilon_k$, then for any $x \in \R^N \setminus \Lambda_\epsilon$,
$$
|u_{j,\epsilon}(x)|\leq (a/4)^{1/(p-2)}.
$$
This then indicates that $f_{\eps}(x, |u_{j,\epsilon}|) = |u_{j,\epsilon}|^{p-2}$ and $u_{j, \eps}$ is a solution to \eqref{equ1} for any $j \geq 1$.

From Theorem \ref{cvnbmiif9f8ufjhfy1}, we know that $u_{j, \eps}$
is sign-changing solution to \eqref{equ1} for any $j \geq 2$. We
now prove that $u_{1, \eps}$ is a positive solution to
\eqref{equ1}. To do this, according to the maximum principle, it
suffices to show that $u_{1, \epsilon}(x)\geq 0$ for any
$x\in\mathbb{R}^N$. We argue by contradiction that $u_{1, \eps}$
is sign-changing. Then we have that $u^\pm_{1,\epsilon}=\max\{\pm
u_{1,\epsilon},0\}\neq 0.$ This in turn implies that
$w^\pm_{1,\epsilon} \neq 0$. Note that $w_{1,\epsilon}$ satisfies
the equation
\begin{align}\label{equw1}
\left\{
\begin{aligned}
-\mbox{div}(y^{1-2s} \nabla w_{1, \eps})&=0 \quad \  \ \, \hspace{4.5cm}\mbox{in} \,\, \R^{N+1}_+,\\
-k_s \frac{\partial w_{1, \eps}}{\partial {\nu}}&=-V_{\eps}(x) w_{1,\eps} + f_{\eps}(x, |w_{1,\eps}|) w_{1, \eps} \ \qquad  \mbox{on} \,\, \R^N \times \{0\}.
\end{aligned}
\right.
\end{align}
In current situation, we remind that $f_{\eps}(x, |t|):=|t|^{p-2}$ for $x \in \R^N$ and $t \in \R$. Therefore, it holds that
\begin{align*}
\Phi_{\epsilon}(w^\pm_{1,\epsilon})
&=\Phi_{\epsilon}(w^\pm_{1,\epsilon}) - \frac 12 \Phi'_{\epsilon}(w_{1,\epsilon})w^\pm_{1,\epsilon}  \nonumber \\
&=\int_{\mathbb{R}^N}\left(\frac{1}{2}f_{\epsilon}(x,|u^\pm_{1,\epsilon}|)|u^\pm_{1,\epsilon}|^2
-F_{\epsilon}(x, |u^\pm_{1,\epsilon}|)\right)\,dx\\
&=\left(\frac{1}{2}-\frac{1}{p} \right)\int_{\mathbb{R}^N} |u^\pm_{1,\epsilon}|^{p}\,dx >0.
\end{align*}
Since $\Phi_{\epsilon}(w_{1,\epsilon})=\Phi_{\epsilon}(w^+_{1,\epsilon})+\Phi_{\epsilon}(w^-_{1,\epsilon})=c_{1, \epsilon}$,
then $ 0<\Phi_{\epsilon}(w^\pm_{1,\epsilon})<c_{1, \epsilon}$.
For $t\geq 0,$ we now define
$\phi_{\eps}^\pm(t) :=\Phi_{\epsilon}(tw^\pm_{1,\epsilon}).$ It is easy to compute that
\begin{align*}
\frac{d\phi_{\eps}^\pm}{dt}=t\left(\int_{\R^{N+1}_+}y^{1-2s}|w^\pm_{1,\epsilon}|^2 \, dxdy + \int_{\R^N} V_{\eps}(x)|u_{1,\eps}|^2 \, dx
-\int_{\mathbb{R}^N}f_{\epsilon}(x,t|u^\pm_{1,\epsilon}|)|u^\pm_{1,\epsilon}|^2\,dx\right).
\end{align*}
Due to $p>2$, then $\phi_{\eps}^\pm(t)>0$ for any $t>0$ small enough, $\phi_{\eps}^\pm(t)<0$ for any $t>0$ large enough and $\lim_{t\rightarrow+\infty}\phi_{\eps}^\pm(t)=-\infty.$ In addition, we find that $\max_{t\geq
0}\phi_{\eps}^\pm(t)$ is achieved at a unique $t=t_\pm.$ Since
$\frac{d\phi_{\eps}^\pm}{dt}{\mid_{t=1}}=0$, we then get that $t_\pm=1.$ Observe that
\begin{align*}
\lim_{t\rightarrow+\infty}\Phi_{\epsilon}(t((1-\theta)w^\pm_{1,\epsilon}+\theta
R_1e_1))=-\infty
\end{align*}
holds uniformly for $\theta\in[0,1]$, where $R_1>0$ comes from \eqref{cmvnbjjgug77tyruur1}. It then follows that there exists  $t_*>1$
such that
$$
\sup_{\theta\in[0,1]}\Phi_{\epsilon}(t_*((1-\theta)w^\pm_{1,\epsilon}+\theta R_1e_1))<0.
$$
Define
\begin{align*}
\gamma_\pm(\theta):=\left\{
\begin{array}
[c]{ll}
3\theta t_*w^\pm_{1,\epsilon},& \mbox{if}\ 0\leq \theta<\frac{1}{3},\\
(2-3\theta) t_*w^\pm_{1,\epsilon}+(3\theta-1)t_*R_1e_1, & \mbox{if}\ \frac{1}{3}\leq\theta\leq \frac{2}{3},\\
(3(1-t_*)\theta-(2-3t_*))R_1e_1, & \mbox{if}\ \frac{2}{3}\leq\theta\leq 1,\\
-\gamma_\pm(-\theta),& \mbox{if} -1\leq \theta\leq 0.
\end{array}
\right.
\end{align*}
Then $\gamma_\pm\in G_1$ and
\begin{align*}
\sup_{\theta\in[0,1]}\Phi_{\epsilon}(\gamma_\pm(\theta))=\Phi_{\epsilon}(w^\pm_{1,\epsilon})<c_{1, \epsilon}.
\end{align*}
This contradicts the definition of $c_{1,\epsilon}.$ Thus we have that $u_{1, \eps}$ is positive solution to \eqref{equ1}.

Making a change of
variable and utilizing \eqref{cmvnvjfu99sscv}, we then have the
desired decay of $u_{\epsilon,j}$. Thus we have completed the
proof.
\end{proof}

\section{Sobolev critical growth}\label{critical}

In this section, we shall investigate semiclassical states to
\eqref{equ11} in the Sobolev critical case, i.e.
$\mathcal{N}(u)u=\mu |u|^{p-2} u+|u|^{2^*_s-2}u$ for $2<p<2^*_s$
and  $\mu>0$. To do this, let us first introduce some notations. Let
$\varphi \in C^{\infty}(\R^+, [0, 1])$ be such that $\varphi(t)=1$
for $ 0 \leq t \leq 1$, $\varphi(t)=0$ for $t \geq 2$ and
$\varphi'(t) \leq 0$ for $t \geq 0$. For any $\eps>0$ and $t \geq
0$, we define $b_{\eps}(t):=\varphi(\eps t)$ and
$m_{\eps}(t):=\int_{0}^{t} b_{\eps}(s) \, ds$.

\begin{lem} \cite[Proposition 2.1] {CLW} \label{mprop}
The functions $b_{\eps}$ and $m_{\eps}$ satisfy the following
properties.
\begin{enumerate}
\item[$(\textnormal{i})$] $t b_{\eps}(t) \leq m_{\eps}(t) \leq t$
for any $t \geq 0$. \item[$(\textnormal{ii})$] There exists $c>0$
such that $m_{\eps}(t) \leq c  / \eps $ for any $t \geq 0$. If $ 0
\leq t  \leq 1 / \eps$, then $b_{\eps}(t)=1$ and $m_{\eps}(t)=t$.
\end{enumerate}
\end{lem}

Let $\eta:\R^+ \to [0,1]$ be the cut-off function given in Section \ref{subcritical}.
Define $\chi(x):=\eta(\mbox{dist}(x, \, \Lambda))$ for $x \in \R^N$ and
\begin{align*}
g_{\eps}(t):=\min\left\{h_{\eps}(t), \, a/4\right\}  \quad
\mbox{for} \,\, t \geq 0,
\end{align*}
where
\begin{align*}
h_{\eps}(t):=\mu t^{p-2} + \frac{p}{2^*_s} t^{p-2}
\left(m_{\eps}(t^2) \right)^{\frac{2^*_s-p}{2}} +
\frac{2^*_s-p}{2^*_s} t^{p} \left(m_{\eps}(t^2)
\right)^{\frac{2^*_s-p}{2} -1} b_{\eps}(t^2).
\end{align*}
Observe that
\begin{align*}
H_{\eps}(t):=\int_0^t h_{\eps}(s) s \, ds = \frac{\mu }{p}t^p +
\frac{t^{p}}{2^*_s}\left(m_{\eps}(t^2)
\right)^{\frac{2^*_s-p}{2}},
\end{align*}
then
\begin{align} \label{har}
\frac 1 p h_{\eps}(t)t^2 -H_{\eps}(t)=\frac{2^*_s-p}{2^*_s p} t^{p} \left(m_{\eps}(t^2)
\right)^{\frac{2^*_s-p}{2} -1} b_{\eps}(t^2) \geq 0 \quad \mbox{for any} \,\, t \geq 0.
\end{align}
Let $G_{\eps}(t):=\int_0^t g_{\eps}(s) s \, ds$, it then follows that
\begin{align} \label{gar}
\frac 12 g_{\eps}(t) t^2 -G_{\eps}(t) \geq 0 \quad \mbox{for any} \,\, t \geq 0.
\end{align}
Define
\begin{align*}
f_{\eps}(x, t):=\left(1- \chi(\eps x)\right) h_{\eps}(t)+\chi(\eps x){g}_{\eps}(t)
\end{align*}
and
\begin{align*}
 F_{\eps}(x, t):=\int_{0}^t f_{\eps}(x, s)s \, ds
 =\left(1- \chi(\eps x)\right) \left(\frac{\mu}{p}t^p + \frac{t^{p}}{2^*_s}\left(m_{\eps}(t^2) \right)^{\frac{2^*_s-p}{2}} \right) + \chi(\eps x) G_{\eps}(t).
\end{align*}
We now introduce a modified equation as
\begin{align}\label{equ22}
\left\{
\begin{aligned}
-\mbox{div}(y^{1-2s} \nabla w)&=0 \hspace{4cm}\,\,\mbox{in} \,\, \R^{N+1}_+,\\
-k_s \frac{\partial w}{\partial {\nu}}&=-V_{\eps}(x) w + f_{\eps}(x, |w|) w \, \,\qquad  \mbox{on} \,\, \R^N \times \{0\}.
\end{aligned}
\right.
\end{align}
The energy functional associated to \eqref{equ22} is defined by
$$
\Phi_{\eps}(w):=\frac {k_s}{2} \int_{\R^{N+1}_+} y^{1-2s} |\nabla w|^2 \, dxdy
+ \frac 12 \int_{\R^N} V_{\eps}(x)|w(x,0)|^2 \, dx - \int_{\R^N} F_{\eps}(x, |w(x, 0)|) \, dx.
$$
It is not hard to check that $\Phi_{\eps} \in C^1(X^{1, s}(\R^{N+1}_+), \R)$ and
\begin{align*}
\Phi_{\eps}'(w) \psi&=k_s\int_{\R^{N+1}_+} y^{1-2s} \nabla w \cdot \nabla \psi \, dxdy + \int_{\R^N} V_{\eps}(x) w(x, 0) \psi(x, 0) \, dx \\
& \quad - \int_{\R^N} f_{\eps}(x, |w(x, 0)|) w(x, 0) \psi(x, 0) \, dx
\end{align*}
for any $\psi \in X^{1, s}(\R^{N+1}_+)$. Consequently, any critical point of $\Phi_{\eps}$ corresponds to a solution to \eqref{equ22}.

\subsection{Existence of semiclassical states}
In the following, we are going to adapt Theorem \ref{exist} to
prove the existence of semiclassical states to \eqref{equ22}. To
do this, let us first demonstrate that the energy functional $\Phi_{\eps}$ satisfies the
Palais-Smale  condition in $X^{1, s}(\R^{N+1}_+)$.

\begin{lem} \label{ps}
For any $\eps>0$ small enough, the energy functional $\Phi_{\eps}$ satisfies the Palais-Smale condition in $X^{1, s}(\R^{N+1}_+)$.
\end{lem}
\begin{proof}
Suppose $\{w_n\} \subset X^{1, s}(\R^{N+1}_+)$ is a Palais-Smale
sequence to $\Phi_{\eps}$, i.e.,
$$
 \left |\Phi_{\eps}(w_n) \right| \leq C, \quad \Phi_{\eps}'(w_n)=o_n(1).
$$
Our aim is to deduce that $\{w_n\}$ has a convergent subsequence in $X^{1, s}(\R^{N+1}_+)$. Let us first show that $\{w_n\}$ is bounded in $X^{1, s}(\R^{N+1}_+)$. In light of \eqref{har} and \eqref{gar}, we have that
\begin{align*} 
C(1 +\|w\|_{1, s}) &\geq \Phi_{\eps} (w_n) -\frac 1 p
\Phi_{\eps}'(w_n)w_n \\ \nonumber & =\frac{p-2}{2p} \left(
k_s\int_{\R^{N+1}_+} y^{1-2s} |\nabla w_n|^2 \, dxdy + \int_{\R^N}
V_{\eps}(x)|w_n(x, 0)|^2 \, dx \right) \\ 
& \quad +\int_{\R^N} \frac 1 p f_{\eps}(x, |w_n(x, 0|) |w_n(x, 0)|^2
-F_{\eps}(x, |w_n(x, 0)|) \, dx \\ 
& =\frac{p-2}{2p}\left( k_s\int_{\R^{N+1}_+} y^{1-2s} |\nabla w_n|^2 \, dxdy +
\int_{\R^N} V_{\eps}(x)|w_n(x, 0)|^2 \, dx \right)   \\ 
& \quad + \frac{2^*_s-p}{2^*_s p}\int_{\R^N} \left(1- \chi(\eps
x)\right) |w_n(x, 0)|^{p} \left(m_{\eps}(|w_n(x, 0)|^2)
\right)^{\frac{2^*_s-p}{2} -1} b_{\eps}(|w_n(x, 0)|^2) \,dx\\
& \quad + \int_{\R^N} \chi(\eps x) \left(\frac 1p
g_{\eps}(|w_n(x, 0)|)|w_n(x, 0)|^2-G_\eps(|w_n(x, 0)|) \right) \, dx\\
& \geq \frac{p-2}{2p} \left( k_s\int_{\R^{N+1}_+}
y^{1-2s} |\nabla w_n|^2 \, dxdy + \int_{\R^N} V_{\eps}(x)|w_n(x,
0)|^2 \, dx \right)   \\ 
& \quad - \frac{p-2}{2p}\int_{\R^N} \chi(\eps x) g_{\eps}(|w_n(x, 0)|)|w_n(x, 0)|^2\, dx.
\end{align*}
Note that $V(x) \geq a$ for any $x \in \R^N$. In addition, by the definition of $g_{\eps}$,
there holds that $0 \leq g_{\eps}(|t|) \leq \frac a 4$ for any $t \in \R$.
Thus we can drive from the inequality above that $\{w_n\}$ is bounded in $X^{1, s}(\R^{N+1}_+)$.
In view of Lemma \ref{embedding}, we then know that there is $w \in X^{1 ,s}(\R^{N+1}_+)$
such that $w_n \wto w$ in $X^{1, s}(\R^{N+1}_+)$ and $w_n(\cdot, 0) \to w(\cdot, 0)$
in $L_{\textnormal{loc}}^q(\R^N)$ as $n \to \infty$ for any $2 \leq q <2^*_s$. Furthermore, we have that $\Phi_{\eps}'(w)=0$.

We next prove that, up to a subsequence, $w_n \to w$ in $X^{1, s}(\R^{N+1}_+)$ as $n \to \infty$.
Define $z_n:=w_n -w$, then $\Phi_{\eps}'(w_n) z_n=o_n(1)$ and $ \Phi_{\eps} '(w) z_n=0$. This indicates that
\begin{align} \label{conv}
\begin{split}
o_n(1)&=\Phi_{\eps}'(w_n) z_n -\Phi_{\eps}'(w) z_n \\
&=k_s\int_{\R^{N+1}_+} y^{1-2s} |\nabla z_n|^2 \, dxdy + \int_{\R^N} V_{\eps}(x)|z_n(x, 0)|^2 \, dx \\
& \quad + \int_{\R^N}\left( f_{\eps}(x, |w(x, 0)|)w(x, 0)
-f_{\eps}(x, |w_n(x, 0)|)
w_n(x, 0)\right) z_n(x, 0) \, dx \\
&=k_s\int_{\R^{N+1}_+} y^{1-2s} |\nabla z_n|^2 \, dxdy + \int_{\R^N} V_{\eps}(x)|z_n(x, 0)|^2 \, dx \\
& \quad -\int_{\R^N} \left(1 -\chi(\eps x)\right) \left(h_{\eps}(|w_n(x, 0)|) w_n(x, 0)-h_{\eps}(|w(x, 0)|)w(x, 0) \right) z_n(x, 0) \, dx \\
& \quad - \int_{\R^N} \chi(\eps x) \left(g_{\eps}(|w_n(x, 0)|)-g_{\eps}(|w(x, 0)|) \right) z_n(x, 0)  w(x, 0) \, dx \\
& \quad - \int_{\R^N} \chi(\eps x) g_{\eps}(|w_n(x ,0)|) |z_n(x, 0)|^2 \, dx.
\end{split}
\end{align}
Since $w_n \rightharpoonup w$ in $X^{1, s}(\R^{N+1}_+)$ as $n \to
\infty$, then $\left(g_{\eps}(|w_n(\cdot, 0)|) -g_{\eps}(|w(\cdot,
0)|)\right) z_n(\cdot, 0) \wto 0$ in $L^2(\R^N)$ as $n \to
\infty$. Therefore, it holds that
\begin{align} \label{estg}
\left|\int_{\R^N} \chi(\eps x) \left(g_{\eps}(|w_n(x, 0)|)-g_{\eps}(|w(x, 0)|) \right) z_n(x, 0)  w(x, 0) \, dx \right|=o_n(1).
\end{align}
Taking into account the mean value theorem, we obtain that there exists a
function $\theta_n: \R^N \to \R$ such that
\begin{align*}
&\left|\int_{\R^N} \left(1 -\chi(\eps x)\right) \left(|w_n(x, 0)|^{p-2} w_n(x, 0)-|w(x, 0)|^{p-2}w(x, 0) \right) z_n(x, 0) \,dx \right| \\
&= (p-1)\left|\int_{\R^N} \left(1 -\chi(\eps x)\right) |\theta_n|^{p-2} |z_n(x, 0)|^2\,dx \right|,
\end{align*}
where $\theta_n(x):=\widehat{\theta}_n(x)w_n(x, 0) + (1-\widehat{\theta}_n(x))w(x, 0)$ and $0 \leq \widehat{\theta}(x)\leq 1$ for $x \in \R^N$.
Since $\mbox{supp}\, (1 -\chi(\eps x))$ is
bounded in $\R^N$ for any $\eps>0$ and $z_n(\cdot, 0) \to 0$ in $L^q_{\textnormal{loc}}(\R^N)$ as $n \to \infty$ for any $2 \leq q<2^*_s$,
by H\"older's inequality, then
$$
\left|\int_{\R^N} \left(1 -\chi(\eps x)\right) \left(|w_n(x, 0)|^{p-2}
w_n(x, 0)-|w(x, 0)|^{p-2}w(x, 0) \right) z_n(x, 0) \,dx \right|=o_n(1).
$$
Applying again the mean value theorem, we can deduce from Lemma
\ref{mprop} that
\begin{align*}
&\left|\int_{\R^N} \left(1 -\chi(\eps x)\right) \left(|w_n(x ,0)|^{p-2}
(m_{\eps}(|w_n(x ,0)|^2))^{\frac{2^*_s-p}{2}} w_n(x ,0) \right. \right.\\
& \hspace{3cm} \, \left. \left. -|w(x ,0)|^{p-2}(m_{\eps}(|w(x ,0)|^2))^{\frac{2^*_s-p}{2}} w(x, 0)\right) z_n(x, 0) \,dx \right| \\
& \leq  (p-1)  \left |\int_{\R^N} \left(1 -\chi(\eps x)\right) |\theta_n|^{p-2}(m_{\eps}(|\theta_n|^2))^{\frac{2^*_s-p}{2}} |z_n(x, 0)|^2 \, dx \right| \\
& \quad + (2^*_s-p)  \left|\int_{\R^N}  \left(1 -\chi(\eps x)\right) |\theta_n|^{p-2} (m_{\eps}(|\theta_n|^2))^{\frac{2^*_s-p}{2} -1} b_{\eps}(|\theta_n|^2) |\theta_n|^2 |z_n(x, 0)|^2\, dx \right| \\
& \leq \left(2^*_s-1\right)  \int_{\R^N} \left(1 -\chi(\eps x)\right) |\theta_n|^{p-2}(m_{\eps}(|\theta_n|^2))^{\frac{2^*_s-p}{2}} |z_n(x, 0)|^2\, dx \\
& \leq \frac{c\left(2^*_s-1\right)}{\eps^{\frac{2^*_s-p}{2}}}\int_{\R^N} \left(1-\chi(\eps x)\right) |\theta_n|^{p-2}|z_n |^2\, dx,
\end{align*}
from which we can also infer that, for any $\eps>0$,
\begin{align*}
&\left|\int_{\R^N} \left(1 -\chi(\eps x)\right) \left(|w_n(x ,0)|^{p-2}
(m_{\eps}(|w_n(x ,0)|^2))^{\frac{2^*_s-p}{2}} w_n(x ,0) \right. \right.\\
& \hspace{3cm} \, \left. \left. -|w(x ,0)|^{p-2}(m_{\eps}(|w(x ,0)|^2))^{\frac{2^*_s-p}{2}} w(x, 0)\right) z_n(x, 0) \,dx \right| =o_n(1).
\end{align*}
Similarly, we can derive that
\begin{align*}
& \left|\int_{\R^N}  \left(1 -\chi(\eps x)\right) \left( |w_n(x
,0)|^{p} (m_{\eps}(|w_n(x ,0)|^2))^{\frac{2^*_s-p}
{2}-1} b_{\eps}(|w_n(x ,0)|^2)w_n(x ,0) \right.  \right. \\
&\hspace{7.5em} \left. \left. -|w(x ,0)|^{p}(m_{\eps}(|w(x ,0)|^2))^{\frac{2^*_s-p}{2}-1} b_{\eps}(|w(x ,0)|^2)w(x ,0) \right) z_n(x, 0) \,dx \right| =o_n(1).
\end{align*}
Therefore, it holds that
\begin{align} \label{esth}
\hspace{-1cm}\left|\int_{\R^N} \left(1 -\chi(\eps x)\right) \left(h_{\eps}(|w_n(x, 0)|) w_n(x, 0)-h_{\eps}(|w(x, 0)|)w(x, 0) \right) z_n(x, 0) \, dx \right|=o_n(1).
\end{align}
By using \eqref{estg}, \eqref{esth} and the fact that $g_{\eps}(|t|) \leq a/4$ for any $t \in \R$, 
we then conclude from \eqref{conv} that $\|z_n\|=o_n(1)$. This suggests that $\{w_n\}$
has a convergent subsequence in $X^{1, s}(\R^{N+1}_+)$, and the proof is completed.
\end{proof}

Let $\Phi_0: X^{1, s}(\mathcal{C}_{B_1(0)}) \to \R$ be the
functional defined by \eqref{defphi0}  and $\varphi_n: B_n \to
\R^N$ be the function defined by \eqref{defvarphi}. For $j \in \mathbb {N}$, we define
$$
c_{j, \eps}:=\left\{
\begin{array}
[c]{ll} \displaystyle\inf_{B\in \Lambda_j}\sup_{u\in B\setminus
W} \Phi_{\eps}(u),& \mbox{if} \ j\geq 2,\\
\displaystyle\inf_{B\in \Lambda_1}\sup_{u\in
B} \Phi_{\eps}(u), & \mbox{if} \ j= 1,
\end{array}
\right.
$$
and
$$
\widetilde{c}_j:=\left\{
\begin{array}
[c]{ll} \displaystyle\inf_{B \in \widetilde{\Lambda}_j} \sup_{u \in B \backslash W}
 \Phi_0(u),& \mbox{if} \ j\geq 2,\\
\displaystyle\inf_{B\in \widetilde{\Lambda}_1}\sup_{u\in
B} \Phi_0(u), & \mbox{if} \ j= 1,
\end{array}
\right.
$$
where $\Lambda_j$, $G_n$, $\widetilde{\Lambda}_j$ and
$\widetilde{G}_n$ are defined by \eqref{defLa1} and
\eqref{defLa2}, respectively. In addition, there holds that
$c_{1, \eps}>0$, $0< c_{2,\eps} \leq  \dots \leq c_{j, \eps} \leq \cdots$,
$\widetilde{c}_1>0$, $0<\widetilde{c}_2\leq  \cdots \leq \widetilde{c}_j \leq \cdots$
and
$0<c_{j,\eps} \leq \widetilde{c}_j $ for any $j \geq 1$.
\medskip

Applying Theorem \ref{exist} and arguing as the proof of Theorem \ref{cvnbmiif9f8ufjhfy1}, we then have the following theorem.

\begin{thm}\label{cvnbmiif9f8ufjhfy}
For any $k \in \mathbb{N}$, there exists a constant $\eps_k>0$
such that, for any $0<\eps<\eps_k$, \eqref{equ22} admits at least
$k$ pairs of solutions $\pm w_{j, \eps} \in X^{1, s}(\R^{N+1}_+)$
satisfying $\Phi_{\eps}(w_{j, \eps})=c_{j,\eps} \leq
\widetilde{c}_{k}$ for any $1 \leq j \leq k$. Moreover, $w_{j,
\eps}$ is a sign-changing solution to \eqref{equ22} for any $2
\leq j \leq k$.
\end{thm}

\subsection{Exclusion of blow-up of semiclassical states}
In this following, we are going to prove that the solutions $\{w_{j, \eps}\} \subset X^{1, s}(\R^{N+1}_+)$
obtained in Theorem \ref{cvnbmiif9f8ufjhfy} cannot blow up. More precisely, we shall prove the following result.

\begin{prop} \label{bcbvhfyfyufuadx}
Assume $\max \{2,(N+2s)/(N-2s)\}<p<2^*_s$. Let $\{w_{j, \eps}\}
\subset X^{1, s}(\R^{N+1}_+)$ be the solutions obtained in Theorem
\ref{cvnbmiif9f8ufjhfy}, then there exist $M_k>0$ and $\eps_k'>0$
such that, for any $0<\eps<\eps_k'$,
\begin{align*} 
\sup_{(x, y) \in \overline{\R^{N+1}_+}} |w_{j, \eps}(x, y)| \leq M_k.
\end{align*}
\end{prop}

To prove this proposition, we need to establish the following lemmas.

\begin{lem} \label{bounded}
Let $w_{j, \eps} \in X^{1, s}(\R^{N+1}_+)$ be the solution to \eqref{equ22} obtained in Theorem \ref{cvnbmiif9f8ufjhfy}.
Then, for any $k \in \mathbb{N}$ and $0<\eps<\eps_k$, there exist a constant $\rho>0$
depending only on $p, N$ and a constant $\eta_k>0$ independent of $\eps$
such that $\rho \leq \|w_{j, \eps}\|_{1, s} \leq \eta_k$ for any $1 \leq j \leq k$,
where the constant $\eps_k$ is determined in Theorem \ref{cvnbmiif9f8ufjhfy}.
\end{lem}
\begin{proof}
In view of the definition of $f_{\eps}$ and Lemma \ref{mprop},
we have that $|f_{\eps}(x, |t|)| \leq |t|^{p-2}+|t|^{2^*_s-2}$ for any $x \in \R^N$ and $t \in \R$.
Therefore, by using the same way as the proof of Lemma \ref{bddsol}, we can complete the proof.
\end{proof}

\begin{lem}\label{cmvnvhhguf7fuufss}
Let $\{\eps_n\} \subset \R^+$ with $\eps_n \to 0$ as $ n \to
\infty$ and $w_{j, \eps} \in X^{1, s}(\R^{N+1}_+)$ be a solution
to \eqref{equ22} with $\eps=\eps_n$. Then, for any $1 \leq j \leq
k$, there exist a sequence $\{\sigma_{j, i, n} \} \subset \R^+$
satisfying $\lim_{n \to \infty} \sigma_{j ,i, n} =\infty$ and a
sequence $\{x_{j ,i ,n}\} \subset \R^N$ such that $w_{j, \eps_n}$
has a profile decomposition
\begin{align} \label{pd}
w_{j, \eps_n}=\sum_{i \in \Lambda_1} U_{j ,i}(\cdot -(x_{j, i,
n},0)) + \sum_{i \in \Lambda_{\infty}} \sigma_{j, i,
n}^{\frac{N-2s}{2}} U_{j, i}(\sigma_{j, i, n}(\cdot-(x_{j, i,
n},0))) + r_{j, n},
\end{align}
where $\Lambda_{1}$ and $ \Lambda_{\infty}$ are finite sets and
the numbers of the sets are bounded from above by an integer
depending only on $k$. Moreover,
\begin{enumerate}
\item[$(1)$] for any $i \in \Lambda_1$, $w_{j, \eps_n}(\cdot +
(x_{j, i, n},0)) \wto U_{j, i} \neq 0 $ in $X^{1, s}(\R^{N+1}_+)$
as $n \to \infty$ and for any $i \in \Lambda_{\infty}$,
$$
\sigma_{j, i, n}^{-\frac{N-2s}{2}}w_{j, \eps_n}(\sigma_{j, i, n}^{-1}  \cdot + (x_{j, i, n},0)) \wto U_{j, i} \neq 0 \,\, \mbox{in}
\,\, X^s(\R^{N+1}_+) \,\, \mbox{as}\,\,  n \to \infty.
$$
\item[$(2)$] For any $i, i' \in \Lambda_1 $ with $i \neq i'$,
$|x_{j, i, n}-x_{j, i', n}| \to \infty$ as $n \to \infty$ and for
any $i, i' \in \Lambda_{\infty}$ with $i \neq i'$,
$$
\frac{\sigma_{j, i, n}}{\sigma_{j, i', n}}+\frac{\sigma_{j, i',
n}}{\sigma_{j, i, n}}+ \sigma_{j, i, n}\sigma_{j, i', n}|x_{j, i,
n}-x_{j, i', n}|^2 \to \infty \,\, \mbox{as} \,\, n \to \infty.
$$
\item[$(3)$] There holds that
$$
\sum_{i \in \Lambda_1 \cup \Lambda_{\infty}} \int_{\R^N}
|U_{j,i}(x, 0)|^{2^*_s} \, dx \leq \liminf_{n \to \infty}
\int_{\R^N} |w_{j, \eps_n}(x, 0)|^{2^*_s} \, dx.
$$
\item[$(4)$] $r_{j, n}(\cdot, 0) \to 0$ in $L^{2^*_s}(\R^N)$ as $n \to \infty$.
\end{enumerate}
\end{lem}
\begin{proof}
The proof of this lemma is a straightforward adaption of \cite[Theorem 3.3]{SST}, then we omit its proof here.
\end{proof}

\begin{lem}\label{2xncbdggdtdtttd}
If $i\in\Lambda_1$, then, up to a subsequence,
$x_{j,i}:=\lim_{n\rightarrow\infty}\epsilon_nx_{j,i,n}\in
\mathcal{V}^{\delta_0}$ and $U_{j,i}$ satisfies the equation
\begin{align}\label{cdhfdhfgyyfyf66ftf}
\left\{
\begin{aligned}
-\mbox{div} \, (y^{1-2s} \nabla U_{j ,i})&=0 \hspace{5cm}\quad \mbox{in} \,\, \R^{N+1}_+,\\
-k_s \frac{\partial U_{j ,i}}{\partial {\nu}}&=-V(x_{j, i}) U_{j ,i} + f(x_{j ,i},
|U_{j ,i}|) U_{j ,i} \qquad \mbox{on}\,\, \R^N \times \{0\},
\end{aligned}
\right.
\end{align}
where
\begin{align} \label{deff}
f(x, t):=(1-\chi(x))(t^{p-2}+t^{2^*_s-2})+\chi(x)g(t)
\end{align}
and
$$
g(t):=\min\left\{t^{p-2}+t^{2^*_s-2}, a/4\right\} \,\,\, \mbox{for}
\,\,  x \in \R^N \,\, \mbox{and} \,\, t \geq 0.
$$
\end{lem}
\begin{proof}
Since $w_{j,\eps_n}$ satisfies the equation
\begin{align}\label{equ30}
\left\{
\begin{aligned}
-\mbox{div}(y^{1-2s} \nabla w_{j, \eps_n})&=0 \hspace{6cm} \mbox{in} \,\, \R^{N+1}_+,\\
-k_s \frac{\partial w_{j, \eps_n}}{\partial {\nu}}&=-V_{\eps_n}(x) w_{j,\eps_n} + f_{\eps_n}(x, |w_{j,\eps_n}|)  w_{j, \eps_n} \, \qquad \,\, \mbox{on} \,\, \R^N \times \{0\},
\end{aligned}
\right.
\end{align}
then $\widetilde{w}_{j,i,n}:=w_{j, \eps_n}(\cdot+(x_{j,i,n}, 0))$
satisfies the equation
\begin{align*}
\left\{
\begin{aligned}
-\mbox{div}(y^{1-2s} \nabla \widetilde{w}_{j, i, n})&=0 \,\, \, \hspace{8.5cm}\, \mbox{in} \,\, \R^{N+1}_+,\\
-k_s \frac{\partial \widetilde{w}_{j, i, n}}{\partial {\nu}}&=-V_{\eps_n}(x+x_{j,i,n}) \widetilde{w}_{j, i, n} + f_{\eps_n}(x+x_{j,i,n}, |\widetilde{w}_{j, i, n}|) \widetilde{w}_{j, i, n} \ \ \, \, \qquad  \mbox{on} \,\, \R^N \times \{0\}.
\end{aligned}
\right.
\end{align*}
If $\lim_{n\rightarrow\infty}\epsilon_n x_{j,i,n}\notin \mathcal{V}^{\delta_0}$, by the definition of $f_{\eps_n}$ and the fact that $\widetilde{w}_{j, i, n} \wto U_{j ,i}$ in $X^{1, s}(\R^{N+1}_+)$ as $n \to \infty$, see Lemma \ref{cmvnvhhguf7fuufss}, then $U_{j ,i}$ satisfies the equation
\begin{align}\label{equ31}
\left\{
\begin{aligned}
-\mbox{div}(y^{1-2s} \nabla U_{j, i})&=0 \hspace{4cm} \,\,\,\,\mbox{in} \,\, \R^{N+1}_+,\\
-k_s \frac{\partial  U_{j, i}}{\partial {\nu}}&=-V_{\infty} U_{j, i}+ g(|U_{j, i}|)  U_{j, i} \, \qquad \,\, \mbox{on} \,\, \R^N \times \{0\},
\end{aligned}
\right.
\end{align}
where $V_{\infty}:=\liminf_{n \to \infty} V(\eps_n x_{j ,i ,n})
\geq a$. Multiplying \eqref{equ31} by $U_{j, i}$ and integrating
on $\R^{N+1}_+$, we then have that
\begin{align*}
k_s\int_{\R^{N+1}_+} y^{1-2s} |\nabla  U_{j, i}|^2 \, dxdy +
\int_{\R^N} V_{\infty} | U_{j, i}(x, 0)|^2 \, dx &=\int_{\R^N}
g(| U_{j, i}(x, 0)|) | U_{j, i}(x, 0)|^2 \,dx\\
& \leq \frac a 4\int_{\R^N} | U_{j, i}(x, 0)|^2,
\end{align*}
where we used the fact that $g(|t|) \leq a/4$ for any $t \in \R$. This then follows that $U_{j ,i}=0$,
which is impossible, because of $U_{j ,i} \neq 0$, see Lemma \ref{cmvnvhhguf7fuufss}.
Hence there exists $x_{j ,i} \in \mathcal{V}^{\delta_0}$
such that $x_{j,i}=\lim_{n\rightarrow\infty}\epsilon_nx_{j,i,n}\in \mathcal{V}^{\delta_0}$.
Since $\widetilde{w}_{j, i, n} \wto U_{j ,i}$ in $X^{1, s}(\R^{N+1}_+)$ as $n \to \infty$,
by using the definition of $f_{\eps_n}$, we then easily assert that $U_{j ,i}$ satisfies \eqref{cdhfdhfgyyfyf66ftf}.
Thus the proof is completed.
\end{proof}

\begin{lem}\label{7cncbvggftd6ettd66d}
If $i\in\Lambda_1$, then there exists a constant $C>0$ such that
$$
|U_{j,i}(\xi)|\leq \frac{C}{1+|\xi|^{N}}, \quad \xi:=(x, y) \in \overline{\R^{N+1}_+}.
$$
\end{lem}
\begin{proof}
Since $U_{j ,i}$ satisfies \eqref{cdhfdhfgyyfyf66ftf}, it then follows from \cite[Lemma 4.3]{FQT} that
\begin{align} \label{udecay}
|U_{j, i} (x, 0)|\leq \frac{C}{1+|x|^{N+2s}}, \quad x \in \R^N.
\end{align}
Let $v$ be the solution to the equation
\begin{align*}
\left\{
\begin{aligned}
-\mbox{div}(y^{1-2s} \nabla v)&=0 \hspace{2cm} \quad  \mbox{in} \,\, \R^{N+1}_+,\\
v(x,0)&=\frac{1}{1+|x|^{N+2s}} \quad  \ \,\,\, \mbox{on} \,\,
\R^N \times \{0\}.
\end{aligned}
\right.
\end{align*}
From \eqref{udecay} and the maximum principle, we then get that
\begin{align*}
|U_{j,i}(\xi)|\leq C v(\xi),\quad \xi\in\overline{\R^{N+1}_+}.
\end{align*}
At this point, reasoning as the proof of \eqref{mmmmz9999agtrrr4},
we can derive that
$$
|v(\xi)| \leq\frac{C}{1+|\xi|^{N}},\quad\xi\in\overline{\R^{N+1}_+}.
$$
Thus the proof is completed.
\end{proof}

\begin{lem}\label{xncbdggdtdtttd}
If $i\in\Lambda_\infty$, then, up to a subsequence,
$x_{j,i}:=\lim_{n\rightarrow\infty}\epsilon_nx_{j,i,n}\in
\mathcal{V}^{\delta_0}$ and $U_{j,i}$ satisfies the equation
\begin{align}\label{cnvcbv99fufjhhf}
\left\{
\begin{aligned}
-\mbox{div} \, (y^{1-2s} \nabla U_{j ,i})&=0 \hspace{4cm}\qquad \, \, \, \, \, \mbox{in} \,\, \R^{N+1}_+,\\
-k_s \frac{\partial U_{j ,i}}{\partial
{\nu}}&=\left(1-\chi(x_{j,i})\right) \psi_{j ,i}(|U_{j,i}|) U_{j,
i} \qquad \quad \mbox{on}\,\, \R^N \times \{0\}.
\end{aligned}
\right.
\end{align}
where
\begin{align}\label{cmvnbhhgyt664rr}
\psi_{j,i}(t):=\left\{
\begin{array}
[c]{ll} 0, & \mbox{if} \ \varrho_{j,i}=+\infty,\\
 t^{2^*_s-2}, & \mbox{if} \ \varrho_{j,i}=0,\\
\frac{p}{2^*_s}\varrho^{-\frac{2^*_s-p}{2}}_{j,i} t^{p-2}
\left(m(\varrho_{j,i} t^2) \right)^{\frac{2^*_s-p}{2}} \\+
\frac{2^*_s-p}{2^*_s}  \varrho_{j,i}^{1-\frac{2^*_s-p}{2}
}t^{p}\left(m(\varrho_{j,i}t^2) \right)^{\frac{2^*_s-p}{2} -1}
\varphi(\varrho_{j,i}t^2),& \mbox{if} \ 0<\varrho_{j,i}<+\infty,
\end{array}
\right.
\end{align}
and
$$
\varrho_{j,i}:=\lim_{n\rightarrow\infty}\epsilon_n\sigma_{j,i,n}^{N-2s}\in[0,+\infty], \quad m(t):=\int^t_0\varphi(s)\,ds.
$$
\end{lem}
\begin{proof}
Notice that $w_{j, \eps_n}$ satisfies \eqref{equ30}, then
$\widehat{w}_{j,i,n}:=\sigma_{j, i, n}^{-\frac{N-2s}{2}}w_{j,
\eps_n}( \sigma_{j, i, n}^{-1}\cdot +(x_{j, i, n}, 0))$ satisfies
the equation
\begin{align*}
\left\{
\begin{aligned}
-\mbox{div}(y^{1-2s} \nabla \widehat{w}_{j, i, n})&=0 \quad  \hspace{7.5cm} \mbox{in} \,\, \R^{N+1}_+,\\
-k_s \frac{\partial \widehat{w}_{j, i, n}}{\partial {\nu}}&=-\sigma_{j, i, n}^{-2s}V_{\eps_n}(\sigma_{j, i, n}^{-1}x+x_{j, i, n}) \widehat{w}_{j, i, n} \\
& \quad + \sigma_{j, i, n}^{-2s} f_{\eps_n}(\sigma_{j, i, n}^{-1}x+x_{j, i, n}, \sigma_{j, i,
n}^{\frac{N-2s}{2}}|\widehat{w}_{j, i, n}|) \widehat{w}_{j, i, n}
\, \, \,\ \qquad  \, \mbox{on} \,\, \R^N \times \{0\}.
\end{aligned}
\right.
\end{align*}
Since $\sigma_{j ,i ,n} \to +\infty$ and $\widehat{w}_{j, i, n}
\wto U_{j ,i}$ in $X^{1, s}(\R^{N+1}_+)$ as $n \to \infty$, see
Lemma \ref{cmvnvhhguf7fuufss}, we then get that, for any $\psi \in X^{1,
s}(\R^{N+1}_+)$,
\begin{align}\label{vesti}
\sigma_{j, i, n}^{-2s}\int_{\R^N}  V_{\eps_n}(\sigma_{j, i, n}^{-1}x+x_{j, i, n}) \widehat{w}_{j, i, n}(x, 0) \psi (x, 0) \, dx=o_n(1)
\end{align}
and
\begin{align}\label{gesti}
\hspace{-1cm}\sigma_{j, i, n}^{-2s}\int_{\R^N}  \chi(\sigma_{j, i, n}^{-1}x+x_{j, i, n})g_{\eps_n}(\sigma_{j, i,
n}^{\frac{N-2s}{2}}|\widehat{w}_{j, i, n}(x, 0)|) \widehat{w}_{j,
i, n}(x, 0) \psi (x, 0) \, dx=o_n(1),
\end{align}
where we used the assumption that $a \leq V(x) \leq b$ for any $x \in \R^N$ and the fact that $0 \leq g_{\eps_n}(|t|) \leq a/4$ for any $t \in \R$. If $\lim_{n\rightarrow\infty}\epsilon_n x_{j,i,n}\notin \mathcal{V}^{\delta_0}$, by the definition of $\chi$ and the fact that
$$
\sigma_{j, i, n}^{-2s} |h_{\eps_n}(\sigma_{j, i, n}^{\frac{N-2s}{2}}|\widehat{w}_{j, i, n}(x, 0)|) \widehat{w}_{j, i, n}(x, 0)| \leq \sigma_{j, i, n}^{\frac p 2 (N-2s) -N} |\widehat{w}_{j, i, n}(x, 0)|^{p-1} + |\widehat{w}_{j, i, n}(x, 0)|^{2^*_s-1},
$$
we then deduce that, for any $\psi \in X^{1, s}(\R^{N+1}_+)$,
\begin{align*}
\sigma_{j, i, n}^{-2s}\int_{\R^N}  \big(1-\chi(\eps_n (\sigma_{j, i, n}^{-1}x+x_{j, i,
n}))\big)h_{\eps_n}(\sigma_{j, i,
n}^{\frac{N-2s}{2}}|\widehat{w}_{j, i, n}(x, 0)|) \widehat{w}_{j,
i, n}(x, 0) \psi (x, 0) \, dx=o_n(1).
\end{align*}
This together with \eqref{vesti} and \eqref{gesti} implies that $U_{j, i}$ solves the equation
\begin{align*}
\left\{
\begin{aligned}
-\mbox{div}(y^{1-2s} \nabla U_{j,i})&=0 \quad  \,\, \, \mbox{in} \,\, \R^{N+1}_+,\\
-k_s \frac{\partial U_{j,i}}{\partial {\nu}}&=0 \, \, \quad  \mbox{on} \,\, \R^N \times \{0\}.
\end{aligned}
\right.
\end{align*}
Hence we get that $U_{j,i}=0$, which is impossible, see Lemma
\ref{cmvnvhhguf7fuufss}. Therefore, we have that there exists
$x_{j ,i} \in \mathcal{V}^{\delta_0}$ such that
$x_{j,i}=\lim_{n\rightarrow\infty}\epsilon_nx_{j,i,n}\in
\mathcal{V}^{\delta_0}$. Since $2<p<2^*_s$, $\sigma_{j ,i ,n} \to
\infty$ and $\widehat{w}_{j, i, n} \wto U_{j ,i}$ in $X^{1,
s}(\R^{N+1}_+)$ as $n \to \infty$ and for any $t\geq 0,$
$$
\sigma_{j, i, n}^{-2s} h_{\eps_n}(\sigma_{j, i,
n}^{\frac{N-2s}{2}}t) \rightarrow\psi_{j,i}(t) \,\,\, \mbox{as}
\,\, n \to \infty,
$$ 
we then obtain that
$$
\sigma_{j, i, n}^{-2s} h_{\eps_n}(\sigma_{j, i,
n}^{\frac{N-2s}{2}}|\widehat{w}_{j, i, n}(x, 0)|) \to \psi_{j,
i}(|U_{j,i}(x)|) \quad \mbox{a.e. in }  \R^N  \,\,\, \mbox{as}
\,\, n \to \infty.
$$
From Lemma \ref{mprop}, it is easy to see that for any $\varrho_{j,i}\in[0,+\infty],$
\begin{align}\label{icxic11856ttgfr}
0\leq\psi_{j,i}(t)\leq t^{2^*_s-2},\ \forall \, t\in[0,+\infty).
\end{align}
 As a consequence, we derive
that, for any $\psi \in X^{1, s}(\R^{N+1}_+)$,
\begin{align} \nonumber
&\sigma_{j, i, n}^{-2s}\int_{\R^N}  \big(1-\chi(\eps_n (\sigma_{j, i, n}^{-1}x+x_{j, i,
n}))\big)h_{\eps_n}(\sigma_{j, i,
n}^{\frac{N-2s}{2}}|\widehat{w}_{j, i, n}(x, 0)|) \widehat{w}_{j,
i, n}(x, 0) \psi (x, 0) \, dx \\ \label{hesti} &\to \int_{\R^N}
\big(1-\chi(x_{j, i})\big) \psi_{j, i}(|U_{j,i}(x)|)  U_{j, i}(x,
0) \psi(x, 0) \, dx \,\,\, \mbox{as} \,\, n \to \infty.
\end{align}
Utilizing \eqref{vesti}, \eqref{gesti} and \eqref{hesti}, we then get that $U_{j,i}$ satisfies \eqref{cnvcbv99fufjhhf}. This completes the proof.
\end{proof}

\begin{lem}\label{3cncbvggftd6ettd66d}
If $i\in\Lambda_\infty$, then there exists a constant $C>0$ such
that
$$
|U_{j,i}(\xi)|\leq \frac{C}{1+|\xi|^{N-2s}}, \quad \xi:=(x, y) \in \overline{\R^{N+1}_+}.
$$
\end{lem}
\begin{proof}
Let
$$
\widetilde{U}_{j,i}(\xi)=\frac{1}{|\xi|^{N-2s}}U_{j,i}\left(\frac{\xi}{|\xi|^2}\right),\quad \xi \in \overline{\R^{N+1}_+} \backslash \{0\}
$$ 
be the Kelvin transformation of $U_{j,i}$. It then yields from \cite[Proposition A.1]{O} that $\widetilde{U}_{j,i}$
satisfies the equation 
\begin{align*}
\left\{
\begin{aligned}
-\mbox{div} \, (y^{1-2s} \nabla \widetilde{U}_{j ,i})&=0 \quad \hspace{7.5cm}\qquad \, \, \mbox{in} \,\, \R^{N+1}_+,\\
-k_s \frac{\partial \widetilde{U}_{j ,i}}{\partial
{\nu}}&=\left(1-\chi(x_{j,i})\right)|x|^{-4s} \psi_{j
,i}(|x|^{N-2s}|\widetilde{U}(x,0)|) \widetilde{U}_{j, i}(x,0)
\ \ \quad \qquad \, \mbox{on}\,\, \R^N \times \{0\}.
\end{aligned}
\right.
\end{align*}
Let $a(x)=\left(1-\chi(x_{j,i})\right)|x|^{-4s} \psi_{j
,i}(|x|^{N-2s}|\widetilde{U}_{j,i}(x,0)|).$ From
\eqref{icxic11856ttgfr}, we get that
\begin{align}\label{cnvbygt77rydgdfgff}
\lim_{r\rightarrow0^+}\int_{B_r(0)}|a(x)|^{N/2s}\,dx\leq\lim_{r\rightarrow0^+}\int_{B_r(0)}|\widetilde{U}_{j,i}(x,0)|^{2^*_s}\,dx=0.
\end{align}
It then follows from \cite[Lemma 2.8]{JLX} that there exists $r>0$
such that
$$
\|\widetilde{U}_{j,i}(\cdot,0)\|_{L^{q}(B_{r/2}(0))}<+\infty,
$$
where $2N/(N-2s)<q<\min\{2(N+1)/(N-2s),2N^2/(N-2s)^2\}$.
Consequently, we have that 
$$
\|a\|_{L^{q/(2^*_s-2)}(B_{r/2}(0))}<+\infty.
$$
Since $q/(2^*_s-2)>N/2s,$  by \cite[Proposition 2.6 (i)]{JLX}, we then deduce that
$$
\|\widetilde{U}_{j,i}\|_{L^{\infty}(B_{r/4}(0)\times[0,r/4])}<+\infty.
$$
This implies the result of this lemma, and the proof is completed.
\end{proof}

If $\Lambda_{\infty} \neq \emptyset$, let
$i_\infty\in\Lambda_\infty$ be such that
\begin{align}\label{mmmzcczdr4s4s4sss}
\sigma_n:=\sigma_{j,i_\infty,n}=\min\{\sigma_{j,i,n} \ |\
i\in\Lambda_\infty\}.
\end{align}
Let
\begin{align}\label{mcnvbbbvuvjyyfffff}
x_n:=x_{j,i_\infty,n}.
\end{align}
For further discusison, we define $\mathcal{B}_r^+(z):=\{\xi \in \R^{N+1}: |\xi-z|<r\} \cap \R^{N+1}_+$ with $r>0$ and $z \in \R^{N+1}$.
\begin{lem}\label{ncbvjgfu88g8g8f7}
There exists a constant $\overline{C}>0$ such that, up to a
subsequence,
$$
\mathcal{A}^1_{n}\cap \{x_{j,i,n} \ |\
i\in\Lambda_\infty\}=\emptyset,
$$
where
$$
\mathcal{A}^1_{n}:=\mathcal{B}^+_{(\overline{C}+5)\sigma^{-\frac{1}{2}}_n}(z_n)\setminus
\mathcal{B}^+_{\overline{C}\sigma^{-\frac{1}{2}}_n}(z_n), \quad z_n:=(x_n, 0).
$$
\end{lem}
\begin{proof}
The proof of this lemma can be completed by using the similar way as the proof of \cite[Lemma 4.8]{CLW}, then we omit its proof here.
\end{proof}

\begin{lem} \label{esta2}
Let $\{\eps_n\} \subset \R^+$ with $\eps_n \to 0$ as $ n \to
\infty$ and $w_{j,\eps_n} \in X^{1, s}(\R^{N+1}_+)$ be the solution to \eqref{equ22} with $\eps=\eps_n$
obtained in Theorem \ref{cvnbmiif9f8ufjhfy}. Then there exists a constant $C>0$ independent of $n$ such that
$$
|w_{j, \eps_n}(z)| \leq C, \quad z \in \mathcal{A}^2_n,
$$
where
$$
\mathcal{A}^2_{n}:=\mathcal{B}^+_{(\overline{C}+9/2)\sigma^{-\frac{1}{2}}_n}(z_n)\setminus
\mathcal{B}^+_{(\overline{C}+1/2)\sigma^{-\frac{1}{2}}_n}(z_n).
$$
\end{lem}
\begin{proof}
For simplicity, we shall write $w_{j, n}=w_{j, \eps_n}$ and $u_{j,
n}=w_{j, \eps_n}(\cdot,0)$, respectively. Notice that $w_{j,n}$
satisfies the equation
\begin{align*}
\left\{
\begin{aligned}
-\mbox{div}(y^{1-2s} \nabla w_{j, n})&=0 \hspace{5cm} \qquad \,\, \mbox{in} \,\, \R^{N+1}_+,\\
-k_s \frac{\partial w_{j, n}}{\partial {\nu}}&=-V_{\eps_n}(x) w_{j, n} + f_{\eps_n}(x, |w_{j, n}|)  w_{j, n} \quad \qquad \mbox{on} \,\, \R^N \times \{0\}.
\end{aligned}
\right.
\end{align*}
By using \eqref{pd}, Lemmas \ref{2xncbdggdtdtttd} and \ref{xncbdggdtdtttd}, we then see that $r_{j, n}$ satisfies the equation
\begin{align}\label{equrn}
\left\{
\begin{aligned}
-\mbox{div}(y^{1-2s} \nabla r_{j, n})&=0 \hspace{6cm} \,  \, \mbox{in} \,\, \R^{N+1}_+,\\
-k_s \frac{\partial r_{j, n}}{\partial {\nu}}&=-V_{\eps_n}(x) r_{j, n} + f_{\eps_n}(x, |w_{j, n}|)  w_{j, n} \\
& \quad + I_n^1(x)  +I_n^{\infty}(x)  \quad \, \quad \hspace{3cm}
\mbox{on} \,\, \R^N \times \{0\},
\end{aligned}
\right.
\end{align}
where
\begin{align*}
\begin{split}
I^1_n(x)&:=\sum_{i\in\Lambda_1}\left(-f(x_{j,i},|U_{j,i}(x-x_{j,i,n}, 0)|)U_{j,i}(x-x_{j,i,n}, 0) +V(x_{j,i})U_{j,i}(x-x_{j,i,n}, 0)\right.\\
&\quad \qquad \,\, \, \left.-V_{\eps_n}(x)U_{j,i}(x-x_{j,i,n}, 0)\right)
\end{split}
\end{align*}
and
\begin{align*}
\begin{split}
I^\infty_n(x)&:=\sum_{i\in\Lambda_\infty}\left(-\sigma_{j, i,
n}^{\frac{N+2s}{2}}
(1-\chi(x_{j,i}))\psi_{j,i}(U_{j,i}(\sigma_{j,i,n}(x-x_{j,i,n},
0)))
U_{j,i}(\sigma_{j,i,n}(x-x_{j,i,n}, 0)) \right. \\
&\qquad \qquad \left.-\sigma_{j, i, n}^{\frac{N-2s}{2}}
V_{\eps_n}(x) U_{j,i}(\sigma_{j,i,n}(x-x_{j,i,n}, 0))\right).
\end{split}
\end{align*}
By the definition of $f_{\eps}$, we have that $ |f_{\eps}(x, |t|)t| \leq \frac a 2 |t| + C |t|^{2^*_s-1}$ for any $x \in \R^N$ and $t \in \R$. Applying \eqref{pd} again, we then get that
\begin{align*}
|f_{\eps_n}(x, |w_{j, n}|) w_{j, n}| &\leq  \frac a 2 |w_{j, n}| + C |w_{j, n}|^{2^*_s-1} \\
&\leq \frac a 2 |r_{j, n}| + C |r_{j, n}|^{2^*_s-1} + C'\sum_{i \in \Lambda_1} \left(| U_{j,i}(\cdot-x_{j,i,n}, \cdot)| +  |U_{j,i}(\cdot-x_{j,i,n}, \cdot)|^{2^*_s-1}\right) \\
& \quad + C'\sum_{i \in \Lambda_{\infty}}  \left(\sigma_{j, i, n}^{\frac{N-2s}{2}} \left|U_{j,i}(\sigma_{j,i,n}(\cdot-x_{j,i,n}), \cdot)\right|\right. \\
&  \left. \hspace{2cm} + \sigma_{j, i, n}^{\frac{N+2s}{2}} \left|U_{j,i}(\sigma_{j,i,n}(\cdot-x_{j,i,n}), \cdot)\right|^{2^*_s-1} \right).
\end{align*}
By the definition of $f$, we have that $|f(x,
|t|)t| \leq C \left(|t| + |t|^{2^*_s -1}\right)$ for any $x \in
\R^N$ and $t \in \R$. Therefore, 
\begin{align*}
|I^1_n(x)| \leq C \sum_{i\in\Lambda_1} \left( |U_{j,i}(x-x_{j,i,n}, 0)| +|U_{j,i}(x-x_{j,i,n}, 0)|^{2^*_s -1}\right).
\end{align*}
Furthermore, by the definition of $\psi_{j,i}$, see
\eqref{cmvnbhhgyt664rr},  we have that $|\psi_{j,i}( |t|)| \leq
 |t|^{2^*_s -2}$ for any  $t \in
\R$. Therefore, 
\begin{align*}
|I^\infty_n(x)| \leq C \sum_{i\in\Lambda_\infty} \left(\sigma_{j,
i, n}^{\frac{N-2s}{2}} |U_{j,i}(\sigma_{j,i,n}(x-x_{j,i,n}, 0))|
+\sigma_{j, i,
n}^{\frac{N+2s}{2}}|U_{j,i}(\sigma_{j,i,n}(x-x_{j,i,n},
0))|^{2^*_s -1}\right).
\end{align*}

Since $r_{j,n} \in X^{1,s}(\R^{N+1}_+)$ satisfies \eqref{equrn}, from the s-harmonic extension arguments,
Lemma \ref{cnvnbghfyf7yfyftttd} and the assumption that $V(x) \geq a$ for any $x \in \R^N$, we then derive that
\begin{align} \label{equrjn0}
\left\{
\begin{aligned}
&-\mbox{div}(y^{1-2s} \nabla (|r_{j, n}|))\leq0 \hspace{7cm} \,  \, \mbox{in} \,\, \R^{N+1}_+,\\
&(-\Delta)^s |r_{j, n}(x ,0)| + \frac a 2 |r_{j, n}(x ,0)| - C |r_{j, n}(x ,0)|^{2^*_s -2} |r_{j ,n}(x ,0)|\\
&\leq  C' h_n(x) \quad  \hspace{9cm}  \mbox{on} \,\, \R^N \times \{0\}.
\end{aligned}
\right.
\end{align}
where
\begin{align*}
h_n(x):=&\sum_{i \in \Lambda_1} \left(| U_{j,i}(x-x_{j,i,n}, 0)| +  |U_{j,i}(x-x_{j,i,n}, 0)|^{2^*_s-1}\right) \\
& \quad  + \sum_{i \in \Lambda_{\infty}}\left(\sigma_{j, i,
n}^{\frac{N+2s}{2}} \left|U_{j,i}(\sigma_{j,i,n}(x-x_{j,i,n},
0))\right|^{2^*_s-1} + \sigma_{j, i, n}^{\frac{N-2s}{2}}
\left|U_{j,i}(\sigma_{j,i,n}(x-x_{j,i,n}, 0)) \right| \right).
\end{align*}
Next we shall prove that there exists $C>0$ independent of $n$ such that $|r_{j,n}(x, 0)| \leq C$ for any $x \in \R^N$.
Let $r_{j ,n}^{(1)} \in X^{1,s}(\R^{N+1}_+)$ be such that
\begin{align} \label{equrn1}
\left\{
\begin{aligned}
&-\mbox{div}(y^{1-2s} \nabla r^{(1)}_{j, n})=0 \hspace{6cm} \qquad \,  \, \mbox{in} \,\, \R^{N+1}_+,\\
&(-\Delta)^s r_{j, n}^{(1)}(x, 0) + \frac a 2 r_{j, n}^{(1)}(x, 0)
- C |r_{j, n}(x ,0)|^{2^*_s -2} r_{j, n}^{(1)}(x, 0)\\
&=C'\sum_{i \in \Lambda_1} \left(| U_{j,i}(x-x_{j,i,n}, 0)| +
|U_{j,i}(x-x_{j,i,n}, 0)|^{2^*_s-1}\right) \qquad \quad \quad \ \ \mbox{on} \,\,
\R^N \times \{0\},
\end{aligned}
\right.
\end{align}
It follows from Lemma \ref{cmvnvhhguf7fuufss} that $\||r_{j,
n}(\cdot, 0)|^{2^*_s -2}\|_{L^{N/2s}(\R^N)} = \|r_{j, n}(\cdot,
0)\|_{L^{2^*_s}(\R^N)} \to 0$ as $n \to \infty$. Using
\eqref{equrn1}, we then infer that $r_{j,n}^{(1)}(x, 0) \geq 0$
for any $x \in \R^N$. Since $U_{j, i}(\cdot, 0) \in
L^{\infty}(\R^N)$ for $i \in \Lambda_1$ (see Lemma
\ref{7cncbvggftd6ettd66d}) and $\||r_{j, n}(\cdot, 0)|^{2^*_s
-2}\|_{L^{N/2s}(\R^N)}   \to 0$ as $n \to \infty$, by applying
Lemma \ref{cmvnbghyyt755534e}, we then have that there exists $C>0$
independent of $n$ such that $r_{j,n}^{(1)}(\xi) \leq C$ for any
$\xi=(x,y) \in \overline{\R^{N+1}_+}$. Accordingly, we get that
\begin{align} \label{r1}
0 \leq r_{j,n}^{(1)}(\xi) \leq C, \quad \xi \in
\overline{\R^{N+1}_+}.
\end{align}
For $i \in \Lambda_{\infty}$, let $T_{j, i, n} \in X^s(\R^{N+1}_+)$ be such that
\begin{align} \label{equt}
\hspace{-1.5cm}\left\{
\begin{aligned}
&-\mbox{div}(y^{1-2s} \nabla T_{j,i, n})=0 \hspace{4cm}\hspace{2cm}\qquad \quad \, \,\, \,\, \, \mbox{in} \,\, \R^{N+1}_+,\\
&(-\Delta)^s T_{j, i, n}(x, 0)-C|\widehat{r}_{j,i,n}(x, 0)|^{2^*_s-2} T_{j,i, n}(x, 0)=C'|U_{j,i}(x, 0)|^{2^*_s-1} \,\quad \quad \mbox{on} \,\, \R^N
\times \{0\},
\end{aligned}
\right.
\end{align}
where
$$
\widehat{r}_{j,i,n}(\xi):=\sigma_{j, i, n}^{-\frac{N-2s}{2}}
r_{j,n}(\sigma_{j, i, n}^{-1} \xi + (x_{j,i,n}, 0)).
$$
According to Lemma \ref{3cncbvggftd6ettd66d}, we know that $U_{j,i}(\cdot, 0) \in L^{\infty}(\R^N)$ for $i \in \Lambda_{\infty}$. Moreover, there holds that $\|\widehat{r}_{j,i,n}(\cdot, 0)\|_{L^{2^*_s}(\R^N)}=\|r_{j,n}(\cdot, 0)\|_{L^{2^*_s}(\R^N)} \to 0$ as $n \to \infty$.
By arguing as before, we can deduce that there exists $C>0$ independent of $n$ such that
\begin{align} \label{t1}
0 \leq T_{j,i, n}(\xi) \leq C, \quad \xi \in
\overline{\R^{N+1}_+}.
\end{align}
We now define the Kelvin transform of $T_{j,i,n}$ by
$$
\widetilde{T}_{j,i,n}(\xi):=|\xi|^{-(N-2s)}T_{j,i,n}\left(\frac{\xi}{|\xi|^2}\right),
\quad \xi \in \overline{\R^{N+1}_+} \backslash \{0\}.
$$
Since $T_{j,i,n}$ satisfies \eqref{equt}, by \cite[Proposition A.1]{O}, then
\begin{align*}
\left\{
\begin{aligned}
&-\mbox{div}(y^{1-2s} \nabla \widetilde{T}_{j,i, n})=0 \hspace{7cm}\qquad \,\,\, \mbox{in} \,\, \R^{N+1}_+,\\
&(-\Delta)^s \widetilde{T}_{j, i, n}(x, 0)-C|\widetilde{\widehat{r}}_{j,i,n}(x ,0)|^{2^*_s-2} \widetilde{T}_{j,i, n}(x, 0)=C'|\widetilde{U}_{j,i}(x, 0)|^{2^*_s-1}\, \ \quad\quad \quad
\mbox{on} \,\, \R^N \times \{0\},
\end{aligned}
\right.
\end{align*}
where $\widetilde{U}_{j ,i} (\cdot, 0)$ denotes the Kelvin transform of ${U}_{j ,i} (\cdot, 0)$. From Lemma \ref{3cncbvggftd6ettd66d}, we know that
$\widetilde{U}_{j ,i} (\cdot, 0)\in L^{\infty}(\R^N)$ for $i \in
\Lambda_{\infty}$. In addition, from Lemma
\ref{cmvnvhhguf7fuufss}, we find that
\begin{align*}
\||\widetilde{\widehat{r}}_{j,i,n}(\cdot,
0)|^{2^*_s-2}\|_{L^{N/2s}(\R^N)}
=\|\widetilde{\widehat{r}}_{j,i,n}(\cdot,
0)\|_{L^{2^*_s}(\R^N)}=\|\widehat{r}_{j,i,n}(\cdot,
0)\|_{L^{2^*_s}(\R^N)} =o_n(1).
\end{align*}
 Thus we can conclude that there exists $C>0$
independent of $n$ such that
\begin{align} \label{t2}
0 \leq \widetilde{T}_{j,i, n}(\xi) \leq C, \quad \xi \in
\overline{\R^{N+1}_+}.
\end{align}
Combining \eqref{t1} and \eqref{t2}, we then have that there exists $C>0$ independent of $n$ such that
\begin{align} \label{t3}
0\leq {T}_{j,i, n}(\xi) \leq \frac{C}{1+|\xi|^{N-2s}}, \quad \xi
\in \overline{\R^{N+1}_+}.
\end{align}
Define
$$
r_{j ,n}^{(2)}(\xi):=\sum_{i \in \Lambda_{\infty}}\sigma_{j, i,
n}^{\frac{N-2s}{2}} T_{j,i,
n}(\sigma_{j,i,n}(\xi-(x_{j,i,n},0))), \quad \xi=(x,y)\in
\overline{\R^{N+1}_+}.
$$
It then follows from \eqref{equt} that
\begin{align} \label{equr2}
\left\{
\begin{aligned}
&-\mbox{div}(y^{1-2s} \nabla r^{(2)}_{j, n})=0 \hspace{5cm}\quad \quad\quad  \, \mbox{in} \,\, \R^{N+1}_+,\\
&(-\Delta)^s r_{j, n}^{(2)}(x, 0) + \frac a 2 r_{j, n}^{(2)}(x, 0) - C |r_{j, n}(x ,0)|^{2^*_s -2} r_{j, n}^{(2)}(x, 0)\\
&\geq C' \sum_{i \in \Lambda_{\infty}} \sigma_{j, i,
n}^{\frac{N+2s}{2}}
\left|U_{j,i}(\sigma_{j,i,n}(x-x_{j,i,n}),0)\right|^{2^*_s-1}
\quad \quad\quad \,\hspace{1cm} \ \ \mbox{on} \,\, \R^N \times \{0\}.
\end{aligned}
\right.
\end{align}
Let $\varphi: \R^N \to \R$ satisfy $\varphi \in C^{\infty}_0(\R^N, \R)$, $\varphi \geq 0$
and $\varphi \not\equiv 0$ in $\R^N$. For $i \in \Lambda_{\infty}$, let $W_{j, i, n} \in X^{s}(\R^{N+1}_+)$ be such that
\begin{align} \label{equw}
\left\{
\begin{aligned}
&-\mbox{div}(y^{1-2s} \nabla W_{j,i, n})=0 \hspace{5cm}\quad \quad \,\,\, \, \mbox{in} \,\, \R^{N+1}_+,\\
&(-\Delta)^s  {W}_{j, i, n}(x, 0)-C|\widehat{r}_{j,i,n}(x ,0)|^{2^*_s-2} {W}_{j,i, n}(x, 0)=\varphi(x)\quad \qquad \mbox{on} \,\, \R^N
\times \{0\}.
\end{aligned}
\right.
\end{align}
Since $\||r_{j, n}(\cdot, 0)|^{2^*_s -2}\|_{L^{N/2s}(\R^N)} \to 0$
as $n \to \infty$, then $W_{j,i,n}(x) \geq 0$ for any $x \in
\R^N$. It then yields from \eqref{equw} that $(-\Delta)^s W_{j, i,
n}(x, 0) \geq 0$. Since  $|x|^{-(N-2s)}$ is a fundamental solution
of $(-\Delta)^su=0$ in $\R^N\setminus\{0\}$ (see \cite[Section
2.2]{CS}), by the maximum principle (see \cite[Theorem
1]{CL}), we then infer that there exists $C>0$ independent of $n$ such that
\begin{align} \label{estw1}
W_{j,i,n}(x,0) \geq \frac{C}{1+|x|^{N-2s}}, \quad x \in \R^N.
\end{align}
On the other hand, using the Kelvin transform of $W_{j,i,n}$, we can deduce as
before that there exists $C>0$ independent of $n$ such that
\begin{align} \label{estw2}
W_{j,i,n}(\xi) \leq \frac{C}{1+|\xi|^{N-2s}}, \quad \xi \in
\overline{\R^{N+1}_+}.
\end{align}
Hence, by \eqref{estw1} and Lemma \ref{3cncbvggftd6ettd66d}, there
exists a constant $\rho>0$ such that for $i \in
\Lambda_{\infty}$,
\begin{align} \label{wuest}
\frac a 2 \rho W_{j,i,n}(x,0) \geq C' U_{j, i}(x,0), \quad x \in
\R^N.
\end{align}
Define
$$
r_{j ,n}^{(3)}(\xi):=\sum_{i \in \Lambda_{\infty}}\sigma_{j, i,
n}^{\frac{N-2s}{2}} \rho W_{j,i, n}(\sigma_{j,i,n}(\xi-(x_{j,i,n},
0))), \quad \xi=(x,y) \in \overline{\R^{N+1}_+}.
$$
It follows from \eqref{equw} and \eqref{wuest} that
\begin{align} \label{equr3}
\left\{
\begin{aligned}
&-\mbox{div}(y^{1-2s} \nabla r^{(3)}_{j, n})=0 \,\,\hspace{4cm}\qquad \quad\quad\quad \quad\quad  \, \mbox{in} \,\, \R^{N+1}_+,\\
&(-\Delta)^s r_{j, n}^{(3)}(x, 0) + \frac a 2 r_{j, n}^{(3)}(x, 0) - C |r_{j, n}(x ,0)|^{2^*_s -2} r_{j, n}^{(3)}(x, 0)\\
&\geq C' \sum_{i \in \Lambda_{\infty}} \sigma_{j, i,
n}^{\frac{N-2s}{2}} \left|U_{j,i}(\sigma_{j,i,n}(x-x_{j,i,n}),
0)\right| \hspace{2cm} \,\, \qquad \quad\quad \mbox{on} \,\, \R^N \times \{0\}.
\end{aligned}
\right.
\end{align}
Taking into account \eqref{equrjn0}, \eqref{equrn1},
\eqref{equr2}, \eqref{equr3} and the maximum principle, we then
have that
\begin{align*}
|r_{j, n}(\xi)|  \leq |r_{j,  n}^{(1)}(\xi)| +  |r_{j,
n}^{(2)}(\xi)| +  |r_{j,  n}^{(3)}(\xi)|, \quad \xi \in
\overline{\R^{N+1}_+}.
\end{align*}
Using \eqref{r1}, \eqref{t3} and \eqref{estw2}, we then get that
\begin{align*}
|r_{j, n}(\xi)|  \leq C + C \sum_{i \in
\Lambda_{\infty}}\sigma_{j, i, n}^{\frac{N-2s}{2}} \left(1 +
\sigma_{j,i,n}|\xi-(x_{j,i,n},0)|\right)^{-(N-2s)}, \quad \xi \in
\overline{\R^{N+1}_+}.
\end{align*}
By means of \eqref{pd}, Lemmas \ref{7cncbvggftd6ettd66d} and
\ref{3cncbvggftd6ettd66d}, we then obtain that
\begin{align} \label{estiu}
|w_{j, n}(\xi)|  \leq C + C \sum_{i \in
\Lambda_{\infty}}\sigma_{j, i, n}^{\frac{N-2s}{2}} \left(1 +
\sigma_{j,i,n}|\xi-(x_{j,i,n},0)|\right)^{-(N-2s)}, \quad \xi \in
\overline{\R^{N+1}_+}.
\end{align}
It follows from Lemma \ref{ncbvjgfu88g8g8f7} that
$$
\sigma_{j,i,n}|z-(x_{j,i,n},0)| \geq \frac 12 \sigma_{j,i,n}
\sigma_n^{-\frac 12} \geq \frac 12 \sigma_{j,i,n}^{\frac 12},
\quad z \in \mathcal{A}_n^2.
$$
This along with \eqref{estiu} gives the result of the lemma, and
the proof is completed.
\end{proof}

\begin{lem} \label{esta3}
Let $\{\eps_n\} \subset \R^+$ with $\eps_n \to 0$ as $ n \to
\infty$ and $w_{j,\eps_n} \in X^{1, s}(\R^{N+1}_+)$ be the solution to \eqref{equ22} with $\eps=\eps_n$
obtained in Theorem \ref{cvnbmiif9f8ufjhfy}. Then there exists a constant $C>0$ independent of $n$ such that
$$
\int_{\mathcal{A}_n^3} y^{1-2s} |w_{j,\eps_n}|^2\, dxdy \leq C \sigma_n^{-\frac{N+2-2s}{2}},
$$
where
$$
\mathcal{A}^3_{n}:=\mathcal{B}^+_{(\overline{C}+4)\sigma^{-\frac{1}{2}}_n}(z_n)\setminus
\mathcal{B}^+_{(\overline{C}+1)\sigma^{-\frac{1}{2}}_n}(z_n).
$$
\end{lem}
\begin{proof}
By Lemma \ref{esta2}, we get that
\begin{align*}
\int_{\mathcal{A}_n^3} y^{1-2s} |w_{j,\eps_n}|^2\, dxdy & \leq
C\int_{\mathcal{A}_n^3} y^{1-2s} \, dxdy \leq
C\int_{\mathcal{B}^+_{(\overline{C}+4)\sigma^{-\frac{1}{2}}_n}(z_n)}
y^{1-2s} \, dxdy\nonumber\\
&= C\int_{\mathcal{B}^+_{(\overline{C}+4)\sigma^{-\frac{1}{2}}_n}(0)}
y^{1-2s} \,dxdy \nonumber\\
&=C\int^{(\overline{C}+4)\sigma^{-\frac{1}{2}}_n}_{0} r^{N+1-2s}\,dr=C\sigma_n^{-\frac{N+2-2s}{2}},
\end{align*}
where $r=\sqrt{|x|^2+y^2}$. Thus we have completed the proof.
\end{proof}

\begin{lem} \label{esta4}
Let $\{\eps_n\} \subset \R^+$ with $\eps_n \to 0$ as $ n \to \infty$ and $w_{j,\eps_n} \in X^{1, s}(\R^{N+1}_+)$ be the solution to \eqref{equ22} with $\eps=\eps_n$ obtained in Theorem \ref{cvnbmiif9f8ufjhfy}. Then there exists a constant $C>0$ independent of $n$ such that
$$
\int_{\mathcal{A}_n^4} y^{1-2s} |\nabla w_{j,\eps_n}|^2\, dxdy \leq C \sigma_n^{-\frac{N-2s}{2}},
$$
where
$$
\mathcal{A}^4_{n}:=\mathcal{B}^+_{(\overline{C}+3)\sigma^{-\frac{1}{2}}_n}(z_n)\setminus
\mathcal{B}^+_{(\overline{C}+2)\sigma^{-\frac{1}{2}}_n}(z_n).
$$
\end{lem}
\begin{proof}
Let $\eta_{n} \in C_0^{\infty}(\R^{N+1}_+, [0, 1])$ be a cut-off
function satisfying $\eta_n(z)=1$ for $ z \in
{\mathcal{A}^4_n}$, $ \eta_n(z)=0$ for $z \not \in
\mathcal{A}^3_n$ and $|\nabla\eta_{n}(\xi)|\leq C\sigma^{1/2}_n$
for  $z \in \mathcal{A}^3_n$, where the constant $C>0$ is
independent of $n$. Note that $w_{j,\eps_n}$ satisfies the equation
\begin{align}\label{equ3000}
\left\{
\begin{aligned}
-\mbox{div}(y^{1-2s} \nabla w_{j, \eps_n})&=0 \hspace{5cm} \qquad \quad \,\, \mbox{in} \,\, \R^{N+1}_+,\\
-k_s \frac{\partial w_{j, \eps_n}}{\partial {\nu}}&
=-V_{\eps_n}(x) w_{j, \eps_n} + f_{\eps_n}(x, |w_{j, \eps_n}|)  w_{j, \eps_n}\quad \qquad \,\mbox{on} \,\, \R^N \times \{0\}.
\end{aligned}
\right.
\end{align}
Multiplying \eqref{equ3000} by $\eta_n^2w_{j,\eps_n} $ and
integrating on $\R^{N+1}_+$, we have that
\begin{align*}
\int_{\R^{N+1}_+} y^{1-2s} \nabla w_{j, \eps_n} \cdot \nabla (\eta_n^2w_{j, \eps_n} ) \, dxdy
&= \int_{\R^N} \left(-V_{\eps_n}(x) |u_{j, \eps_n}|^2  + f_{\eps_n}(x,|u_{j, \eps_n}|) |u_{j, \eps_n}|^2 \right) \eta_n^2(x, 0) \, dx \\
& \leq C \int_{\mathcal{A}_n^3} (|u_{j, \eps_n}|^2  + |u_{j,
\eps_n}|^{2_s^*}) \, dx,
\end{align*}
where we used the assumption that $a \leq V(x) \leq b$ for any $x \in \R^N$ and the fact that $|f_{\eps_n}(x, |t|) \leq C \left(1 + |t|^{2^*-2}\right)$ for any $x \in \R^N$ and $t \in \R$. Using Lemmas \ref{esta2} and \ref{esta3}, we then infer that
\begin{align*}
\int_{\mathcal{A}_n^4}y^{1-2s} |\nabla w_{j, \eps_n}|^2 \, dxdy & =\int_{\mathcal{A}_n^4}y^{1-2s} |\nabla w_{j, \eps_n}|^2 \eta_n^2\, dxdy\\
& \leq C \int_{\mathcal{A}_n^4} y^{1-2s} |w_{j, \eps_n}|^2 |\nabla \eta_n|^2\, dxdy \\
& \quad+ C\int_{\mathcal{A}_n^4} y^{1-2s} \nabla w_{j, \eps_n} \cdot \nabla (\eta_n^2w_{j, \eps_n} ) \, dxdy  \\
&\leq C \sigma_n \int_{\mathcal{A}_n^3} y^{1-2s} |w_{j,\eps_n}|^2\, dxdy + C \int_{\mathcal{A}_n^3} (|u_{j, \eps_n}|^2  + |u_{j, \eps_n}|^{2_s^*}) \, dx \\
& \leq C \sigma_n^{-\frac{N-2s}{2}}+ C \sigma_n^{-\frac{N}{2}}.
\end{align*}
This then gives rise to the result of the lemma, and the proof is completed.
\end{proof}

\begin{lem} \label{ph}
Let $\{\eps_n\} \subset \R^+$ with $\eps_n \to 0$ as $ n \to \infty$ and $w_{j,\eps_n} \in X^{1, s}(\R^{N+1}_+)$
be the solution to \eqref{equ22} with $\eps=\eps_n$ obtained in Theorem \ref{cvnbmiif9f8ufjhfy}.
Let $\mathcal{B}_n:=\mathcal{B}^+_{(\overline{C}+3)\sigma^{-{1}/{2}}_n}(z_n)$, $\phi\in C^\infty_0(\mathbb{R}, [0, 1])$
be such that $\phi(t)=1$ for $t\leq (\overline{C}+2)\sigma^{-{1}/{2}}_n$,
$\phi(t)=0$ for $t\geq (\overline{C}+3)\sigma^{-{1}/{2}}_n$, $\phi'(t)\leq 0$ and
$|\phi'(t)|\leq 2\sigma^{{1}/{2}}_n$ for $t \in \R$. Define
$\varphi(\xi):=\phi(|\xi-z_n|)$ for $\xi \in \R^{N+1}_+$.
Then the following identity holds,
\begin{align} \label{locph} 
\begin{split}
&\frac{1}{k_s} \int_{\R^N} \left(N F_{\eps_n}(x, |u_{j,
\eps_n}|)-\frac{N-2s}{2} f_{\eps_n}(x, |u_{j, \eps_n|}) |u_{j,
\eps_n}|^2\right) \varphi \, dx \\
& \quad -\frac {s} {k_s} \int_{\R^N}
V_{\eps_n}(x) |u_{j, \eps_n}|^2  \varphi(x, 0) \, dx \\
& \quad -\frac {1} {2k_s} \int_{\R^N}
\left(\left(x-x_n\right)\cdot \nabla V_{\eps_n}(x)\right)  |u_{j,
\eps_n}|^2  \varphi(x, 0) \, dx \\ 
& \quad +\frac{1}{k_s} \int_{\R^N}\left(\left(x-x_n\right)\cdot
\left(\nabla_x F_{\eps_n}\right)(x, |u_{j, \eps_n}|)\right)
\varphi(x, 0) \, dx \\ 
&=\frac 12 \int_{\R^{N+1}_+} y^{1-2s}  |\nabla w_{j,\eps_n}|^2 \left(\left(\xi-z_n\right)\cdot
\nabla \varphi \right) \, dxdy \\
& \quad-\int_{\R^{N+1}_+} y^{1-2s} \left(\nabla w_{j,\eps_n} \cdot \nabla \varphi \right) \left(\left(\xi-z_n\right)\cdot \nabla w_{j,\eps_n} \right)\, dxdy\\ 
& \quad +\frac {1}{2k_s} \int_{\R^N} V_{\eps_n}(x) |u_{j,\eps_n}|^2  \left(\left(x-x_n\right)\cdot \nabla  \varphi(x, 0) \right) \, dx\\ 
& \quad- \frac {1}{k_s}\int_{\R^N} F_{\eps_n}(x, |u_{j,\eps_n}|) \left(\left(x-x_n\right)\cdot \nabla \varphi(x, 0) \right)  \,dx \\
& \quad -\frac{N-2s}{2} \int_{\R^{N+1}_+} y^{1-2s} w_n \left(\nabla w_{j,\eps_n} \cdot \nabla \varphi \right) \, dxdy.
\end{split}
\end{align}
\end{lem}
\begin{proof}
For simplicity, we shall write $w_n=w_{j,\eps_n}$ and $u_n=u_{j,\eps_n}$, where $u_{j,\eps_n}=w_{j,\eps}(\cdot, 0)$. Note that $w_n$ satisfies the equation
\begin{align}\label{equph}
\left\{
\begin{aligned}
-\mbox{div}(y^{1-2s} \nabla w_n)&=0 \, \hspace{5cm}\mbox{in} \,\, \R^{N+1}_+,\\
-k_s \frac{\partial w_n}{\partial {\nu}}&=-V_{\eps_n}(x) w_n + f_{\eps_n}(x, |w_n|) w_n \, \, \, \qquad  \mbox{on} \,\, \R^N \times \{0\}.
\end{aligned}
\right.
\end{align}
Multiplying \eqref{equph} by $\left(\left(\xi-z_n\right)\cdot \nabla w_n\right) \varphi$ and integrating on $\R^{N+1}_+$, we get that
\begin{align} \label{idph1}
\int_{\R^{N+1}_+}-\mbox{div}(y^{1-2s} \nabla w_n) \left(\left(\xi-z_n\right)\cdot \nabla w_n\right) \varphi \, dxdy=0.
\end{align}
Using the divergence theorem, we then derive from \eqref{idph1} that
\begin{align} \label{div1}
\begin{split}
&\int_{\R^{N+1}_+}(y^{1-2s} \nabla w_n) \cdot \nabla \left(\left(\left(\xi-z_n\right)\cdot \nabla w_n \right)\varphi \right) \, dxdy \\
&=-\int_{\R^N} \lim_{ y \to 0^+} \left(y^{1-2s} \partial_y {w_n}\right)\left(\left(x-x_n\right)\cdot \nabla u_n\right) \varphi(x, 0) \, dx \\
&=\frac{1}{k_s} \int_{\R^N}\left(f_{\eps_n}(x, |u_n|) u_n-V_{\eps_n}(x) u_n\right)\left(\left(x-x_n\right)\cdot \nabla u_n\right) \varphi(x, 0) \, dx.
\end{split}
\end{align}
We are now going to compute every term in \eqref{div1}. By straightforward calculations, we see that
\begin{align} \label{gterm}
\begin{split}
&\int_{\R^{N+1}_+}(y^{1-2s} \nabla w_n) \cdot \nabla \left(\left(\left(\xi-z_n\right)\cdot \nabla w_n  \right)\varphi\right) \, dxdy \\
&=\int_{\R^{N+1}_+} y^{1-2s} |\nabla w_n|^2 \varphi \, dxdy \\
&\quad + \frac 12 \int_{\R^{N+1}_+} y^{1-2s} \left(\left(\xi-z_n\right)\cdot \nabla|\nabla w_n|^2\right) \varphi \, dxdy \\
&\quad +\int_{\R^{N+1}_+} y^{1-2s} \left(\nabla w_n \cdot  \nabla \varphi \right) \left(\left(\xi-z_n\right)\cdot \nabla w_n \right)\, dxdy \\
& =-\frac{N-2s}{2}\int_{\R^{N+1}_+} y^{1-2s} |\nabla w_n|^2 \varphi \, dxdy \\
& \quad -\frac 12 \int_{\R^{N+1}_+} y^{1-2s}  |\nabla w_n|^2 \left(\left(\xi-z_n\right)\cdot \nabla \varphi \right) \, dxdy\\
&\quad +\int_{\R^{N+1}_+} y^{1-2s} \left(\nabla w_n \cdot  \nabla \varphi \right) \left(\left(\xi-z_n\right)\cdot \nabla w_n \right)\, dxdy.
\end{split}
\end{align}
In addition, we have that
\begin{align}\label{vterm}
\begin{split}
&\int_{\R^N} V_{\eps_n}(x) u_n\left(\left(x-x_n\right)\cdot \nabla u_n\right) \varphi(x, 0) \, dx \\
&=\frac 1 2 \int_{\R^N} V_{\eps_n}(x) \left(\left(x-x_n\right)\cdot \nabla |u_n|^2\right) \varphi(x, 0) \, dx  \\
& = -\frac N 2 \int_{\R^N} V_{\eps_n}(x) |u_n|^2  \varphi(x, 0) \, dx \\
& \quad - \frac {1} {2} \int_{\R^N} \left(\left(x-x_n\right)\cdot \nabla V_{\eps_n}(x)\right)  |u_n|^2  \varphi(x, 0) \, dx \\
& \quad -\frac 12 \int_{\R^N} V_{\eps_n}(x) |u_n|^2  \left(\left(x-x_n\right)\cdot \nabla  \varphi(x, 0) \right) \, dx
\end{split}
\end{align}
and
\begin{align}\label{fterm}
\begin{split}
&\int_{\R^N}f_{\eps_n}(x, |u_n|) u_n\left(\left(x-x_n\right)\cdot \nabla u_n\right) \varphi(x, 0) \, dx \\
& = \int_{\R^N} \left(\left(x-x_n\right)\cdot \nabla \left( F_{\eps_n}(x, |u_n|)\right)\right) \varphi(x, 0) \, dx \\
& \quad -\int_{\R^N}\left(\left(x-x_n\right)\cdot \left(\nabla_x F_{\eps_n}\right)(x, |u_n|)\right) \varphi(x, 0) \, dx \\
&=-N \int_{\R^N} F_{\eps_n}(x, |u_n|) \varphi(x, 0) \, dx \\
& \quad - \int_{\R^N} F_{\eps_n}(x, |u_n|) \left(\left(x-x_n\right)\cdot \nabla \varphi(x, 0) \right)  \, dx \\
& \quad - \int_{\R^N}\left(\left(x-x_n\right)\cdot \left(\nabla_x F_{\eps_n}\right)(x, |u_n|)\right) \varphi(x, 0) \, dx.
\end{split}
\end{align}
Making use of \eqref{gterm}, \eqref{vterm} and \eqref{fterm}, we then obtain from \eqref{div1} that
\begin{align} \label{div2} 
\begin{split}
& -\frac{N-2s}{2}\int_{\R^{N+1}_+} y^{1-2s} |\nabla w_n|^2 \varphi \, dxdy -\frac {N} {2k_s} \int_{\R^N} V_{\eps_n}(x) |u_n|^2  \varphi(x, 0) \, dx \\ 
& \quad -\frac {1} {2k_s} \int_{\R^N} \left(\left(x-x_n\right)\cdot \nabla V_{\eps_n}(x)\right)  |u_n|^2  \varphi(x, 0) \, dx +\frac{N}{k_s}\int_{\R^N} F_{\eps_n}(x, |u_n|) \varphi(x, 0) \, dx \\ 
&\quad + \frac{1}{k_s} \int_{\R^N}\left(\left(x-x_n\right)\cdot \left(\nabla_x F_{\eps_n}\right)(x, |u_n|)\right) \varphi(x, 0) \, dx \\
&=\frac 12 \int_{\R^{N+1}_+} y^{1-2s}  |\nabla w_n|^2 \left(\left(\xi-z_n\right)\cdot \nabla \varphi \right) \, dxdy \\
& \quad -\int_{\R^{N+1}_+} y^{1-2s} \left(\nabla w_n \cdot  \nabla \varphi \right) \left(\left(\xi-z_n\right)\cdot \nabla w_n \right)\, dxdy\\
& \quad +\frac {1}{2k_s} \int_{\R^N} V_{\eps_n}(x) |u_n|^2  \left(\left(x-x_n\right)\cdot \nabla  \varphi(x, 0) \right) \, dx
\\ 
& \quad - \frac{1}{k_s} \int_{\R^N} F_{\eps_n}(x, |u_n|) \left(\left(x-x_n\right)\cdot \nabla \varphi(x, 0) \right)  \, dx.
\end{split}
\end{align}
On the other hand, by multiplying \eqref{equph} by $w_n \varphi$ and integrating on $\R^{N+1}_+$, we get that
\begin{align} \label{div3} 
\begin{split}
&\int_{\R^{N+1}_+} y^{1-2s} |\nabla w_n|^2 \varphi \, dxdy + \int_{\R^{N+1}_+} y^{1-2s} \left(\nabla w_n \cdot \nabla \varphi \right) w_n\, dxdy \\
&+  \frac{1}{k_s} \int_{\R^N} V_{\eps_n}(x) |u_n|^2 \varphi(x, 0) \, dx = \frac{1}{k_s} \int_{\R^N} f_{\eps_n}(x, |u_n|) |u_n|^2 \varphi(x, 0) \, dx.
\end{split}
\end{align}
Combining \eqref{div2} and \eqref{div3}, we then obtain \eqref{locph}. This completes the proof.
\end{proof}

\begin{lem} \label{lames}
If $\max \{2, (N+2s)/(N-2s)\}<p<2^*_s$, then
$\Lambda_{\infty}=\emptyset$.
\end{lem}
\begin{proof}
We argue by contradiction that $\Lambda_\infty\neq\emptyset$. Then we can choose $x_n$ and $\sigma_n$
satisfying \eqref{mmmzcczdr4s4s4sss} and \eqref{mcnvbbbvuvjyyfffff}, respectively. From Lemma \ref{ph},
we see that $\text{supp} \, \nabla \varphi \subset \mathcal{A}_n^4$ and $|\nabla \varphi(x)| \leq 2 \sigma_n^{1/2}$
for any $x \in \mathcal{A}_n^4$. In addition, by the definition of $F_{\eps_n}$,
we find that $|F_{\eps_n}(x, |t|)| \leq C\left(|t|^2+|t|^{2^*_s}\right)$
for any $x \in \R^N$ and $t \in \R$. Taking advantage of Lemmas \ref{esta2}, \ref{esta3}
and \ref{esta4}, we then get that the integrals of the right side hand of \eqref{locph} can be estimated as
\begin{align} \label{estiph1} 
\begin{split}
&\left|\frac 12 \int_{\R^{N+1}_+} y^{1-2s}  |\nabla w_{j,\eps_n}|^2 \left(\left(\xi-z_n\right)\cdot \nabla \varphi \right) \, dxdy \right.\\ 
& \quad -\int_{\R^{N+1}_+} y^{1-2s} \left(\nabla w_{j,\eps_n} \cdot  \nabla \varphi \right) \left(\left(\xi-z_n\right)\cdot \nabla w_{j,\eps_n} \right)\, dxdy \\ 
& \quad + \frac {1}{2k_s} \int_{\R^N} V_{\eps_n}(x) |u_{j,\eps_n}|^2  \left(\left(x-x_n\right)\cdot \nabla  \varphi(x, 0) \right) \, dx  \\
& \quad -\frac{1}{k_s}\int_{\R^N} F_{\eps_n}(x, |u_{j,\eps_n}|) \left(\left(x-x_n\right)\cdot \nabla \varphi(x, 0) \right)  \, dx\\
&\quad - \left. \frac{N-2s}{2} \int_{\R^{N+1}_+} y^{1-2s} w_{j,\epsilon_n} \left(\nabla w_{j,\eps_n} \cdot \nabla \varphi \right) \, dxdy\right| \\
& \leq C \int_{\mathcal{A}_n^4} y^{1-2s}  |\nabla w_{j,\eps_n}|^2  \, dxdy
+ C\int_{A_n^4}(|u_{j, \eps_n}|^2 + |u_{j ,\eps_n}|^{2^*_s})\, dx  \\
& \quad + C \sigma_n \int_{\mathcal{A}_n^4} y^{1-2s} |w_{j, \eps_n}|^2 \,dxdy \\
& \leq C \sigma_n^{-\frac{N-2s}{2}},
\end{split}
\end{align}
where $A_n^4:=\mathcal{A}_n^4 \cap \R^N$. We next estimate the integrals in the left side hand of \eqref{locph}.
From Lemma \ref{esta2} and the fact that $\text{supp} \, \varphi(\cdot, 0) \subset B_n:=\mathcal{B}_n \cap \R^N$, it follows that
\begin{align} \label{estiph2}
\left|\frac {s} {k_s} \int_{\R^N} V_{\eps_n}(x) |u_{j, \eps_n}|^2  \varphi(x, 0) \, dx\right|
\leq C \sigma_n^{-\frac N 2}
\end{align}
and
\begin{align} \label{estiph3}
\left|\frac {1} {2k_s} \int_{\R^N} \left(\left(x-x_n\right)\cdot \nabla V_{\eps_n}(x)\right)  |u_{j, \eps_n}|^2  \varphi(x, 0) \, dx\right|  \leq C \sigma_n^{-\frac N 2}.
\end{align}
Furthermore, from Lemma \ref{esta2}, the fact that $\text{supp} \, \varphi(\cdot, 0) \subset B_n$
and $|\left(\nabla_x F_{\eps_n}\right)(x, |t|)| \leq C\left(|t|^2+|t|^{2^*_s}\right)$
for any $x \in \R^N$ and $t \in \R$, we infer that
\begin{align} \label{estiph4}
\left|\frac{1}{k_s} \int_{\R^N}\left(\left(x-x_n\right)\cdot \left(\nabla_x F_{\eps_n}\right)(x, |u_{j, \eps_n}|)\right) \varphi(x, 0) \, dx \right|
\leq C \sigma_n^{-\frac N 2}.
\end{align}
Applying the definitions of $f_{\eps_n}$ and $F_{\eps_n}$, we have that
\begin{align}\label{mmvcnvii8g9fff}
\begin{split}
&\int_{\R^N}\left(2^*_s F_{\epsilon_n}(x, |u_{j,\eps_n}|) -
f_{\epsilon_n}(x, |u_{j,\eps_n}|)|u_{j,\eps_n}|^2 \right)
\varphi(x, 0) \, dx \\
&=\int_{\R^N}\left(1- \chi(\epsilon_n
x)\right) \left( \frac {2^*_s\mu}{p}|u_{j,\eps_n}|^p +
|u_{j,\epsilon_n}|^{p} \left(m_{\epsilon_n}(|u_{j,\eps_n}|^2)
\right)^{\frac{2^*_s-p}{2}}
\right)\varphi(x, 0) \, dx \\
&\quad-\int_{\R^N}\left(1- \chi(\epsilon_n
x)\right)\left(\mu|u_{j,\eps_n}|^{p} + \frac{p}{2^*_s}
|u_{j,\eps_n}|^{p} \left(m_{\epsilon_n}(|u_{j,\eps_n}|^2)
\right)^{\frac{2^*_s-p}{2}}\right) \varphi(x, 0) \, dx \\
&\quad-\frac{2^*_s-p}{2^*_s} \int_{\R^N}\left(1- \chi(\epsilon_n x)\right) |u_{j,\eps_n}|^{p+2} \left(m_{\epsilon_n}(|u_{j,\eps_n}|^2) \right)^{\frac{2^*_s-p}{2} -1} b_{\epsilon_n}(|u_{j,\eps_n}|^2) \varphi(x, 0) \, dx \\ 
&\quad+\int_{\R^N} \chi(\epsilon_n x) \left(2^*_s G_{\eps_n}(|u_{j,\eps_n}|)-g_{\eps_n}( |u_{j,\eps_n}|)|u_{j,\eps_n}|^2\right) \varphi(x, 0) \, dx.
\end{split}
\end{align}
In view of the definition of $g_{\eps_n}$, Lemma \ref{esta2} and the fact that $\text{supp} \, \varphi(\cdot, 0) \subset B_n$, we get that
\begin{align} \label{gestimate}
\left|\int_{\R^N}\chi(\epsilon_n x) \left(2^*_s
G_{\eps_n}(|u_{j,\eps_n}|)-g_{\eps_n}(
|u_{j,\eps_n}|)|u_{j,\eps_n}|^2\right) \varphi(x, 0) \, dx \right|
\leq C\sigma^{-\frac{N}{2}}_n.
\end{align}
Note that $tb_\epsilon(t)\leq m_\epsilon(t)$ for any $t \geq 0$, see Lemma \ref{mprop}.
It then yields from \eqref{mmvcnvii8g9fff} and \eqref{gestimate} that
\begin{align*}
&\int_{\R^N}\left(2^*_s F_{\epsilon_n}(x, |u_{j,\eps_n}|) -
f_{\epsilon_n}(x, |u_{j, \eps_n}|)|u_{j,\eps_n}|^2 \right)
\varphi(x, 0) \, dx  \\
&\geq\frac
{(2^*_s-p)\mu}{p}\int_{\R^N}\left(1- \chi(\epsilon_n x)\right)
|u_{j,\eps_n}|^p \varphi(x, 0)\, dx -C\sigma^{-\frac N 2}_n.
\end{align*}
According to Lemma \ref{xncbdggdtdtttd}, we have that, up to a subsequence,
$x_{j,i_\infty}:=\lim_{n\rightarrow\infty}\epsilon_nx_{n}\in
\mathcal{M}^{\delta_0}$. It then gives that
$1-\chi(x_{j,i_\infty})>0$. As a consequence, for any $n \in \mathbb{N}$ large
enough,
\begin{align} \label{estiph5}
\begin{split}
&\int_{\R^N}\left(2^*_s F_{\epsilon_n}(x, |u_{j, \eps_n}|) -f_{\epsilon_n}(x, |u_{j,\eps_n}|)|u_{j,\eps_n}|^2 \right) \varphi(x, 0) \, dx \\
&\geq\frac {(2^*_s-p)\mu}{2p}(1-\chi(x_{j,i_\infty}))\int_{\R^3}|u_{j,
\eps_n}|^p \varphi(x, 0) \, dx-C\sigma^{-\frac N 2}_n.
\end{split}
\end{align}
Using \eqref{locph}, \eqref{estiph1}-\eqref{estiph4} and \eqref{estiph5}, we then arrive at
\begin{align}\label{cnvbvhhfyf6vyyvyv}
\int_{\R^N}|u_{j, \eps_n}|^p \varphi(x, 0) \, dx \leq C\sigma^{-\frac {N-2s} {2}}_n.
\end{align}

We now denote $G_n:=B_{L\sigma^{-1}_n}(x_n)$, where the constant $L>0$ is large enough
such that
\begin{align*}
\int_{B_L(0)}|U_{j, i_\infty}(x, 0)|^{p} \, dx \geq \frac{C_*}{2}, \quad C_*:=\int_{\R^N}|U_{j,i_\infty}(x ,0)|^{p}\,dx.
\end{align*}
Thus, for any $n \in \mathbb{N}$ large enough,
$$
G_n\subset B_{(\overline{C}+2)\sigma^{-\frac{1}{2}}_n}(x_n)\subset B_n
$$
and $\varphi(x, 0)=1$ for any $x \in G_n$. Since
$\widehat{w}_{j,\eps_n}:=\sigma^{-\frac{N-2s}{2}}_n
w_{j,\eps_n}(\sigma^{-1}_n\cdot+(x_{n}, 0))\rightharpoonup
U_{j,i_\infty}$ in $X^s(\R^{N+1}_+)$ as $n \to \infty$, see Lemma
\ref{cmvnvhhguf7fuufss}, then, for any $n \in \mathbb{N}$ large
enough,
\begin{align}\label{ncbvuuf78f77fesd} \nonumber
\int_{B_n}|u_{j, \eps_n}|^p\varphi(x, 0) \, dx & \geq \int_{G_n}|u_{j, \eps_n}|^p \, dx
=\sigma^{\frac{N-2s}{2}p-N}_n\int_{B_L(0)}|\sigma^{-\frac{N-2s}{2}}_{n}u_{j, \eps_n}(\sigma^{-1}_nx+x_n)|^p\, dx \\ 
&=\sigma^{\frac{N-2s}{2}p-N}_n\int_{B_L(0)}|U_{j,i_\infty}(x, 0)|^p \, dx +
o_n(1)\sigma^{\frac{N-2s}{2}p-N}_n \\ \nonumber
&\geq\frac{C_*}{4}\sigma^{\frac{N-2s}{N}p-N}_n.
\end{align}
Combining  \eqref{cnvbvhhfyf6vyyvyv} and \eqref{ncbvuuf78f77fesd} leads to
\begin{align*}
\sigma^{\frac{N-2s}{2}p-N}_n\leq C\sigma^{-\frac {N-2s} {2}}_n.
\end{align*}
It is a contradiction if $p>\frac{N+2s}{N-2s}$, because $\sigma_n
\to \infty$ as $n \to \infty$. This in turn indicates that
$\Lambda_{\infty}= \emptyset$, and the proof is completed.
\end{proof}

We are now ready to prove Proposition \ref{bcbvhfyfyufuadx}.
\medskip

\noindent{\bf Proof of Proposition \ref{bcbvhfyfyufuadx}.} From Lemma \ref{lames},
we know that $\Lambda_{\infty}=\emptyset$. It then yields from \eqref{pd} that
\begin{align} \label{prebdd}
w_{j, \eps_n}=\sum_{i \in \Lambda_1} U_{j ,i}(\cdot -x_{j, i, n}, \cdot) + r_{j, n}.
\end{align}
Furthermore, by means of \eqref{equrjn0}, we then see that
\begin{align*}
(-\Delta)^s |r_{j, n}(x ,0)| + \frac a 2 |r_{j, n}(x ,0)| - C |r_{j, n}(x ,0)|^{2^*_s -2} |r_{j ,n}(x ,0)| \leq C' h_n(x),
\end{align*}
where
\begin{align*}
h_n(x):=&\sum_{i \in \Lambda_1} \left(| U_{j,i}(x-x_{j,i,n}, 0)| +  |U_{j,i}(x-x_{j,i,n}, 0)|^{2^*_s-1}\right).
\end{align*}
Using the same arguments as the proof of \eqref{r1}, we then have
that there exists a constant $C''>0$ independent of $n$ such that
$|r_{j, n}(x, 0)| \leq C''$ for any $x \in \R^N$. Applying
\eqref{convolution}, we then derive that $|r_{j, n}(x, y)| \leq
C''$ for any $(x, y) \in \R^{N+1}_+$. From \eqref{prebdd} and the
fact that $U_{j,i} \in L^{\infty}(\R^N)$, see Lemma
\ref{7cncbvggftd6ettd66d}, the result of this proposition follows.
This completes the proof. \hfill$\Box$

\subsection{Proof of Theorem \ref{wejgh77rtff11}} With the help of Proposition \ref{bcbvhfyfyufuadx}, we are now able to present proof of Theorem \ref{wejgh77rtff11}.

\noindent{\bf Proof of Theorem \ref{wejgh77rtff11}.} In view of Proposition
\ref{bcbvhfyfyufuadx}, we get that there exists a constant  $\epsilon''_k>0$ such
that, for any $0<\epsilon<\epsilon''_k$,
\begin{align}\label{bcvvuuvyctttc}
m_\epsilon(|u_{j,\epsilon}|^2)=|u_{j,\epsilon}|^2, \quad
b_\epsilon(|u_{j,\epsilon}|^2)=1.
\end{align}
Since $\Lambda_\infty=\emptyset$, see Lemma \ref{lames}, by using \eqref{pd},
we then find that the situation under consideration is similar to the case of the Sobolev subcritical equations treated in Section \ref{subcritical}.
Then we can obtain that, for any $\delta>0$, there exist a constant $C=C(\delta, N)>0$ such that
\begin{align*}
|u_{j,\epsilon}(x)|\leq\frac{C}{1+|\mbox{dist}(x,(\mathcal{V}^\delta)_\epsilon)|^{N+2s}},
\quad x\in \mathbb{R}^N.
\end{align*}
By the definition of $f_{\eps}$, we then have that $ f_\epsilon(x,
|u_{j,\epsilon}|)=\mu
|u_{j,\epsilon}|^{p-2}+|u_{j,\epsilon}|^{2^*_s-2}$. This infers
that $u_{j,\epsilon}$ is actually solution to \eqref{equ1} for any
$j \geq 1$. From Theorem \ref{cvnbmiif9f8ufjhfy}, we know that
$u_{j, \eps}$ is sign-changing solution to \eqref{equ1} for any $j
\geq 2$. Arguing as the proof of Theorem \ref{wejgh77rtff111}, we
are able to show that $u_{1, \eps}$ is positive solution to
\eqref{equ1}. Making a change of variable, we then obtain the
result of Theorem \ref{wejgh77rtff11}. Thus the proof is
completed. \hfill$\Box$

\section{Sobolev supercritical growth} \label{supercritical}

We devote this section to the study of semiclassical states to
\eqref{equ11} in the Sobolev supercritical case, i.e.
$\mathcal{N}(u)u= |u|^{p-2} u+\lambda |u|^{r-2}u$ for
$2<p<2^*_s<r$ and $\lambda>0$. Let us first introduce some notations. We set
\begin{align*}
g_{\lambda,K}(t):=\min\left\{h_{\lambda, K}(t), \, a/4\right\}  \quad
\mbox{for} \,\, t \geq 0,
\end{align*}
where
$$
h_{\lambda,K}(t):=\left\{
\begin{array}
[c]{ll} \displaystyle t^{p-2} + \lambda t^{r-2}, & \ \, 0 \leq t \leq K,\\
\displaystyle t^{p-2} + \lambda K^{r-p}t^{p-2}, & \ \qquad t \geq K,
\end{array}
\right.
$$
and $K>0$ is a constant to be determined later. Let $\eta:\R^+ \to [0,1]$ be the cut-off function given in Section \ref{subcritical}. Define $\chi(x):=\eta(\mbox{dist}(x, \, \Lambda))$,
\begin{align*}
f_{\eps, \lambda,K}(x, t):=\left(1- \chi(\eps x)\right)
h_{\lambda, K}(t)+\chi(\eps x){g}_{\lambda, K}(t)
\end{align*}
and
\begin{align*}
F_{\eps, \lambda,K}(x, t):=\int_{0}^t f_{\eps, \lambda, K}(x, s)s \,
ds \quad \mbox{for}\,\, x \in \R^N, t \in \R.
\end{align*}
We now introduce a modified equation as
\begin{align}\label{equ23}
\left\{
\begin{aligned}
-\mbox{div}(y^{1-2s} \nabla w)&=0 \quad \hspace{4cm}\,\, \ \ \mbox{in} \,\, \R^{N+1}_+,\\
-k_s \frac{\partial w}{\partial {\nu}}&=-V_{\eps}(x) w + f_{\eps,
\lambda,K}(x, |w|) w \quad \quad \,\,\mbox{on} \,\, \R^N \times \{0\}.
\end{aligned}
\right.
\end{align}
The energy functional associated to \eqref{equ23} is defined by
$$
\Phi_{\eps, \lambda,K}(w):=\frac {k_s}{2} \int_{\R^{N+1}_+}
y^{1-2s} |\nabla w|^2 \, dxdy + \frac 12 \int_{\R^N}
V_{\eps}(x)|w(x,0)|^2 \, dx - \int_{\R^N} F_{\eps, \lambda,K}(x,
|w(x, 0)|) \, dx.
$$
It is easy to check that $\Phi_{\eps, \lambda,K} \in C^1(X^{1,
s}(\R^{N+1}_+), \R)$ and
\begin{align*}
\Phi_{\eps, \lambda,K}'(w) \psi&=k_s\int_{\R^{N+1}_+} y^{1-2s} \nabla w \cdot \nabla \psi \, dxdy + \int_{\R^N} V_{\eps}(x) w(x, 0) \psi(x, 0) \, dx \\
& \quad - \int_{\R^N} f_{\eps, \lambda,K}(x, |w(x, 0)|) w(x, 0)
\psi(x, 0) \, dx
\end{align*}
for any $\psi \in X^{1, s}(\R^{N+1}_+)$. Consequently, any
critical point of $\Phi_{\eps,\lambda,K}$ corresponds to a
solution to \eqref{equ23}.

Let $\Phi_0: X^{1, s}(\mathcal{C}_{B_1(0)}) \to \R$ be the
functional defined by \eqref{defphi0}  and $\varphi_n: B_n \to
\R^N$ be the function defined by \eqref{defvarphi}. For $j \in
\mathbb {N}$, we define
$$
c_{j, \eps, \lambda,K}:=\left\{
\begin{array}
[c]{ll} \displaystyle\inf_{B\in \Lambda_j}\sup_{u\in B\setminus
W} \Phi_{\eps, \lambda,K}(u),& \mbox{if} \ j\geq 2,\\
\displaystyle\inf_{B\in \Lambda_1}\sup_{u\in B} \Phi_{\eps,
\lambda,K}(u), & \mbox{if} \ j= 1,
\end{array}
\right.
$$
and
$$
\widetilde{c}_j:=\left\{
\begin{array}
[c]{ll} \displaystyle\inf_{B \in \widetilde{\Lambda}_j} \sup_{u \in B \backslash W}
 \Phi_0(u),& \mbox{if} \ j\geq 2,\\
\displaystyle\inf_{B\in \widetilde{\Lambda}_1}\sup_{u\in
B} \Phi_0(u), & \mbox{if} \ j= 1,
\end{array}
\right.
$$
where $\Lambda_j$, $G_n$, $\widetilde{\Lambda}_j$ and
$\widetilde{G}_n$ are defined by \eqref{defLa1} and
\eqref{defLa2}, respectively. Therefore, there holds that
$c_{1, \eps, \lambda,K}>0$, $0< c_{2,\eps, \lambda,K} \leq  \dots
\leq c_{j, \eps, \lambda,K} \leq \cdots$, $\widetilde{c}_1>0$,
$0<\widetilde{c}_2\leq  \cdots \leq \widetilde{c}_j \leq \cdots$
and $0<c_{j,\eps, \lambda,K} \leq \widetilde{c}_j $ for any $j
\geq 1$.

\begin{thm}\label{cvnbmiif9f8ufjhfy111}
For any $k \in \mathbb{N}$ and $\lambda>0$, there exists a
constant $\eps_{k, \lambda,K}>0$ such that, for any
$0<\eps<\eps_{k,\lambda,K}$, \eqref{equ23} admits at least $k$
pairs of solutions $\pm w_{j, \eps, \lambda,K} \in X^{1,
s}(\R^{N+1}_+)$ satisfying $\Phi_{\eps, \lambda,K}(w_{j, \eps,
\lambda,K})=c_{j ,\eps, \lambda,K} \leq \widetilde{c}_{k}$ for any
$1 \leq j \leq k$. Moreover, $w_{j, \eps, \lambda,K}$ is a
sign-changing solution to \eqref{equ23} for any $2 \leq j \leq k$.
\end{thm}
\begin{proof}
The proof of this theorem is almost identical to the one of Theorem \ref{cvnbmiif9f8ufjhfy1}, then we omit its proof.
\end{proof}

\begin{lem} \label{moseriteration}
Let $w_{j,\eps, \lambda,K} \in X^{1, s}(\R^{N+1}_+)$ be the
solution to \eqref{equ23} obtained in Theorem
\ref{cvnbmiif9f8ufjhfy111}. Then there exists $K_0>0$ such that,
for any $k \in \N$ and $K \geq K_0$, there exists $\lambda_{k, K}>0$
such that, for any $0 < \lambda <\lambda_{k, K}$ and
$0<\eps<\eps_{k,\lambda,K}$,
\begin{align} \label{unibdd111}
\sup_{(x, y) \in \overline{\R^{N+1}_+}} |w_{j, \eps, \lambda,K}(x,
y)| \leq K.
\end{align}
\end{lem}
\begin{proof}
From Theorem \ref{cvnbmiif9f8ufjhfy111}, we have that $\Phi_{\eps, \lambda, K}'(w_{j, \eps, \lambda, K})w_{j, \eps, \lambda, K}=0$ and $\Phi_{\eps, \lambda}(w_{j, \eps, \lambda, K})=c_{j ,\eps, \lambda, K} \leq \widetilde{c}_{k}$ for any $1 \leq j \leq k$. Using the definition of $f_{\eps, \lambda}$ and arguing as the proof of the boundedness of the Palais-Smale sequence in Lemma \ref{ps1} ,
we can deduce that there exists $C_k>0$ independent of $\lambda$ such that $\|w_{j,\eps, \lambda, K}\|_{1,s} \leq C_k$ for any $1 \leq j \leq k$.

We now employ the Moser iteration arguments to prove
\eqref{unibdd111}. For simplicity, we shall write $w_{\eps,
\lambda, K}=w_{j,\eps,\lambda, K}$. For $L>0$, we define
\begin{align*}
w_{\eps, \lambda, K, L}(x ,y):=\left\{
\begin{aligned}
&|w_{\eps,\lambda, K}(x, y)|, \quad \mbox{if} \,\, |w_{\eps, \lambda, K}(x, y)| \leq L,\\
&L, \hspace{1.5cm}\,\,\ \ \ \ \mbox{if} \,\, |w_{\eps, \lambda, K}(x, y)| > L.
\end{aligned}
\right.
\end{align*}
Let $\psi_{\eps, \lambda, K, L}:=w_{\eps,\lambda, K, L}^{2(\beta -1)} w_{\eps, \lambda, K}$, where $\beta>1$ is a constant to be determined later. Multiplying \eqref{equ23} by $\psi_{\eps, \lambda, L, K}$ and integrating on $\R^{N+1}_+$, we obtain that
\begin{align*}
&\int_{\R^{N+1}_+} y^{1-2s} \nabla w_{\eps,\lambda, K}\cdot \nabla \psi_{\eps,\lambda, K, L} \, dxdy + \int_{\R^N} V_{\eps} (x) |w_{\eps,\lambda, K}(x, 0)|^2 w_{\eps,\lambda, K, L}^{2(\beta -1)}(x, 0)\, dx\\
&= \int_{\R^N} f_{\eps, \lambda, K}(x, |w_{\eps, \lambda, K}(x, 0)|) |w_{\eps,\lambda, K}(x, 0)|^2 w_{\eps,\lambda, K, L}^{2(\beta -1)}(x, 0) \, dx.
\end{align*}
By the definition of $f_{\eps, \lambda, K}$, we know that $|f_{\eps, \lambda, K}(x, |t|)| \leq \frac a 4 + \left(1+\lambda K^{r-p}\right) |t|^{p-2}$ for any $x \in \R^N$ and $t \in \R$. Therefore, by the assumption that $V(x) \geq a$ for any $x \in \R^N$, we get that
\begin{align} \label{it1}
\int_{\R^{N+1}_+} y^{1-2s} \nabla w_{\eps,\lambda, K}\cdot \nabla \psi_{\eps,\lambda, K, L} \, dxdy \leq C_{\lambda, K} \int_{\R^N} |w_{\eps, \lambda, K}(x, 0)|^{p} w_{\eps,\lambda, K, L}^{2(\beta -1)} (x, 0)\, dx,
\end{align}
where $C_{\lambda, K}:=1+\lambda K^{r-p}$. Observe that
$$
\nabla w_{\eps,\lambda, K} \cdot \nabla \psi_{\eps,\lambda, K, L}= w_{\eps, \lambda, K, L}^{2(\beta-1)} |\nabla w_{\eps,\lambda, K}|^2 + 2(\beta -1) w_{\eps,\lambda, K, L}^{2(\beta-2)} w_{\eps,\lambda, K, L} w_{\eps,\lambda, K} \nabla w_{\eps,\lambda, K, L} \cdot \nabla w_{\eps,\lambda, K}
$$
and
\begin{align*}
|\nabla (w_{\eps,\lambda, K, L}^{\beta -1} w_{\eps,\lambda, K})|^2&=w_{\eps,\lambda, K, L}^{2(\beta-1)} |\nabla w_{\eps,\lambda, K}|^2 + 2(\beta -1) w_{\eps,\lambda, K, L}^{2(\beta-2)} w_{\eps,\lambda, K, L} w_{\eps, \lambda, K} \nabla w_{\eps,\lambda, K, L} \cdot \nabla w_{\eps,\lambda, K} \\
& \quad + (\beta -1)^2 w_{\eps,\lambda, K, L}^{2(\beta -2)} |w_{\eps,\lambda, K}|^2 |\nabla w_{\eps,\lambda, K, L}|^2.
\end{align*}
Thus we conclude that
$$
|\nabla (w_{\eps,\lambda, K, L}^{\beta -1} w_{\eps,\lambda, K})|^2 \leq \beta^2 \left(\nabla w_{\eps,\lambda, K} \cdot \nabla z_{\eps,\lambda, K, L}\right).
$$
It then follows from \eqref{it1} that
$$
\int_{\R^{N+1}_+} y^{1-2s} |\nabla (w_{\eps,\lambda, K, L}^{\beta -1} w_{\eps,\lambda, K})|^2 \, dxdy \leq C_{\lambda, K}\beta^2\int_{\R^N} |w_{\eps,\lambda, K}(x, 0)|^{p} w_{\eps,\lambda, K, L}^{2(\beta -1)}(x, 0) \, dx.
$$
Let $z_{\eps, \lambda, K, L}:=w_{\eps,\lambda, K, L}^{\beta -1} w_{\eps, \lambda, K}$. Using Lemma \ref{embedding1} and H\"older's inequality, we then have that
\begin{align*}
\left(\int_{\R^N} |z_{\eps, \lambda, K, L}(x, 0)|^{2^*_s} \, dx\right)^{\frac{2}{2^*_s}} & \leq C_{\lambda, K}\beta^2\int_{\R^N} |w_{\eps,\lambda, K}(x, 0)|^{p} w_{\eps,\lambda, K, L}^{2(\beta -1)}(x, 0) \, dx\\
& \leq C_{\lambda, K} \beta^2 \left(\int_{\R^N}|w_{\eps, \lambda, K} (x, 0) |^{2^*_s}\, dx\right)^{\frac{p-2}{2^*_s}} \|z_{\eps, \lambda, K, L}(\cdot, 0)\|_{L^{\alpha_s^*}}^2,
\end{align*}
where 
$$
\alpha^*_s:=\frac{22^*_s}{2^*-p+2}<2^*_s.
$$
Letting $L \to \infty$ and applying Fatou's Lemma, we then get that
\begin{align} \label{it2}
\|w_{\eps, \lambda, K}(\cdot, 0)\|_{L^{2^*_s \beta}(\R^N)} \leq \left(C_k C_{\lambda, K}\right)^{\frac{1}{2\beta}} \beta^{\frac{1}{\beta}} \|w_{\eps, \lambda, K}(\cdot, 0)\|_{L^{\alpha^*_s \beta}(\R^N)},
\end{align}
where we used Lemma \ref{embedding1} and the fact that $\|w_{\eps, \lambda, K}\|_{1,s} \leq C_k$.
Setting $\beta:=2^*_s/\alpha^*_s>1$ and iterating $m$ times, we then derive from\eqref{it2} that
$$
\|w_{\eps, \lambda, K}(\cdot, 0)\|_{L^{2^*_s \beta^m}(\R^N)} \leq \left(C_k C_{K, \lambda} \right)^{\beta_{1, m}} \beta^{\beta_{2, m}} \|w_{\eps, \lambda, K}(\cdot, 0)\|_{L^{2^*_s}(\R^N)},
$$
where
$$
\beta_{1, m}:=\sum_{i=1}^{m} \frac{1}{2\beta^{i}}, \quad \beta_{2, m}:=\sum_{i=1}^{m}\frac{i}{\beta^i}.
$$
Letting $m \to \infty$, we then have that
$$
\|w_{\eps, \lambda, K}(\cdot, 0)\|_{L^{\infty}(\R^N)} \leq \left(C_k C_{\lambda, K}\right)^{\beta_{1}} \beta^{\beta_{2}},
$$
where
$$
\beta_{1}:=\sum_{i=1}^{\infty} \frac{1}{2\beta^{i}}<\infty, \quad \beta_{2}:=\sum_{i=1}^{\infty}\frac{i}{\beta^i}<\infty.
$$
From the definition of $C_{\lambda, K}$, then there exists $K_0>0$ such that, for any $k \in \N$ and $K \geq K_0$, there exists $\lambda_{k, K}>0$ such that, for any $0 < \lambda <\lambda_{k, K}$, it holds that $\|w_{\eps, \lambda, K}(\cdot, 0)\|_{L^{\infty}(\R^N)} \leq K$. Using \eqref{convolution}, we then obtain the result of the lemma. Thus the proof is completed.
\end{proof}

\subsection{Proof of Theorem \ref{wejgh77rtff1}} As an immediate consequence of Lemma \ref{unibdd111}, we are able to show the proof of Theorem \ref{wejgh77rtff1}.

\noindent{\bf Proof of Theorem \ref{wejgh77rtff1}.}  Benefitting from Lemma \ref{unibdd111}, we find that
$$
|f_{\eps,\lambda, K}(x, |w_{j,\eps, \lambda, K}(\cdot, 0)|)| \leq C |w_{j,\eps, \lambda, K}(\cdot, 0)|^{p-2}.
$$
This suggests that the problem we consider is reduced to the Sobolev subcritical case. Then we can closely follow the outline for the proof of Theorem \ref{wejgh77rtff111} to complete the proof.
\hfill$\Box$

\section {Appendix}\label{appendix}

\begin{lem}\label{cmvnbghyyt755534e}
Let $w\in X^{1, s}(Q_1(0))$ be a nonnegative solution to the equation
\begin{align} \label{equ61}
\left\{
\begin{aligned}
-\mbox{div}(y^{1-2s} \nabla w)&=0  \hspace{3cm} \,\,\,\mbox{in} \,\, Q_1(0),\\
-k_s \frac{\partial w}{\partial {\nu}}&=-a(x) w + b(x)
\qquad \quad \mbox{on} \,\, B_1(0),
\end{aligned}
\right.
\end{align}
where $a\in L^{N/2s}(B_1(0))$, $b \in L^\infty(B_1(0))$ and $Q_r(0):=B_r(0) \times (0, r)$ for $r >0$.
If there exists a constant $\delta>0$ such that $\|a\|_{L^{N/2s}(B_1(0))}<\delta$, then
$$
\|w\|_{L^{\infty}(Q_{1/2}(0))}\leq C(\|w\|_{L^{2}(y^{1-2s}, Q_1(0))}+\|b\|_{L^\infty(B_1(0))}).
$$
\end{lem}
\begin{proof}
For $k, m>0$, we set that $\overline{w}:=w+k$ and
\begin{align*}
\overline{w}_m:=\left\{
\begin{aligned}
&\overline{w}\quad & \mbox{if} \,\, w <m,\\
& m+k \quad & \mbox{if}\,\, w \geq m.
\end{aligned}
\right.
\end{align*}
Let $\eta \in C^{\infty}_0(Q_1(0) \cup B_1(0))$ be a truncation function. For $\beta>0$, we define 
$$
\psi:=\eta^2(\overline{w}_m^{\beta} \overline{w}-k^{\beta +1}).
$$
Multiplying \eqref{equ61} by $\psi$ and integrating on $Q_1(0)$, we have that
\begin{align} \label{testint}
\int_{Q_1(0)} y^{1-2s} \nabla w \cdot \nabla \psi \, dxdy + \int_{B_1(0)} a w \psi(x, 0) \, dx = \int_{B_1(0)} b \psi(x, 0)\, dx.
\end{align}
Observe that
$$
\nabla w \cdot \nabla \psi= \beta \eta^2 \overline{w}_m^{\beta-1}  \overline{w} (\nabla \overline{w}_m \cdot \nabla w) + \eta^2 \overline{w}_m^{\beta}  (\nabla \overline{w} \cdot \nabla w) + 2 \eta (\nabla \eta  \cdot \nabla w) (\overline{w}_m^{\beta} \overline{w}-k^{\beta +1}).
$$
In addition,  $\nabla \overline{w}_m=0$, $\overline{w}_m=\overline{w}$ in $\{w \geq m\}$ and $\nabla \overline{w}=\nabla w$. It then follows from \eqref{testint} that
\begin{align} \label{iteration1}
\begin{split}
& \hspace{-1cm}\int_{Q_1(0)} y^{1-2s} \eta^2 \overline{w}_m^{\beta} (\beta |\nabla \overline{w}_m|^2 + |\nabla \overline{w}|^2)\, dxdy \\
& \hspace{-1cm}=-2\int_{Q_1(0)} y^{1-2s} \eta (\nabla \eta  \cdot \nabla \overline{w})(\overline{w}_m^{\beta} \overline{w}-k^{\beta +1}) \, dxdy -\int_{B_1(0)} a w \eta^2(\overline{w}_m^{\beta} \overline{w}-k^{\beta +1})\, dx \\
& \hspace{-1cm}\quad + \int_{B_1(0)} b\eta^2(\overline{w}_m^{\beta} \overline{w}-k^{\beta +1})\, dx.
\end{split}
\end{align}
We now choose $k=\|b\|_{L^{\infty}(B_1(0))}$ if $b$ is not identically zero. Otherwise, we choose an arbitrary $k>0$ and eventually let $k\to 0^+$. This then infers that $ |b(x)|\eta^2 \overline{w}_m^{\beta} \overline{w} \leq \eta^2 \overline{w}_m^{\beta} |\overline{w}|^2$ for any $x \in \R^N$, where we used the definition of $\overline{w}$.
Note that $\overline{w}_m^{\beta} \overline{w}-k^{\beta +1}<\overline{w}_m^{\beta} \overline{w}$ and 
$$
2 \eta |(\nabla \eta  \cdot \nabla \overline{w})| \overline{w}_m^{\beta} \overline{w} \leq  \frac 12 \eta^2 \overline{w}_m^{\beta} |\nabla \overline{w}|^2 + 2 \overline{w}_m^{\beta} \overline{w}^2 |\nabla \eta|^2.
$$
Therefore, we deduce from \eqref{iteration1} that
\begin{align} \label{ineqm}
\begin{split}
\int_{Q_1(0)} y^{1-2s} \eta^2 \overline{w}_m^{\beta} (\beta |\nabla \overline{w}_m|^2 + |\nabla \overline{w}|^2)\, dxdy 
& \leq 4 \int_{Q_1(0)} y^{1-2s} \overline{w}_m^{\beta} \overline{w}^2 |\nabla \eta|^2 \, dxdy \\
&\quad + 2\int_{B_1(0)} |a| \eta^2 \overline{w}_m^{\beta} \overline{w}^2\, dx \\
& \quad + 2\int_{B_1(0)}\eta^2 \overline{w}_m^{\beta} \overline{w}^2\, dx.
\end{split}
\end{align}
Let $z:=\overline{w}_m^{\frac{\beta}{2}}\overline{w}$, there holds that
$$
|\nabla z|^2 \leq (1+ \beta) \overline{w}_m^{\beta} (\beta |\nabla \overline{w}_m|^2 + |\nabla \overline{w}|^2).
$$
Moreover, by H\"older's inequality and Lemma \ref{embedding1}, we know that
\begin{align*}
\int_{B_1(0)} |a(x)| \eta^2 \overline{w}_m^{\beta} \overline{w}^2\, dx 
&\leq  \|a\|_{L^{\frac{N}{2s}}(B_1(0))} \left(\int_{\R^N}|\eta z|^{2_s^*} \,dx\right)^{\frac{N-2s}{N}} \\
&\leq C \|a\|_{L^{\frac{N}{2s}}(B_1(0))}\int_{\R^N}y^{1-2s}|\nabla(\eta z)|^2 \,dx.
\end{align*}
Taking $\|a\|_{L^{{N}/{2s}}(B_1(0))}<\delta$ and utilizing \eqref{ineqm}, we then conclude that
\begin{align} \label{iteration}
\hspace{-1cm}\int_{Q_1(0)} y^{1-2s} |\nabla(\eta z)|^2 \, dxdy \leq C(1+\beta)  \left(\int_{Q_1(0)} y^{1-2s} z^2 |\nabla \eta|^2 \, dxdy +\int_{B_1(0)} |z \eta|^2 \, dx\right),
\end{align}
where $\delta>0$ is a small constant. Applying \cite[Lemma 2.4]{TX}, we obtain that
\begin{align*}
 C(1+\beta) \int_{B_1(0)} |z \eta|^2 \, dx \leq \frac 12 \int_{Q_1(0)} y^{1-2s} |\nabla(\eta z)|^2 \, dxdy  +  
 C(1+\beta)^{\gamma} \int_{Q_1(0)} y^{1-2s} |\eta z|^2 \, dxdy,
\end{align*}
where $\gamma>0$ is a constant depending only on $s$ and $N$. This along with \eqref{iteration} shows that
\begin{align*}
 \int_{Q_1(0)} y^{1-2s} |\nabla(\eta z)|^2 \, dxdy \leq C(1+\beta)^{\gamma}\int_{Q_1(0)} y^{1-2s} (\eta^2 +|\nabla \eta|^2) z^2 \, dxdy.
\end{align*}
It then yields from \cite[Lemma 2.2]{TX} that
\begin{align*}
\left(\int_{Q_1(0)}y^{1-2s}|\eta z|^{2\chi}\right)^{\frac{1}{\chi}} \leq C(1+\beta)^{\gamma}\int_{Q_1(0)} y^{1-2s} (\eta^2 +|\nabla \eta|^2) z^2 \, dxdy,
\end{align*}
where $\chi=(N+1)/N$ and $\chi>1$. For any $0<r<R \leq 1$, let $\eta \in C^{\infty}_0(Q_1(0) \cup B_1(0))$ be such that $\eta(x)=1$ for $x \in Q_r(0)$, $\eta(x)=0$ for $x \in Q_R(0)$ and $|\nabla \eta(x)| \leq {2}/(R-r)$ for $x \in Q_1(0) \cup B_1(0)$. By the definition of $z$ and the fact that $\overline{w}_m \leq \overline{w}$, we then derive that
\begin{align*}
\left(\int_{Q_r(0)}y^{1-2s}\overline{w}_m^{\alpha \chi}\right)^{\frac{1}{\chi}} \leq \frac{C(1+ \beta)^{\gamma}}{(R-r)^2} \int_{Q_R}y^{1-2s} \overline{w}^{\alpha} \, dxdy,
\end{align*}
where $\alpha:=2+\beta>2$. Letting $m \to \infty$, we then have that
$$
\|\overline{w}\|_{L^{\alpha \chi}(y^{1-2s}, Q_r)} \leq \left(\frac{C(1+ \beta)^{\gamma}}{(R-r)^2}\right)^{\frac {1}{\alpha}} \|\overline{w}\|_{L^{\alpha}(y^{1-2s}, Q_R)}.
$$
By the Moser iteration technique adapted to the proof of Lemma \ref{moseriteration}, we now arrive at
$$
\|\overline{w}\|_{L^{\infty}(Q_{1/2}(0))} \leq C \|\overline{w}\|_{L^2(y^{1-2s}, Q_1)}.
$$
This then gives rise to the result of the lemma, and the proof is completed.
\end{proof}

\end{document}